\documentclass[12pt, oneside, a4paper]{article}
\usepackage[utf8]{inputenc}
\usepackage[english]{babel}
\usepackage{amsfonts}
\usepackage{amsmath}
\usepackage{amsthm}
\usepackage{MnSymbol}
\usepackage{wasysym}
\usepackage{setspace}
\usepackage{hyperref}
\usepackage{mathtools}
\usepackage{multicol}
\usepackage[letterpaper,margin=2.5cm,left=3cm,right=2.5cm, includehead]{geometry}
\usepackage{graphicx}
\usepackage{float}
\usepackage{enumitem}
\usepackage[titletoc]{appendix}

\usepackage{mathrsfs}
\usepackage{pdfpages}
\usepackage{titling}
\usepackage{tocbibind}
\usepackage[spanish]{babelbib}
\usepackage{authblk}
\usepackage{tikz}
\usetikzlibrary{arrows.meta, positioning}
\usepackage{tikz-cd}
\usepackage{color,soul}

\usepackage{import}
\usepackage{xifthen}
\usepackage{pdfpages}
\usepackage{transparent}

\usepackage{eso-pic}

\newcommand{\R}{\mathbb{R}}
\newcommand{\N}{\mathbb{N}}
\newcommand{\LL}{\mathbb{L}}
\newcommand{\Z}{\mathbb{Z}}

\newcommand{\m}{\mathscr{M}}
\newcommand{\n}{\mathscr{N}}

\newcommand{\B}{\mathcal{B}}
\newcommand{\C}{\mathcal{C}}
\DeclareMathOperator{\Span}{span}
\def\bb#1\eb{\textcolor{blue}{#1}}
\def\br#1\er{\textcolor{red}{#1}}
\def\bg#1\eg{\textcolor{green}{#1}}
\definecolor{brown}{rgb}{0.55,0.27,0.07}
\def\bn#1\en{\textcolor{brown}{#1}}
\definecolor{orange}{rgb}{1,0.5,0}
\def\bo#1\eo{\textcolor{orange}{#1}}

\newcommand{\bM}{\overline{M}}

\newcommand{\bSigma}{\overline{\Sigma}}
\newcommand{\mb}{\mathbb}

\newcommand{\btau}{\bar{\tau}}

\newcommand{\bsigma}{\bar{\sigma}}

\newcommand{\inj}{\mathrm{inj}}

\newtheorem{theorem}{Theorem}[section]
\newtheorem{lemma}[theorem]{Lemma}
\newtheorem{proposition}[theorem]{Proposition}
\newtheorem{corollary}[theorem]{Corollary}
\newtheorem{definition}[theorem]{Definition}
\newtheorem{example}[theorem]{Example}
\newtheorem{convention}[theorem]{Convention}
\newtheorem{note}[theorem]{Note}

\theoremstyle{definition}
\newtheorem{remark}[theorem]{Remark}

\newtheoremstyle{named}{}{}{\itshape}{}{\bfseries}{.}{.5em}{\e #1}
\theoremstyle{named}

\date{}
\begin{document}

\title{Lorentzian Cheeger-Gromov convergence \\ and temporal functions}
\author[1]{Sa\'ul Burgos} 
\affil[1]{\textit{Departamento de Geometr\'ia y Topolog\'ia, Facultad de Ciencias \& IMAG 
Universidad de Granada, 18071 Granada, Spain}}
\author[2]{José L. Flores} 
\affil[2]{\textit{Departamento de \'Algebra, Geometr\'{i}a y Topolog\'{i}a\\ Facultad de Ciencias, Universidad de M\'alaga\\ Campus Teatinos, 29071 M\'alaga, Spain}}
\author[1]{Miguel S\'anchez} 



\maketitle

\begin{center}
	\textit{Dedicated to Alfonso Romero on the occasion of his 70th birthday.}
\end{center}

\begin{abstract} 
Uniqueness  (up to isometries) and existence of limits are studied in the context of Cheeger–Gromov convergence of spacetimes. To address the non-compactness of the linear isometry group in the semi-Riemannian setting, standard {\em pointed} convergence is strengthened to {\em anchored} convergence, which in the Lorentzian case requires the convergence of a timelike direction. This allows one to construct a local isometry between the neighborhoods of the basepoints, which can be extended globally when one of the limits is geodesically complete for a single causal type of geodesics and the other one is inextensible; a Misner-type example shows that inextensibility of both limits is not enough. 
In spacetimes, by using Cauchy temporal functions both as a strengthening of the anchors and as a tool to ``Wick rotate'' metrics, a special notion of convergence for globally hyperbolic spacetimes  (including those with timelike boundaries)  is introduced. After revisiting the tools related to time functions and studying their connections with the Sormani-Vega null distance, the machinery of Riemannian Cheeger-Gromov theory becomes applicable. In particular, several results of independent interest are obtained, including global and local characterizations of $h$-steep functions,  independence of steepness and $h$-steepness for temporal functions, compatibility of both conditions for Cauchy temporal functions, and the stability of the latter.


\bigskip

\bigskip


\end{abstract}

\newpage

\tableofcontents

\section{Introduction}	


{\bf The  problem of convergence for Lorentzian metrics.} Since the pioneering work of Noldus \cite{Noldus2, Noldus1}, several authors have tried to make sense of intrinsic convergence of metrics in the Lorentzian setting. The first natural option would be to extend classical Cheeger-Gromov convergence, introduced for the Riemannian case in Cheeger's thesis \cite{Cheeger1967,Cheeger70} and developed further in the {\em revolutionary} article (in Cheeger's words \cite{Cheeger2010}) by Gromov \cite{gromov-revolutionary}.
M. T. Anderson \cite{Anderson2004CheegerGromov} considered a notion of Cheeger-Gromov convergence for Lorentzian metrics analogous to the one for the Riemannian case, pointed out difficulties caused by the non-compactness of the Lorentz group and focused on the case of metrics on the same manifold. Indeed, such a Riemannian definition may lead to the non-uniqueness of the limit, even in simple Lorentzian cases (see Example \ref{ejemplo1}). The prominence of this problem becomes more evident after the work by Mounoud \cite[Theorem 1]{Mounoud2003}, which implies the existence of Lorentzian metrics on a compact manifold which, up to diffeomorphisms, converge uniformly to multiple non-isometric metrics (Example \ref{ejemplo2})\footnote{\label{foot:Zeghib} The authors are grateful to Prof. A. Zeghib (ENS Lyon) for providing both the reference and a detailed proof of the mentioned fact.}.

More recently, M\"uller in \cite{Muller2024gromovhausdorff} and, independently, Minguzzi and Suhr  \cite{minguzzi-suhr2023}, as well as these authors with Bykov \cite{Minguzzi2024unbounded}, studied Gromov-Hausdorff convergence in restricted classes of globally hyperbolic spacetimes obtaining existence and uniqueness (in appropriate senses) by imposing some assumptions. 

 Sormani and Vega \cite{sormanivega} introduced the notion of null distance, in order to reduce the problem to a more standard one in metric spaces. This distance is defined by using suitable time functions. Then,  the original convergence of spacetimes can be associated with the convergence of metric spaces and integral currents. However, the relation with the convergence of the original Lorentzian metrics is unclear in general. This approach has been further studied by Meco, Sakovich and Sormani \cite{SakovichSormani2024Variousnotions, meco2025uniformtemplecharts}
and Burtscher and García-Heveling~\cite{GHnulldistance}, the latter obtaining a result on global codification of causality and relating global hyperbolicity with metric completeness. See also \cite{AllenBurtscher2022Properties, Kunzinger2022Nulldistance, Allen2023Nulldistanceandconvergence, Galloway2024encoding}.

In the present article, we will face this problem from a novel viewpoint.

\subsection{Overview}


\noindent 
{\bf  The general semi-Riemannian viewpoint}.  The first goal is to show how the uniqueness of Cheeger-Gromov limits in the Riemannian case admits generalizations  to sequences $\{(M_i, g_i)\}$ in the semi-Riemannian setting. 
Namely, in the Riemannian case, $\{(M_i, g_i)\}$ and its possible limit 
are required to be {\em pointed}, that is, with a prescribed basepoint $p_i\in M_i$  which will converge up to isometries to a basepoint in the limit.  
\begin{quote}
	{\em We will redefine Cheeger-Gromov convergence for the semi-Riemannian case, imposing that both the sequence $(M_i, g_i)$ and the eventual limit are 
    {\em anchored}, that is, pointed with an orthogonal splitting at the basepoint $T_{p} M = \mathcal{B}_{p} \oplus \mathcal{B}_{p} ^\perp$  in two complementary subspaces, one of them, $\mathcal{B}_{p}$, negative definite and the other one positive definite, the sequence of anchors $\{\mathcal{B}_{p_i}\}$ being also required to converge to the anchor of the limit.}	
\end{quote}
	Notice that pointed Riemannian manifolds can be regarded as anchored with $\B_p=\{0_p\}$ and, in the Lorentzian case, anchoring becomes equivalent to choosing  a timelike direction (or a unit timelike vector) at the prescribed point. 
    
This notion is enough to prove uniqueness up to isometries in neighborhoods around the points selected by the anchor (without any global hypothesis) at least under $C^2$-convergence (see Theorem \ref{teor:uniqueness_semi-riemannian}). 
    Even though the natural global hypotheses for the existence of a full isometry might be disputable because the Hopf--Rinow properties split into non-equivalent semi-Riemannian ones, the following one extends the case of Riemannian completeness (see Remark \ref{r_inext}) by requiring geodesic completeness for a single causal type of geodesics in one of the limits and only inextensibility in the other one. A Misner-type example (Example \ref{ex:misner_double_limit}) shows that inextensibility of both limits is not enough, in agreement with the non-Hausdorff phenomena pointed out by Geroch \cite{Geroch} (Remark \ref{r_Geroch}).

	\begin{theorem}\label{coro_simplyconnected}
	    Let $\{(M_i,g_i,\mathcal{B}_{p_i})\}$ be a sequence of anchored semi-Riemannian manifolds $C^2$-converging to two limits   $(M, g, \mathcal{B}_p)$, $(\tilde M, \tilde g, \mathcal{B}_{\tilde p})$. 
                If one of the limits is geodesically complete for at least one causal type of geodesics (timelike, spacelike or lightlike; in particular, if it is geodesically complete) and the other one is $C^2$-inextensible,
                then there exists a global isometry $F: (M,g)\longrightarrow (\tilde M, \tilde g)$ such that $F_*(\mathcal{B}_{p})=\mathcal{B}_{\tilde p}$. 
	\end{theorem}

\smallskip

\noindent {\bf Back to Lorentz: temporal functions as anchors,  and Wick rotation.}
These difficulties for the uniqueness of the limits, in addition to the impossibility of transplanting general  Riemannian theorems for the existence of limits, suggest addressing the Lorentzian case by taking further advantage of  possibilities inherent to this signature.  
Specifically, globally hyperbolic spacetimes (which share key analogies with complete Riemannian manifolds) are naturally equipped with  Cauchy temporal functions $\tau$, the existence of which is always guaranteed \cite{BernalSanchez2005Splitting}. 
Then, $\tau$ itself can be regarded as an element for  stronger anchoring, to be preserved by convergence. Reversing the sign of the metric $g$ in the direction of $\nabla \tau$, one obtains a ``Wick-rotated'' Riemannian metric.\footnote{\label{foot:Wick} We adopt this name -- despite the absence of a literal rotation in the complex plane, as in the standard Wick rotation -- because it facilitates identification in our context and has also been used in the literature. For a thorough analysis on different procedures to Wick rotate metrics, all of which are equivalent in the static case, see \cite{samuel2016wickrotation, visser2017wickrotate}.}    
\begin{quote}
     {\em As we will see in \S \ref{sec:convergence_globally_hyperbolic},  Riemannian Cheeger-Gromov machinery for both existence and uniqueness of the limits will be available in the setting of  globally hyperbolic spacetimes anchored by suitable temporal functions, by using Wick-rotated metrics.}
\end{quote}  
Moreover, this setting also makes it possible to select a conformal representative of each metric and split the problem of convergence into   a conformally invariant one (the convergence of the causal structures) and the convergence of the conformal factors (i.e., convergence of the lapse  functions, a trivial issue for our purposes). 

\smallskip
\smallskip

\noindent {\bf Wick-rotated  vs null distance viewpoints.}
 The following  natural links with the framework of  the null distance emerge: 
\begin{itemize}
    \item{(a)} The null distance $\hat d_\tau$ is also constructed  by using a time function \cite{sormanivega, SakovichSormani2024Variousnotions}. 
  As discussed in Section \ref{sec:temporalfunctions}, some information on this function is ``hidden'' in $\hat d_\tau$, but  it will be explicit in our Wick-rotated metrics $g_W$. 

\item{(b)} Bernard and Suhr  introduced the notion of 
$h$-steep function \cite{bernard-suhr, Bernard2020Cauchyanduniform}, and Burtscher and García-Heveling   showed the equivalence of global hyperbolicity and completeness of the null distance $\hat d_\tau$ using $h$-steep functions for a complete $h$ (\cite{GHnulldistance},
see also Th. \ref{thm:GHnulldistance-conborde}).  This last property will turn out to be equivalent to the completeness of the Wick-rotated metric $g_W^\tau$ (see 
Theorem~\ref{intro:thm_A} below). 
\end{itemize}
This suggests that the progress on the null distance can be applied for our Wick-rotated viewpoint, which may then lead to sharper results. It is worth mentioning that such a possibility was already explored by specialists by using the (regular) cosmological time function $\tau_{\hbox{\small{cosm}}}$ (introduced by L. Andersson et al. \cite{Galloway1998cosmologicaltime} for some classes of globally hyperbolic spacetimes).
A Wick-rotated approach using the cosmological time function was already suggested by L. Andersson but it was abandoned after years of trying to deal with the singularities that arise when $\tau_{\hbox{\small{cosm}}}$ is not smooth
\cite{sormani2018oberwolfach}.  
However, the {cosmological volume function} introduced by García-Heveling \cite{Leo2024volume} is a $C^1$ temporal function and provides a canonical Wick-rotation for a certain class of spacetimes \cite{Leo2026differentiability}. These choices present a major difficulty pointed  out in Remark \ref{remark_correccion}. In any case, they are natural candidates for the application of our methods.

Let us emphasize that a key ingredient for applying Wick-rotated metrics is the existence of temporal functions yielding global splitting of the manifolds, as the Cauchy temporal ones. It is worth noting that such functions have been obtained not only for globally hyperbolic spacetimes \cite{BernalSanchez2005Splitting} but also for those with timelike boundary \cite{zepp:structure}, which makes it possible to switch to a problem in Riemannian manifolds with boundary. Moreover,  since  the existence of steep Cauchy temporal functions \cite{muller-sanchez}  allows one to find isometric embeddings in Lorentz-Minkowski spacetime,  it suggests  possible connections to Gromov-Hausdorff convergence (see Note \ref{note} below).  


\subsection{Outline by sections}

\noindent Now that our approach has been motivated,  we briefly outline our results to guide the reader through the article. 

\smallskip
\smallskip

\noindent {\bf \S \ref{sec:preliminaries}, \S \ref{sec:temporalfunctions}. Revision of time functions. $h$-steepness and null distance}. The importance of time functions in our approach and in previous ones makes a revision of them necessary. This is carried out mostly in Section \ref{sec:temporalfunctions}, after some preliminaries in Section \ref{sec:preliminaries}.
This revision allows us to extend previous notions to the case of spacetimes with timelike boundary. However, our study is of interest in its own right, independent of boundaries.

Indeed, there is a  proliferation of  types of time functions  in the literature (temporal, steep, $h$-steep, weak, locally Lipschitz or anti-Lipschitz), some of which combine local requirements of regularity with global properties.

From the local viewpoint, it is worth pointing out 
 that there are  four levels of time functions:  time,  locally anti-Lipschitz, weak temporal and temporal (\S \ref{sec:localtime_excepthsteep}, see Figure \ref{fig:local_properties_time}). The latter will be preferred here because they provide orthogonal decompositions. 
Moreover, $h$-steep functions (Definition  \ref{def:steeptimefunctions}) become equivalent to temporal ones. This justifies using them in a global context, where $h$ is a complete Riemannian metric (\S \ref{sec:time_hsteep}, see Figure~\ref{figure:temporal_hsteep_equiv}).     

From the global viewpoint, a previous issue is  the independence and compatibility of $h$-steepness   and steepness (Definition \ref{def:generalizedtimefunction}). Indeed, the discussion and counterexamples in \S \ref{sec:indep_steephsteep} (including Appendix \ref{appendix:example2.2})
show the logical independence between both notions.

Once this is clarified, the interplay between Cauchy temporal functions and $h$-steepness becomes crucial. 
 The following excerpt of  Theorem \ref{thm:GHnulldistance-conborde} is representative. 

\begin{theorem}\label{thm_intro:GHnulldistance-conborde}
	Let $(\bM, g)$ be a spacetime-with-timelike-boundary and $\tau: \bM \to \mb R$ be a function:
	\begin{enumerate}[label=(\roman*)]
		\item If $\tau$ is $h$-steep for a complete Riemannian  metric $h$,  then $\tau$ is a time function and $\hat{d}_\tau$ is complete (and (ii), (iii) apply).
		\item If $\tau$ is a time function such that $(\bM, \hat{d}_{\tau})$ is a complete metric space, then $\tau$ is a Cauchy time function. In particular, $(\bM, g)$ is globally hyperbolic
        (and (iii) applies).
		\item If $(\bM, g)$ is globally hyperbolic, then there exists a steep and $h$-steep (with respect to a complete Riemannian metric $h$) adapted (to the boundary) Cauchy  temporal function $\tau$. 
	\end{enumerate}
\end{theorem}

\smallskip

\noindent We emphasize that, in particular, item (iii) means that, in any case,
\begin{quote} 
{\em   globally hyperbolic spacetimes admit Cauchy temporal functions (and, thus, global orthogonal splittings) adapted to all the global issues: the null-distance ($h$-steepness),  Nash-type isometric embeddings (steepness) and the possible existence of  $\partial M$ ({\em adaptability to the boundary}).    
}\end{quote}
This is proven in Appendix \ref{appendix:thm_hsteep} by revisiting the procedures to obtain the orthogonal splittings developed by one of the authors and his coworkers \cite{BernalSanchez2005Splitting, muller-sanchez, zepp:structure}.

In \S \ref{sec:nulldistance}, we apply this background to show how the null distance codifies time and the conformal structure for locally anti-Lipschitz functions (Theorem \ref{thm:isometrias_distancia_nula}), complementing a previous result by Sakovich and Sormani in \cite{Sakovich2023encodescausality} on the cosmological time $\tau_{\hbox{\small{cosm}}}$.  

\smallskip
\smallskip
\noindent {\bf \S \ref{sec:g_W} Wick-rotated Riemannian geometry}. The existence of a temporal function $\tau$ in stably causal spacetimes (\cite{miguel2005revision}, see also Remark \ref{remark:weak_temporal}(b) below) has two interesting effects (see \S \ref{subsec:conformally_invariant_g_W}): \begin{itemize}
    \item (a) To select a representative $g^\tau$ of the conformal class  of $g$, so that the $g^\tau$-gradient of $\tau$ is unit (then, one can focus on  conformally invariant properties). 
 \item{(b)}  The splitting $\nabla \tau \oplus (\nabla 
\tau)^\perp$ of the tangent bundle  which makes it possible to define a Wick-rotated Riemannian metric for any conformal representative of $g$; in particular, the canonical one $g^\tau_W$ associated with the representative $g^\tau$. 
\end{itemize}
The splitting is considerably strengthened when $\tau$ is Cauchy, as it becomes  a splitting $\R\times \Sigma$ of the manifold (orthogonal for both $g$ and $g_W^\tau$). As noticed in \S \ref{subsec:4_2}, $\tau$ is automatically $h$-steep for $h=g^\tau_W/2$ (its $g^\tau_W$-gradient being unit) and the latter is not necessarily complete but (Theorem \ref{intro:thm_A}):
\begin{quote}
     $g_W^\tau$ is complete if and only if $\tau$ is $h$-steep for a complete Riemannian metric (and, thus, the spacetime is globally hyperbolic). 
\end{quote}
Apart from yielding a new characterization of global hyperbolicity (Corollary \ref{coro:wick-rotated_globally hyperbolic}), this  will be the key for our progress on convergence later. Indeed, we will apply convergence for {\em complete} Riemannian manifolds in order to study convergence for globally  hyperbolic manifolds, as the latter admit complete $h$-steep functions. Moreover, our results also suggest the extension to the case with boundary of the consequences on integral currents in the study by Allen and Burtscher \cite{AllenBurtscher2022Properties}, see Remark \ref{r_correccion}.


Two further relevant issues are also studied. The first one (\S \ref{subsec:4_3}) is the fact that, even though the null distance  $\hat d$ codifies causality in this setting, the Wick-rotated distance $d_W^\tau$ does not (see the counterexample in Appendix \ref{app:no_codifica}). 
We emphasize that a small modification in the construction of  $d^\tau_W$ preserves it (Proposition \ref{prop:d_W_encoding}), suggesting that our approach preserves the full information about $\tau$. 
The second issue concerns the proof of stability of $h$-steep functions   (including a redefinition for a subtlety in the adaptability to the boundary), \S \ref{subsec:4_4},  which completes our study of time functions from the global viewpoint (Theorem \ref{thm:stability}).

\smallskip
\smallskip
\noindent {\bf \S  \ref{sec:convergence_semi}, \ref{sec:convergence_globally_hyperbolic} Cheeger-Gromov convergence.} Section \ref{sec:convergence_semi} has a pivotal importance for going from Riemannian to semi-Riemannian Cheeger-Gromov convergence. First, a brief summary of the Riemannian setting and basic existence and uniqueness results are provided, \S \ref{subsec:5_1}. The semi-Riemannian extension in \S \ref{subsec:convergence_semi-riemann} starts  with a discussion of the difficulties (\S \ref{subsec:5_2_1}),   including two illustrative examples (Examples \ref{ejemplo1}, \ref{ejemplo2}). 

These difficulties include first the non-compactness of the isometry group. This is handled satisfactorily by the anchor at the basepoint, which motivates the definition of convergence and yields the first properties in \S \ref{subsec:5_2_2}. 

A second difficulty arises because, given a Cheeger-Gromov sequence and a candidate limit, there is no canonical norm to compare the elements of the sequence with the limit (that is, no quasi-isometries are defined).  
However, assuming that the convergence is at least $C^2$ (see footnote \ref{foot:Ambrose}), the anchors provide natural auxiliary Riemannian metrics (Definition \ref{def:anchor}), which are sufficient for the estimates on convergence 
(see the  lemmas in $\S$ \ref{subsec:5_2_3}).

This leads to Theorem \ref{teor:uniqueness_semi-riemannian}, a fundamental result providing a local isometry around the basepoints of any two possible limits. Going from this semi-local result to a global one, a third difficulty appears, as no analogue of the Hopf--Rinow theorem holds. In any case, the global result in Theorem \ref{coro_simplyconnected} holds under geodesic completeness for a single causal type of geodesics in one of the limits, and an example shows that inextensibility of both limits is not enough (proof and example in \S \ref{subsec:5_2_3}). 

In Section \ref{sec:convergence_globally_hyperbolic}, we develop the announced convergence for globally hyperbolic spacetimes by using a stronger anchoring, Cauchy temporal functions.
This is done in two steps:
\begin{itemize}
    \item 
Prescribe a  (unit) timelike vector field $T_i$ so that the corresponding Wick-rotated metric is complete ({\em Wick-complete anchored spacetime}), and assume convergence for the pullbacks of $T_i$'s too  (this is required only in $C^0$, see  Theorem~\ref{prop:uniqueness_observed_spacetimes}), see \S \ref{sec:convergence_observed_spacetimes}.
\item Prescribe   $h$-steep functions $\tau_i$ adapted to the boundary  ({\em h-steep anchored spacetimes}) and $C^0$-convergence, and choose the conformal representative of the metric associated with $\tau_i$, see $\S \ref{subsec:6_2}$. 
On the one hand, this will imply the convergence of the vector fields $\nabla^i \tau_i$ and, on the other,  $h$-steepness will yield the completeness of the Wick-rotated metrics, then making  the previous case applicable. Notice that the convergence is split into the conformal structure and the lapse. 
\end{itemize}
In both cases, the convergence of the Wick-rotated metrics 
is enough for the 
uniqueness of limits (Theorem~\ref{prop:uniqueness_observed_spacetimes}, Remark \ref{remark:convergence_conformal}), %
and one has (see the proof in \S \ref{subsec:6_2}):

\begin{theorem}[Cheeger-Gromov convergence of $h$-steep anchored structures]\label{thm:convergence_uptodiffeo}
$\hbox{}$

  \noindent  Let $\{(\bM_i, g_i, \tau_i, p_i)\}_{i \in \N}$ be a sequence of h-steep anchored smooth spacetimes. Assume that
	\begin{enumerate}[label=(\roman*)]
		\item the sequence of pointed complete Riemannian manifolds $(\bM_i, g^{\tau_i}_W, p_i)$ converges in $C^k$, $k\geq 0$, to a pointed complete Riemannian manifold $(\bM,h,p)$ in the sense of Definition \ref{def:riemann_convergence} with embeddings $\phi_i : U_i \subset \bM \to \bM_i$, and
		\item the functions $\tau_i \circ \phi_i$ converge in $C^{k+1}$ to a smooth function $\tau$ on $\bM$.
	\end{enumerate}
	Then, the spacetimes $(\bM_i, g_i^{\tau_i}, \tau_i, p_i)$ converge in $C^k$ as Wick-complete anchored spacetimes to the (necessarily globally hyperbolic) spacetime $(\bM,g,\tau,p)$ (in the sense of Definition \ref{def:lorentz_convergence}), where $g = h - 2d\tau^2$ and  $\tau$ is an adapted $\frac{1}{2} h$-steep temporal function for $g$. Moreover, the limit is unique up to isometries.
\end{theorem}

Then, as emphasized in $\S \ref{subsec:6_3}$,     existence of Lorentzian limits is achieved by using the Riemannian criteria for Wick-rotated metrics (Corollary \ref{thm:compacidad_globalmente_hiperbolicos}).

\smallskip
\smallskip
Finally, we show  that the  simple structure of a common orthogonal splitting for all the metrics can be applied within a general framework. Indeed, the splitting associated with a Cauchy temporal function $\tau$ makes it possible to regard all the $h$-
steep anchored spacetimes with diffeomorphic Cauchy hypersurfaces as spacetimes on the same orthogonal product $\R\times \Sigma$, where $\tau$ is identifiable with the natural projection on $\R$ (Proposition \ref{lemma:difeomorfismo_RxSigma}). Indeed, once $\tau$ is chosen, for each diffeomorphism $\tau^{-1}(0)\to \Sigma$ one obtains an isometry $M\to \R\times \Sigma$.  
From a practical viewpoint, this serves to obtain sufficient conditions of convergence, as discussed around
Example~\ref{example:de-Sitter}.

\begin{note}[Gromov-Hausdorff convergence]  \label{note}  
{\em As a final prospective observation, here we have restricted ourselves to Cheeger-Gromov convergence; however, we expect that the underlying ideas can be used for Gromov-Hausdorff convergence too. Indeed, parts of the present article might serve as a preliminary study  for this, especially the study of steep functions. 

In the classical setting, the Gromov-Hausdorff distance $d_{GH}(X,Y)$ measures how far two compact Riemannian manifolds  (or, in general, two metric spaces) $X$ and $Y$ are from being isometric. It is defined by taking all the possible isometric embeddings $F: X\rightarrow Z$, $G: Y\rightarrow Z$ in a compact metric space $Z$, and letting $d_{GH}(X,Y)$ be the infimum of $d_Z(F(X),G(Y))$, the  Hausdorff  distance in $Z$ between the images $F(X)$ and $G(Y)$. Notably, Nash's theorem makes it possible to find isometric embeddings to compare any two $C^3$-Riemannian manifolds in some $\R^N$. Steep and $h$-steep temporal functions might be convenient for the development of these ideas in the Lorentzian setting.

}\end{note}

\section{Preliminaries on spacetimes, causality and timelike boundaries
}\label{sec:preliminaries}

We begin with a basic review of some notions from Lorentzian geometry that we will use throughout the paper. For further background and consistent conventions, we refer the reader to \cite{Oneill,MinguzziSanchez}. 
Let $M$ be a smooth manifold of dimension $n+1$. A Lorentzian metric $g$ on $M$ has signature $(-,+,...,+)$, and classifies a vector $v \in TM$ as \textit{timelike, null} or \textit{spacelike} according to whether $g(v,v)$ is negative, zero, or positive, respectively. Lightlike vectors are defined as nonzero null vectors. Timelike and lightlike vectors are collectively referred to as \textit{causal}. A \textit{time-orientation} is a continuous choice of a future cone at each point of $M$. A \emph{spacetime} $(M,g)$ is a time-oriented Lorentzian manifold in which every  causal vector can be classified as either future-pointing or past-pointing. This notion naturally extends to curves: we write $p\ll q$ (resp. $p<q$) if there exists a future-pointing timelike (respectively, causal) curve from $p$ to $q$. Then, the {\em chronological future} of a point $p$ is defined as $I^+(p):=\{q\in M: p\ll q\}$, the {\em causal future} as $J^+(p):=\{q\in M: p\leq q\}$, where $p\leq q$ means either $p<q$ or $p=q$, and the {\em horismos} of $p$ as $E^+(p):=J^+(p)\setminus I^+(p)$. Given two points $p, q\in M$, we define the {\em causal diamond} $J (p,q):=J^{+} (p) \cap J^{-} (q)$ and the {\em chronological diamond} $I(p,q) := I^{+} (p) \cap I^{-} (q)$.
We also write $(p,q) \in I^{+}$ if $p \ll q$ (resp. $(p,q) \in J^{+}$ if $p \leq q$).
Consistently, for any subset $S\subset M$, we define $I^+(S):=\cup_{p\in S} I^+(p)$. If $U\subset M$ is an open subset regarded as a spacetime in its own right, then $I^+(S,U)$ denotes the chronological future of $S\subset U$, and similar conventions apply for the causal and horismos relations. If $g,g'$ are Lorentzian metrics on $M$, we write $g<g'$ to indicate that the causal cones of $g$ are contained within the timelike cones of $g'$.  
 
We will consider {\em smooth} regularity to mean that the different elements we are dealing with are as regular as required by the context. Typically we will consider $C^r$ convergence for the metrics, which implicitly requires $C^{r+1}$ for the manifold and for temporal functions when regarded as anchors. Except when otherwise specified, $C^{1,1}$ is sufficient and, consistently, lower regularity is allowed for the Levi-Civita connection and for any causal curve\footnote{In general $H^1$ is the natural regularity for the convergence of curves, and it is equivalent to locally Lipschitz when restricted to the class of causal curves, see \cite[Sections 2.3 and 2.4]{zepp:structure}.} 
  $\gamma$; therefore, when necessary, the causal character of its velocity $\gamma'$ is understood almost everywhere (a.e.). Notice that if $g$ is $C^k$,  then $C^k$ causal curves suffice to define chronological and causal futures and pasts, however, in the case of spacetimes with timelike boundaries considered below, locally Lipschitz causal curves must be used to compute $J^\pm (p)$, even if the spacetime with boundary is $C^\infty$; see \cite[Appendix]{zepp:structure}.

We will follow the standard conventions of the well-known causal ladder of spacetimes, as presented in \cite{MinguzziSanchez}. 
In particular, a spacetime is said to be \textit{causal} if it contains no 
closed causal curves, and {\em stably causal} if it admits a time function -- a condition that admits a nontrivial characterization, which in turn justifies the terminology. More precisely, our interest will focus on the following classes of functions.

\begin{definition}\label{def:generalizedtimefunction}
A {\em generalized time function} is a (not necessarily continuous) function $\tau : M \to \mathbb{R}$ that is strictly increasing along all 
future-directed causal curves. 

When, additionally,
\begin{enumerate}[label=(\roman*)]
\item $\tau$ is continuous, then $\tau$ is a {\em time function}.

\item {$\tau$ is smooth (at least $C^1$), then $\tau$ is a {\em smooth time function}.}

\item  $\tau$ is smooth and its gradient  $\nabla \tau$ is timelike (and so necessarily past-directed),  then $\tau$ is a 
{\em temporal} function.  

\item $\tau$ is temporal and $g (\nabla \tau, \nabla \tau) \leq -c$ for some $c>0$,\footnote{\label{foot:steepglobal}In the definition introduced in \cite[Thm 1.1]{muller-sanchez}, $c$ is chosen to be $1$ (otherwise, multiplying by a constant would ensure this). Here, it is convenient to allow $c>0$  to stress the global character of this notion as well as its independence from $h$-steepness below.} then $\tau$ is {\em steep temporal}.
\end{enumerate}
\end{definition}

We will be interested in spacetimes $(M,g)$ with a good causal behaviour. The best one occurs when it is {\em globally hyperbolic}, that is,  $(M,g)$ is causal and 
the causal diamonds $J(p,q)$ are compact for all $p,q\in M$. 
A celebrated result, originally obtained by Geroch at the topological level and later improved by Bernal and one of the authors at the smooth and orthogonal levels,  states that any globally hyperbolic spacetime $(M,g)$ admits a \textit{Cauchy orthogonal splitting}, that is, $(M,g)$ is isometric to $(\mathbb{R} \times \Sigma, -\Lambda d\tau ^2 + \sigma_{\tau})$, where $\tau$ is a {\em Cauchy temporal function} (i.e., a temporal function as above, whose  slices $\tau=$ constant are  Cauchy hypersurfaces),  $\Lambda$ is a smooth positive function and $\sigma_{\tau}$ is a Riemannian metric on each slice $\{ \tau \} \times \Sigma$ varying smoothly with $\tau$, \cite{BernalSanchez2005Splitting}. Recall that a {\em Cauchy hypersurface} is a subset $S\subset M$ that is crossed exactly once by any inextensible timelike curve; in general, these conditions imply that $S$ is  an achronal  topological hypersurface \cite[Lemma 14.29]{Oneill}; in any case, in the splitting above, the slices of $\tau$ are smooth, acausal and spacelike. 

A different issue, studied by M\"uller and one of the authors in \cite{muller-sanchez},
is  the Nash-type isometric embeddability of a   spacetime in Lorentz-Minkowski spacetime  $\mathbb{L}^N$,  for some sufficiently large dimension $N$. This is possible (for $C^3$ metrics) if and only if the spacetime  admits a steep temporal function. A globally hyperbolic spacetime not only satisfies this property but also admits a Cauchy steep temporal function, which can be used for the splitting above.

The case of spacetimes $(\bM,g)$ with timelike boundary was first studied in \cite{tesisdidier, galloway2014notes}. In the systematic study carried out by Ak\'e et al.  \cite{zepp:structure}, it is shown that all the causality  notions can be directly extended to them just by taking into account the regularity issue for causal curves pointed out above.  
Remarkably, both  Bernal-S\'anchez splitting and M\"uller-S\'anchez embeddability above still hold, even though new techniques were required to obtain Cauchy temporal functions $\tau$ adapted to the boundary. 
Next, we introduce these spacetimes following
 \cite{zepp:structure}. 

\begin{definition}\label{def:spacetimeTLB}
	A \emph{Lorentzian manifold with timelike boundary} $(\bM,g)$, $\bM = M \cup \partial M$, is a smooth  manifold with boundary $\bM$ endowed with a Lorentzian metric $g$ (which is also smooth on the boundary $\partial M$) such that the pull-back metric $i^*g$ (induced by the natural inclusion $i:\partial M \hookrightarrow \bM$) defines a Lorentzian metric on the boundary. A \emph{spacetime-with-timelike-boundary} is a time-oriented Lorentzian manifold with timelike boundary. If, in addition, 
	it is causal and all  $J(p,q),\; p,q\in \bM$, are compact, then it is called a {\em globally hyperbolic spacetime-with-timelike-boundary}.  
\end{definition}
Note that the interior $M$ of a globally hyperbolic spacetime-with-timelike-boundary  is intrinsically globally hyperbolic (with the induced metric) if and only if $\partial M$ is  empty; indeed, the points of $\partial M$ can be identified as the naked singularities \cite[Appendix A]{zepp:structure}. The following theorem (\cite[Theorem 1.1]{zepp:structure}, see also \cite[\S 5.4.2]{sanchez:slicings} and \cite{muller-sanchez}) summarizes the   
aforementioned results  for the case with boundary as well as other properties.

\begin{theorem}\label{theorem:splitting_boundary}
	Any globally hyperbolic spacetime-with-timelike-boundary $(\bM, g)$ admits a Cauchy temporal function $\tau$ {\em adapted to the boundary}, that is, so that on the boundary  $\nabla \tau$ remains tangent to $\partial M$. As a consequence, $\bM$ splits smoothly as a product $\mathbb{R} \times \bar{\Sigma}$, where $\bar{\Sigma}$ is an $n$-manifold with boundary, the metric can be written (with natural abuses of notation) as a parametrized orthogonal product
	\begin{equation}\label{eq:splitting_boundary}
		g = -\Lambda d\tau ^2 + \sigma_{\tau},
	\end{equation}
	where $\Lambda: \mathbb{R}\times \bar{\Sigma} \to (0,\infty)$ is a {smooth} positive function, $\sigma_\tau$ is a Riemannian metric on each slice $\{\tau\} \times \bar{\Sigma}$ varying smoothly with $\tau$, and these slices are spacelike Cauchy hypersurfaces-with-boundary. 
	Moreover, the interior $M$ of $\bM$ is always causally continuous, the boundary $\partial M$ is a (possibly non-connected) globally hyperbolic {spacetime (without boundary)} and:
	\begin{enumerate}[label=(\alph*)]
		\item the restriction of $\tau$ to $M$ extends the known orthogonal splitting of globally hyperbolic spacetimes to this class of causally continuous spacetimes without boundary,
		\item the restriction of $\tau$ to $\partial M$ provides a Cauchy temporal function for the boundary whose levels are acausal in $\bM$.
	\end{enumerate}
	
	Finally, $\bM$ admits a steep temporal function (which can be chosen Cauchy and adapted to the boundary and satisfying $g(\nabla \tau, \nabla \tau) \leq -c$ for any prescribed $c > 0$ so that the splitting \eqref{eq:splitting_boundary} holds with $\Lambda \leq 1/c$) and, thus, when the metric is $C^3$, it can be isometrically embedded in Lorentz-Minkowski $\mathbb{L}^N$ for some $N\in \mathbb{N}$. 
    
    Furthermore, when $(\bM,g)$ has a $C^3$ metric, the following properties hold: 
	\begin{itemize}
		\item[(i)] it is globally hyperbolic if and only if each metric in its conformal class is isometrically embeddable in some $\LL^N$, and 
	
	\item[(ii)] it is stably causal if and only if one representative of its conformal class is isometrically embeddable in some $\LL^N$.
	
	\end{itemize}
\end{theorem}

%

\section{Null distance: h-steepness and timelike boundaries} \label{sec:temporalfunctions}

\subsection{Sormani-Vega null distance and time functions}\label{sec:localtime_excepthsteep}

The null distance, introduced by Sormani and Vega \cite{sormanivega} in the context of spacetimes, is -- under certain conditions -- a metric  distance capable of capturing both the topological and causal structure of the manifold. Their construction extends naturally to spacetimes with timelike boundary $(\bM,g)$. We now recall the key ingredients of this construction in that setting. 

\begin{definition}\label{def:piecewise-causal-curve}
	A \emph{piecewise causal curve} $\beta : [a,b] \to \bM$ is a curve admitting a partition $a = s_0 < s_1 < \cdots < s_k = b$, such that each restriction $\beta_i = \beta |_{[s_{i-1},s_i]}$ is either a future or past directed causal curve.
\end{definition}

\noindent Notice that, here, piecewise causal curves are allowed to move forward and backward in time. Then, the null distance is defined as follows.

\begin{definition}\label{def:nulldistance}
	Let $\tau$ be a generalized time function on $(\bM,g)$. The \emph{null length (associated with $\tau$)} of a piecewise causal curve $\beta: [a,b] \to \bM$ is defined by
	\begin{equation*}\label{eq:def:null-length}
		\hat{L}_{\tau} (\beta) := \sum_{i=1}^{k} |\tau (x_i) - \tau(x_{i-1})|,
	\end{equation*}
	where $x_i = \beta (s_i)$ are the break points of $\beta$ and $a = s_0 < s_1 < \cdots < s_k = b$. 
	
	The \emph{null distance} between $p,q \in \bM$ is defined by
	\begin{equation*}\label{eq:def:nulldistance}
		\hat{d}_{\tau} (p,q) := \inf \{\hat{L}_{\tau} (\beta): \beta \text{ is a piecewise causal curve from } p \text{ to } q\}.
	\end{equation*}
\end{definition}

It is clear from the definition that $\hat{d}_\tau$ is symmetric and satisfies the triangle inequality, but positive definiteness does not hold in general. If $\tau$ is \emph{locally anti-Lipschitz} and continuous then {$(\bM, \hat{d}_\tau)$} is a conformally invariant length-metric space which induces the manifold topology (see \cite[Thm 4.6]{sormanivega} and \cite[Thm 1.1]{Burtscher2015Length}). In their study of the null distance, Burtscher and García-Heveling \cite{GHnulldistance} introduced the slightly stronger notion of \textit{weak temporal function}.

\begin{definition}\label{def:weak_temporal}
	Let $h$ be any auxiliary Riemannian metric on $\overline{M}$, and $d_h$ its associated distance. A continuous function $\tau$ is \emph{weak temporal} if every $x$ admits a neighborhood $U$ and $C \geq 1$ such that
	\begin{equation}\label{eq:weak_temporal}
		(p,q) \in J^+ \cap (U \times U)\; \Rightarrow \; \frac{1}{C} d_h (p,q) \leq \tau (q) - \tau (p) \leq C d_h (p,q). 
	\end{equation}
	If $\tau$ is required to satisfy only the first $\leq$ in (\ref{eq:weak_temporal}) we say that it is  \emph{locally anti-Lipschitz} (and $\tau$ becomes a time function).
\end{definition}
    Notice that time functions are locally Lipschitz if and only if they satisfy the last inequality in \eqref{eq:weak_temporal} for some constant $C>0$, \cite[Prop. 2.11]{GHnulldistance}.
    
\begin{remark}
\label{remark:lipschitz_uptorescaling} 
In the definition of Cauchy time function, $\tau$ is assumed to be surjective on $\R$  by convenience, but no relevant difference occurs by dropping it and admitting a proper interval  $I \subset \R$ as the image of $\tau$. {Indeed, if $\phi: \R \rightarrow I$ is an appropriate increasing diffeomorphism,} then $\phi\circ \tau$ and $\phi^{-1} \circ \tau$  would switch from one case to the other.  For any time function,  the  new time functions constructed as above will be called {\em diffeo-rescalings} of the original one. Next, more general rescalings will be analyzed.
\end{remark}

\begin{remark}[\em Local Lipschitzness up to rescaling]\label{remark:rescaling_counterexample} 
One might expect that a time function $\tau$ becomes locally Lipschitz after composing with a suitable increasing homeomorphism $\phi$ of $\R$ (a {\em rescaling}), at least around each point. This is not the case in general, even for locally anti-Lipschitz Cauchy time functions with real-analytic, uniformly spacelike level sets. Indeed, on $\LL^2$ with coordinates $(t,x)$ and $k\geq 3$, consider the function $\tau_k$ whose level set $\{\tau_k=a\}$ is the graph
$$
t=t_a(x):=a-\frac{1}{2\pi k}\Big[\sin\big(2\pi(ka-x)\big)+\sin(2\pi x)\Big], \qquad a\in\R
$$
(these graphs are $\frac{2}{k}$-Lipschitz, pairwise disjoint and exhaust $\LL^2$, as $\partial_a t_a=1-\cos(2\pi(ka-x))\geq 0$). Then $\tau_k$ is a locally anti-Lipschitz Cauchy time function, real-analytic and temporal outside the spacelike curves $\Gamma_m:=\{t=\frac{1}{k}(x+m-\frac{\sin 2\pi x}{2\pi})\}$, $m\in\Z$, where $\partial_a t_a$ vanishes; but, for any $p_0\in \cup_m\Gamma_m$ and any neighborhood $W$ of $p_0$, every $\phi$ such that $\phi\circ\tau_k$ is Lipschitz on $W$ is constant around $\tau_k(p_0)$. The obstruction is that the level sets ``stall'' at the points of $\Gamma_m$: on a neighborhood of $p_0$, the infimum over the chains $a=a_0<a_1<\dots<a_N=b$ of the sum of the distances between consecutive level sets vanishes for $a<b$, so no rescaling can separate them. Conversely, when this infimum is positive it provides the optimal rescaling; a simple sufficient condition, satisfied by weak temporal functions, is that the level sets be {\em uniformly spaced} around each point, that is, $\max(t_b-t_a)\leq K\min (t_b-t_a)$ for all $a<b$, in coordinates $(t,x)$ with $\partial_t$ timelike where the level sets are Lipschitz graphs $t=t_a(x)$ (equivalently, the maps between level sets along the coordinate lines of $t$ are $K$-bi-Lipschitz). 
\end{remark}

\begin{remark}\label{remark:weak_temporal}
	Some notes are in place for weak temporal functions (see \cite[\S 2.5]{GHnulldistance}):
	\begin{enumerate}[label=(\alph*)]
		\item If $\tau$ is a weak temporal function then it is locally Lipschitz, and by Rademacher's Theorem its gradient $\nabla \tau$ exists  almost everywhere. Moreover, by the locally anti-Lipschitz property, the gradient $\nabla \tau$ is past-directed timelike when it exists.\footnote{Burtscher and García-Heveling also proved that regular cosmological time functions \cite{Galloway1998cosmologicaltime} are weak temporal functions. These functions were also studied in the context of the null distance by Sormani and Vega \cite{sormanivega}, and Sakovich and Sormani \cite{Sakovich2023encodescausality, SakovichSormani2024Variousnotions}.}
		
		\item It is known that weak temporal functions are locally stable as time functions, that is, $\tau$ is also a time function for metrics close to the original one $g$ (or just with the cones a bit more opened) around each point (and then globally), see \cite{CMG}. 
Notice that the existence of these functions, indeed, even (more restrictive)  temporal functions, in stably causal spacetimes,  closes the circle of equivalences among all the classical alternative definitions of stable causality 
\cite[Th. 4.15, Fig. 2]{miguel2005revision}.

Analogously,  from a global viewpoint, Cauchy temporal functions are stable in the set of all the Lorentzian metrics and, thus, the existence of such functions \cite{BernalSanchez2005Splitting, zepp:structure} implies the stability of global hyperbolicity (see also 
\cite[5.4.3]{sanchez:slicings} and \cite[\S 4]{zepp:structure}). 
	\end{enumerate}
    Summing up, the logical relations about the local conditions on time functions are illustrated in Figure \ref{fig:local_properties_time}. 
			\begin{figure}[h] 
		\centering
		\[\begin{array}{c}
			\text{Temporal}\; \Leftrightarrow \;\text{Locally anti-Lipschitz}\; +\; C^1 \text{ time} \\ [2pt]
			\Downarrow \\ [2pt]
			\text{Weak temporal} \\ [2pt]
			\Downarrow \\ [2pt]
			\text{Locally anti-Lipschitz time} \\ [2pt]
			\Downarrow \\ [2pt]
			\text{Time function} \\ [2pt]
		\end{array} 
		\]
		\caption{Summary of the logical relations among the {\em local} conditions on time functions in the literature. 
        Steepness must be regarded as a global property (footnote \ref{foot:steepglobal}). } \label{fig:local_properties_time} 
	\end{figure}
\end{remark}

The 
fundamental result by Burtscher and Garc\'ia-Heveling  
which characterizes global hyperbolicity by using the null distance, is the following.

\begin{theorem}\cite[Theorems 1.4 and 1.9]{GHnulldistance}\label{thm:GHnulldistance}
	A spacetime $(M,g)$ is globally hyperbolic if and only if there exists a time function $\tau$ such that $(M, \hat{d}_\tau)$ is a complete metric space. In this case, any locally anti-Lipschitz Cauchy time function\footnote{Not necessarily onto, see Remark \ref{remark:lipschitz_uptorescaling}.} $\tau^*$ encodes causality, that is, 
	\begin{equation*}
		p \leq q \quad \Leftrightarrow \quad \hat{d}_{\tau^*} (p,q) = \tau^* (q) - \tau^* (p), \;\;\,\forall\; p,q \in M.
	\end{equation*}
\end{theorem}
It is worth pointing out that, in the first part of this theorem, $\tau$ must be chosen carefully, which leads to the notion of 
$h$-steep function.    
Next, we will revisit the role of these functions as well as their independence from steep ones,  checking also that the theorem can be extended to the case of timelike boundaries.


\subsection{The $h$-steep  functions}\label{sec:time_hsteep}

In this section, we revisit the notion of $h$-steep function and demonstrate both the logical independence and the compatibility between steep temporal functions and $h$-steep temporal functions.
This notion is based on an auxiliary Riemannian metric $h$ and it has both a local content, related to regularity, and a global one, related to the global conformal structure of the spacetime. The former is independent of the completeness of $h$, and it is considered here first just to complete the revision of the previous concepts by showing that it is equivalent to ``temporal''. For the latter, however, the completeness of $h$ is essential and it will be imposed later. In the next sections, we will only use these functions from a global viewpoint and, thus, the completeness of $h$ will be assumed.

\subsubsection{Characterization of $h$-steepness ($h$ not necessarily complete)}

\begin{definition}\label{def:steeptimefunctions}
	Let $(\bM, g)$ be a spacetime-with-timelike-boundary and $h$  a Riemannian metric on $\bM$. A $C^1$ function $\tau : \bM \to \mathbb{R}$ is
		 \emph{$h$-steep}\footnote{This notation was originally introduced in \cite{bernard-suhr} and then changed in \cite{Bernard2020Cauchyanduniform} to \textit{completely uniform temporal functions}. We prefer to stick to the $h$-steep notation as it refers to the Riemannian metric that is related to our temporal function. This is consistent with \cite{minguzzi:h-steep} and previous versions of \cite{GHnulldistance}.} if, for all future causal vectors $v \in T\bM$,
		\begin{equation} \label{eq:h-steep} 
			d\tau (v) \geq \| v \|_h . 
		\end{equation}
		If $\tau$ is weak temporal and (\ref{eq:h-steep}) holds almost everywhere, we say that $\tau$ is \emph{weak $h$-steep}.
\end{definition}

\begin{proposition}\label{p_hsteep_is_time} Let $(\bM, g)$ be a spacetime-with-timelike-boundary.
	
\begin{itemize}
	\item [(1)]  Any temporal function $\tau$ on $\bM$ is $h$-steep for a suitable Riemannian metric $h$.
	\item [(2)]  Any $h$-steep function $\tau$ on $\bM$  is a ($C^1$) temporal function. 
\end{itemize}	
	 In particular, $(\bM, g)$ is stably causal if and only if it admits an $h$-steep function.
	 
	 \end{proposition}

\begin{proof} (1) If $(\bM, g)$ admits a temporal function $\tau$,  its gradient splits orthogonally the tangent bundle and, then, the expression of the metric in $TM$ splits as $g = -\Lambda d\tau^2 + \sigma_\tau$ (see Section \ref{subsec:conformally_invariant_g_W}). 
    For any $v$ future causal we have
\begin{equation*}
	d \tau (v)^2 \geq \bsigma_\tau (v,v),
\end{equation*}
where $\bsigma_\tau := \sigma_\tau / \Lambda$, which implies 
\begin{equation*}
	d\tau (v)^2  \geq \frac{1}{2} \left( d\tau (v)^2 + \bsigma_\tau (v,v) \right) =: h(v,v).
\end{equation*}
	
(2) For any future-directed causal curve $\gamma: [0,1]\rightarrow \bM$ and $t<t'$ in its domain, putting $p=\gamma(t)$, $q=\gamma(t')$:
	\begin{equation}\label{e_hsteep_is_time}
		\tau (q) - \tau (p) = \int_{t}^{t'} d\tau (\dot{\gamma}) (s) ds 
		\geq \int_{t}^{ t'} \|\dot{\gamma} (s)\|_{h} ds 
		= L_h (\gamma) \geq d_{h} (p,q).
	\end{equation}
	Since $d\tau(\dot\gamma)\geq \|\dot\gamma\|_h>0$, $\tau$ is strictly increasing along future-directed causal curves; in particular, there are no closed causal curves, so $p\neq q$, $d_h(p,q)>0$ and $\tau$ is a smooth time function.
	To see that the gradient $\nabla \tau$ is past-directed timelike, just notice that for any future directed causal vector $v$ we have
	\begin{equation*}
		g (\nabla \tau, v) = d\tau (v) \geq \| v \|_h > 0.
	\end{equation*}
	
	For the last assertion, the necessary condition follows from (1) (taking into account the existence of temporal functions in stably causal spacetimes \cite{splitting, miguel2005revision}), while the sufficient condition follows from (2).
	\end{proof}

\begin{remark} The diagram in Figure \ref{figure:temporal_hsteep_equiv} completes the logical relations between the different types of functions in Figure \ref{fig:local_properties_time}. In particular,
the condition of being $h$-steep with respect to some Riemannian metric $h$ (not necessarily complete) emerges as a local one equivalent to temporal. In what follows, we will use it only with respect to a {\em complete} Riemannian metric $h$.
\end{remark}
\begin{figure}[h]
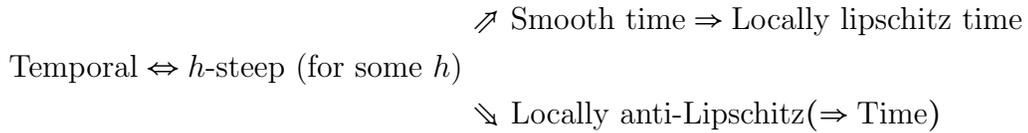
 
	\begin{align*}
		&\Nearrow \text{Smooth time} \Rightarrow \text{Locally Lipschitz time} \\[0.5pt]
		\text{Temporal} \Leftrightarrow h\text{-steep (for some } h\text{)} & \\[0.5pt]
		&\Searrow \text{Locally anti-Lipschitz} (\Rightarrow \text{Time})
	\end{align*}
	\caption{Summary of the local properties of temporal and $h$-steep functions. Temporal functions become equivalent to $h$-steep ones for some (possibly incomplete, non-prescribed) $h$. Thus, $h$-steepness will be used globally (for complete $h$).} \label{figure:temporal_hsteep_equiv}
\end{figure}

\subsubsection{Independence of steepness and $h$-steepness}\label{sec:indep_steephsteep}
Apart from the implications above, the concepts of being steep temporal and $h$-steep (with respect to a complete $h$) are logically independent. Indeed, as discussed by Burtscher and Garc\'ia-Heveling \cite[Example 4.5]{GHnulldistance}, the temporal function $\tau = t$ in the globally hyperbolic spacetime given by one of the authors in \cite[Sec 6.4]{sanchez:slicings} is steep (in fact, its gradient is unit), but it is not $h$-steep with respect to any {\em complete} Riemannian metric, since its null distance $\hat{d}_t$ is incomplete. 
We complement this example with the following one. 

\begin{example}\label{ex:h-steep_no-steep}(An $h$-steep temporal function with $h$ complete that is not steep.)
	{\em 
		Consider the $2$-dimensional Minkowski spacetime $\LL^2 := (\mathbb{R}^{2},g = -dt^2 + dx^2)$. Starting at the steep temporal function $t$,  define a new 
		function $\tau= \psi \circ t$, where $\psi: \R\rightarrow \R$ is an increasing diffeomorphism so that $\psi'=1$ everywhere but in a sequence of small    intervals  
		$I_k \subset  \R$ with disjoint closures, each $I_k$ centered at a point $t_k$ with 
		$\{ t_k \}_k$ diverging and 
				$\lim_{k\rightarrow \infty} \psi'(t_k)=0$.
		Clearly, $\tau$ will not be steep but it will be $h$-steep with respect to the complete Riemannian metric
		$h=\frac{1}{2}(\psi')^2 (dt^2 + dx^2)$ (see Appendix \ref{appendix:example2.2} for an explicit construction of $\tau$).
        Even though this example could be regarded as  steep up to a diffeo-rescaling 
         (as in Remark \ref{remark:lipschitz_uptorescaling}), it is possible to refine it so that the condition $\psi' (t_k) \to 0$ appears only around a (diverging) sequence of points $(t_k, x_k)$ instead of across the entire slice $t = t_k$.
        }
\end{example}
\begin{remark}
	Example  \ref{ex:h-steep_no-steep} can be generalized by showing that the steepness of a temporal function $t$ can {\em always} be  spoiled replacing $t$ by a suitable composition $\tau = \psi\circ t$, where $\psi$ is an increasing diffeomorphism. However, in general, a non-steep time function cannot be converted into a steep one by using such a composition (indeed, there are stably causal spacetimes which do not admit steep temporal functions \cite[Ex 3.3]{muller-sanchez}). The temporal or anti-Lipschitz character of a time function cannot be lost by composing with an increasing diffeomorphism, but it will be by composing with a smooth homeomorphism 
	which is not a diffeomorphism. 
\end{remark}
\medskip

\subsubsection{Case $h$ complete: global hyperbolicity and compatibility with steepness}\label{sec:theor3_12}

As seen in  Theorem \ref{theorem:splitting_boundary}, any globally hyperbolic spacetime-with-timelike-boundary admits an {\em adapted} steep Cauchy time function $\tau : \bM \to \mathbb{R}$ \cite{zepp:structure}, steepness being a property that ensures isometric embeddability in some $\LL^N$ 
\cite{muller-sanchez}. 
The equivalence of global hyperbolicity and the existence of an $h$-steep temporal function with respect to a complete $h$ was proven by Bernard and Suhr for closed cone fields \cite[Theorem 3]{bernard-suhr}, and, therefore, also for spacetimes-with-timelike-boundary. 
Burtscher and Garc\'{\i}a-Heveling \cite{GHnulldistance} linked  $h$-steepness of $\tau$ to the completeness of its associated null distance $\hat d_\tau$. {This allowed them to obtain a new characterization for global hyperbolicity and the codification of causality as in Theorem \ref{thm:GHnulldistance}.}
Next, we will revisit these results, checking that all the requirements (steepness, $h$-steepness with respect to a complete $h$ and adaptability to the splitting of the boundary) can be obtained at the same time for a Cauchy temporal function. 
This will be achieved by adapting the procedure in \cite[Theorem 1.2]{muller-sanchez} to the boundary, following the approach outlined in \cite{zepp:structure}, as suggested in \cite{sanchez:slicings}.

\begin{theorem}\label{thm:GHnulldistance-conborde}
Let $(\bM, g)$ be a spacetime-with-timelike-boundary and $\tau: \bM \to \mb R$ be a function (which is assumed $C^1$ in {\em (i)}):
\begin{enumerate}[label=(\roman*)]
\item If $\tau$ is $h$-steep (for a complete $h$) then $\tau$ is a temporal function and $\hat{d}_\tau$ a complete distance. 

\item If $\tau$ is a time function such that $(\bM, \hat{d}_{\tau})$ is a complete metric space, then $\tau$ is a Cauchy time function. In particular, $(\bM, g)$ is globally hyperbolic. 

\item If $(\bM, g)$ is globally hyperbolic, then there exists an {adapted} Cauchy  temporal function $\tau$ which is steep and $h$-steep with respect to a complete Riemannian metric $h$. 
\end{enumerate}
In particular, the following properties are equivalent for $(\bM, g)$: 
\begin{itemize}
\item[(a)] global hyperbolicity, 
\item[(b)] to admit an $h$-steep function with respect to a complete $h$, 
\item[(c)] to admit a time function $\tau$ such that $\hat d_\tau$ is complete, 
\item[(d)]  to admit a Cauchy time function,
\item[(e)]  to admit an adapted Cauchy steep temporal function,
\item[(f)]  (in the $C^3$ case) each  metric conformal to $g$ is isometrically embeddable in $\LL^N$ for large $N$.  

\end{itemize}
\end{theorem}

\begin{proof}
{\em (i)} $\tau$ is a temporal function by Proposition \ref{p_hsteep_is_time} and, by \eqref{e_hsteep_is_time}, 
$\tau (q) - \tau (p) 
\geq d_{h} (p,q)$. 
Thus, $\tau$ is {(globally) anti-Lipschitz} with respect to the complete distance $d_h$. In the case without boundary,  \cite[Theorem 1.6]{AllenBurtscher2022Properties} ensures that $\hat{d}_{\tau}$ is complete. However, { as all the involved properties are metric}, the proof is directly extended to  spacetimes-with-timelike-boundary: indeed, for any piecewise causal curve $\beta$ with break points $x_0,\dots ,x_k$, $\hat L_\tau(\beta)=\sum_i|\tau(x_i)-\tau(x_{i-1})|\geq \sum_i d_h(x_{i-1},x_i)\geq d_h(x_0,x_k)$, so $\hat d_\tau\geq d_h$; hence, a $\hat d_\tau$-Cauchy sequence is $d_h$-Cauchy and converges in the manifold topology, thus also for $\hat d_\tau$ (which is continuous in that topology).

{\em (ii)} Following \cite[Theorem 4.2]{GHnulldistance}, assume by contradiction that $(\bM, \hat{d}_{\tau})$ is a complete metric space but that $\tau$ is not a Cauchy time function. Then, there exists a future-directed  (or analogously past-directed)  inextensible causal curve $\gamma : \mathbb{R} \to \bM$ such that $\lim_{s \to \infty} \tau (\gamma (s) ) < \infty$. Consider the sequence $\{p_k\}_k$, with $p_k := \gamma (k)$ for $k\in \N$. Since the points $p_k$ are causally related we have, from the definition of $\hat{d}_{\tau}$, 
\begin{equation*}
	\hat{d}_{\tau} (p_k , p_m ) = |\tau (p_k) - \tau (p_m)|.
\end{equation*}
Since $\tau \circ \gamma : \mathbb{R} \to \mathbb{R}$ is strictly increasing and bounded from above, it follows that $\{p_k\}_k$ is a Cauchy sequence in $(\bM, \hat{d}_{\tau})$ and, by completeness, it must converge to a point  in the closure of the image of ${\gamma}$, in contradiction with the inextensibility of $\gamma$.

{\em (iii)} The proof proceeds along the same lines as in \cite[Theorem 1.2]{muller-sanchez}, by choosing at each step of the proof sufficiently large constants in a more restrictive manner, and checking that the construction permits to use the double manifold in order to be adapted to the boundary as in \cite{zepp:structure}; see Appendix \ref{appendix:thm_hsteep} for a detailed proof.

Next we prove the equivalences (a)--(e):  
{(a) $\Rightarrow$ (b)}. This follows from {\em (iii)} of this theorem. 
{(b) $\Rightarrow$ (c)} is the content of {\em (i)} above.
{(c) $\Rightarrow$ (d)} follows from {\em (ii)} as we are assuming that $\hat{d}_\tau$ is complete.
For implication {\em (d) $\Rightarrow$ (e)} note that admitting a Cauchy time function implies global hyperbolicity and point {\em (iii)} of this theorem ensures the existence of a (possibly different) adapted Cauchy steep temporal function; trivially, (e) $\Rightarrow$ (d) $\Rightarrow$ (a).
Finally, {(f) is equivalent to (a)} by Theorem \ref{theorem:splitting_boundary} {\em (i)}.
\end{proof}

%

\subsection{Causality is codified by $\hat{d}_{\tau}$
} \label{sec:nulldistance}

We can recover Burtscher and Garc\'{\i}a-Heveling's result that the null distance encodes causality for temporal functions \cite[Theorem 3.3]{GHnulldistance} with a direct adaptation of their proof. 

\begin{proposition}\label{thm:leonardo_codificacion}
	Let $(\bM, g)$ be a globally hyperbolic spacetime-with-timelike-boundary, and let $\tau$ be any  adapted Cauchy {temporal} function.  Then, the null distance $\hat{d}_\tau$ encodes causality; that is,
\begin{equation*}
		p \leq q \quad \Leftrightarrow \quad \hat{d}_\tau (p,q) = \tau (q) - \tau (p).
\end{equation*}
\end{proposition}

\begin{remark}[Comments on the proof]
	We follow the same ideas as \cite{GHnulldistance} just noticing that the properties obtained by the Hopf-Rinow theorem can be recovered in our setting using the Hopf--Rinow--Cohn-Vossen theorem \cite[Theorem~2.5.28]{burago:metric}. Specifically, metric completeness of a Riemannian metric $h$ in a spacetime-with-timelike-boundary implies that closed metric balls are compact. Moreover, the proof in \cite{GHnulldistance} relies on three results that are also available in the presence of a timelike boundary. The Wick-rotated metric (see \S~\ref{sec:g_W}) is conformal to a complete Riemannian metric $h$. Thus, \cite[Lemmas~2.4 and 3.5]{GHnulldistance} translate directly to our setting since they deal with compact sets. Since the exponential map is a local diffeomorphism, the $h$-metric derivative of an absolutely continuous curve coincides with the $h$-norm of its velocity vector as in \cite[Corollary~3.8]{Burtscher2015Length}.
%
%
\end{remark}

In the case without boundary, the regularity of $\tau$ in Proposition \ref{thm:leonardo_codificacion}  can be lowered to being a {\em locally anti-Lipschitz time} function \cite[Theorem 1.9]{GHnulldistance}. Their proof relies on Sakovich
and Sormani’s local causality encoding \cite[Theorem 1.1]{Sakovich2023encodescausality}, which uses special
coordinate systems.  Moreover, Galloway \cite{Galloway2024encoding} proved that the causality assumption on the level sets of $\tau$ to obtain a global encoding of causality can be weakened to being {\em future causally complete}.  The possible generalization of such results for  spacetimes-with-timelike-boundary  is postponed to future work.

Next, we take a step further by showing that when the null distance $\hat d_{\tau}$ encodes causality, 
bijections that preserve both the null distance and the time function are conformal diffeomorphisms.
A closely related result was previously established by Sakovich and Sormani \cite[Thm 1.3]{Sakovich2023encodescausality}, who proved that a bijective map $F: (M,g,\tau) \to (M^*, g^*, \tau^*)$ between spacetimes equipped with regular cosmological time functions \cite{Galloway1998cosmologicaltime}, and preserving both time values and null distances, must be a Lorentzian isometry; that is, $F^* g^* = g$. Their argument proceeds by first showing that $F$ is a causal bijection, which -- by a result of Levichev \cite{Levichev1987causalstructure} -- implies that $F$ is a conformal diffeomorphism. Then, using the properties of cosmological time functions, they conclude that $F$ is in fact an isometry. We prove a similar result, with the key difference that ours applies to any locally anti-Lipschitz time function. However, in our case, $F$ need not be an isometry but a conformal diffeomorphism, consistent with the conformal invariance of the null distance and our focus on the causal structure. We restrict to the case without boundary.


\begin{theorem}  \label{thm:isometrias_distancia_nula}
Let $(M, g)$ and $(M^*, g^*)$ be 
spacetimes 
of dimension $n+1\geq 3$ and, respectively,  let $\tau$, $\tau^*$ be locally anti-Lipschitz time functions\footnote{\label{pie}Locally anti-Lipschitz time functions yield a positive definite metric space $(M, \hat{d}_\tau)$ which induces the manifold topology \cite[Thm 4.6]{sormanivega}.} 
such that $\hat{d}_{\tau}$ and $\hat{d}_{\tau^*}$ encode the causality on $M$ and $M^*$, respectively.\footnote{Note that this assumption is satisfied when $\hat{d}_{\tau}$ and $\hat{d}_{\tau^*}$ are complete, in which case, the functions $\tau$ and $\tau^*$ are locally anti-Lipschitz Cauchy time functions whose null distances encode causality (see 
Theorem~\ref{thm:GHnulldistance}). This is also the case when one considers {\em proper regular cosmological time functions}, as in \cite{Sakovich2023encodescausality}.}

Then, a time-preserving bijection $F: M \rightarrow M ^*$ preserves the null distances if and only if $F$ is a conformal diffeomorphism (possibly reversing the time orientation).\footnote{In the excluded dimension $n+1=2$ the conclusion fails: if $F(u,v)=(f(u),g(v))$ in null coordinates of $\LL^2$, with $f,g$ increasing bi-Lipschitz homeomorphisms of $\R$ which are not $C^1$, then $F$ preserves the causal relations and, for $\tau^*:=\tau\circ F^{-1}$, both the time functions and the null distances, but it is not a diffeomorphism.}
\end{theorem} 
\begin{proof} The sufficient condition is trivial. 

For the converse, as $F$ preserves time, we have $\tau = \tau^* \circ F$, and because it also preserves null distances, we obtain, for any $p \leq q$,
\begin{equation*}
    \hat{d}_{\tau^*} (F(p), F(q)) = \hat{d}_\tau (p,q) = \tau (q) - \tau (p) = \tau^{*} (F(q)) - \tau^* (F(p)).
\end{equation*}
By the codification of causality via null distance, this implies that $F(p) \leq F(q)$, so $F$ is a causal bijection, that is,
\begin{equation*}
	p \leq q \;\;\Leftrightarrow\;\; F(p) \leq F(q).
\end{equation*}
Note that $F$ is also a homeomorphism as both $\hat{d}_\tau$ and $\hat{d}_{\tau^*}$ induce the manifold topology \cite[Prop. 3.5]{sormanivega}. Therefore $F: M \to M^*$ is a conformal diffeomorphism \cite[Theorem~6]{HKM}. 
\end{proof}

\section{Wick-rotated Riemannian geometry 
} \label{sec:g_W}


\subsection{Conformally invariant $g^{\tau}_W$ associated with an $h$-steep $\tau$} \label{subsec:conformally_invariant_g_W}

Given a spacetime-with-timelike-boundary 
$(\bM, g)$ endowed with a 
temporal function $\tau$, one can write the metric $g$ as  in \eqref{eq:splitting_boundary}: 
\begin{equation*}\label{eq:splitting-tilde}
g = - \Lambda d\tau ^2 + \sigma_{\tau} =: \Lambda (- d\tau ^2 + \bar{\sigma}_{\tau}) =: \Lambda g^\tau,
\end{equation*}
where $\Lambda= |\nabla \tau|^{-2}$ and the last equality defines a representative of the conformal class. {This can be done even if $\bM$ does not split as a product manifold (using the orthogonal splitting
$T\bM= $ Span$(\nabla \tau) \oplus (\nabla \tau) ^\perp$ no matter if $\bM$ truly splits).} 

\begin{definition}\label{d_wick} Let  $(\bM, g, \tau)$ be a   spacetime-with-timelike-boundary endowed with 
	a {\em temporal} function $\tau$. 
	The {\em canonical representative} or {\em $\tau$-representative} $g^\tau$ of the conformal class $[g]$ of $g$ is 
		$$g^\tau = \frac{1}{\Lambda}g= -d\tau ^2 + \bar{\sigma}_{\tau}, $$
		(that is, the $g^\tau$-gradient of $\tau$ is unit; in particular,  $\tau$ is steep for $g^\tau$). 
The Riemannian {\em Wick-rotated} metric of (the conformal class of) $g$ (associated with the temporal function $\tau$) is 
 \begin{equation*}
 	g^{\tau}_W = d\tau ^2 + \bar{\sigma}_{\tau}.
 \end{equation*}
\end{definition}
The Wick-rotated metric is positive definite and conformally invariant. These properties are 
also satisfied for the null distance when $\tau$ is a $C^1$ temporal function \cite{sormanivega}.   
Next, we will also prove that it is complete when $\tau$ is $h$-steep (for a complete $h$). This will give us a characterization of global hyperbolicity and $h$-steepness, in a similar way as in the case of the null distance.


\subsection{Completeness of the Wick-rotated $g^{\tau}_W$}\label{subsec:4_2} 
%

Let $| \cdot |_W^\tau$, $L_W^\tau$ and $d_W^\tau$ denote the norm, length functional and distance associated with the metric $g_W^{\tau}$. \
We will denote $\nabla^W$ the gradient w.r.t $g_W^\tau$.
We will study the completeness of $g_W^\tau$ for $\tau$ $h$-steep with $h$ complete, which implies that $\tau$ is onto $\R$ (inextensible causal curves have infinite $h$-length; notice that a diffeo-rescaling, as in Remark \ref{remark:lipschitz_uptorescaling}, would destroy the completeness of $h$).



The following fundamental result will complement those of Section \ref{sec:temporalfunctions}.

\begin{lemma} \label{lemma:g_w-completa}
Let $\tau$ be a temporal function and $v=v^\tau+v^\perp$ a lightlike vector, where $v^\tau \in$ Span$(\nabla \tau)$ and $v^\perp \in (\nabla \tau)^\perp$. Then:
\begin{equation} \label{e_desc_dtv}
d\tau(v)^2=g_W^\tau(v^\tau,v^\tau)=g_W^\tau(v^\perp,v^\perp) .
\end{equation}
Therefore, for any   Riemannian metric $h$: 
$$d\tau(w)^2 \geq h(w,w), \quad \forall w \; \hbox{causal} 
\qquad \Longrightarrow \qquad  2 g_W^\tau(v,v) \geq h(v,v), \quad \forall v\in T\bM, $$
(the antecedent meaning that $\tau$ is $h$-steep). In particular, if $h$ is complete, then so is $g_W^\tau$.

\end{lemma}

\begin{proof} The first displayed equation follows from the definition of $g_W^\tau$. 

	 The implication in the second one is trivial for $v\in $ Span$(\nabla \tau)$ as well as when  $v=v^\tau+v^\perp$ with 
\begin{equation*}
	g_W^\tau(v^\tau,v^\tau)\geq g_W^\tau(v^\perp,v^\perp),
\end{equation*}
since it implies that $v$ is causal. 

Note that $\nabla^W \tau$ is unit (as it comes from the Wick-rotation of $g^\tau$). Thus, for $v \in (\nabla \tau)^{\perp}$, we can construct lightlike vectors 
$$w^{\pm} = \sqrt{g^{\tau}_W (v,v)} \nabla^W \tau \pm v.$$
By \eqref{e_desc_dtv} and the assumed $h$-steepness of $\tau$, we have:
\begin{align*}
	g^\tau _W (v,v) &= d\tau^2 (w^{\pm}) \geq h(w^{\pm},w^{\pm}) \\
	&= g^{\tau}_W (v,v) h(\nabla^W \tau, \nabla^W \tau) \mp 2\sqrt{g^{\tau}_W (v,v)} h(\nabla^W \tau, v) + h(v,v),
\end{align*}
where the term $\mp 2\sqrt{g^{\tau}_W (v,v)} h(\nabla^W \tau, v)$ can be chosen positive for $+$ or $-$. Thus $g^\tau _W (v,v) \geq h(v,v)$.

{More generally, if $v$ satisfies}
\begin{equation*}
	g_W^\tau(v^\tau,v^\tau) \leq g_W^\tau(v^\perp,v^\perp),
\end{equation*}
we can consider the $g_{W}^\tau$-orthonormal basis $B=\{\frac{v^\tau}{|v^\tau|},\frac{v^\perp}{|v^\perp|}\} = \{w_1, w_2\}$ of the plane $\pi_v:=$ Span$\{v^\tau,v^\perp\}$.  In this basis  the matrix  $\begin{pmatrix}
a & b\\b & c
\end{pmatrix}$ of $h|_{\pi_v}$ satisfies $0<a,c\leq 1$ (by the two previous cases), $b^2<ac$  and the eigenvalues $\lambda_{\pm} > 0$ (since $h$ is Riemannian), where 
$\max\{\lambda_+, \lambda_-\} \leq 2$ since $\lambda_+ + \lambda_- = a + c \leq 2$. Thus, if $v = x w_1 + y w_2$ in the basis B, we have that
\begin{equation*}
	h (v,v) = \begin{pmatrix}
		x & y
	\end{pmatrix} \begin{pmatrix}
		a & b \\
		b & c
	\end{pmatrix} \begin{pmatrix}
	x \\
	y
	\end{pmatrix} \leq \max \{\lambda_+ , \lambda_-\} (x^2 + y^2) \leq 2 g^\tau_W (v,v),
\end{equation*}
(the first inequality being a consequence of the Rayleigh quotient) where $g^\tau_W (v,v) = x^2 + y^2$.
\end{proof}

The following result will be the key to relate Lorentz and Riemannian convergences.

\begin{theorem} \label{intro:thm_A}
	Let $(\overline{M}, g)$ be a stably causal spacetime-with-timelike-boundary and choose an onto temporal function $\tau: \bM \rightarrow \R$. The following properties are  equivalent:
	\begin{itemize}
		\item $\tau$ is $h$-steep for a complete metric $h$,
		\item the Wick-rotated (Riemannian) metric $g_W^{\tau}$ is complete.
	\end{itemize}
	In this case, (a) $\tau$ is $h$-steep with respect to 
	$h=g_W^{\tau}/2$, and (b) $\tau$ is a Cauchy temporal function.
\end{theorem}

\begin{proof} 
If $\tau$ is $h$-steep for a complete metric $h$, Lemma \ref{lemma:g_w-completa} gives us that $g^\tau_W$ is complete. For the converse, as well as assertion (a), the proof is the same as Prop. \ref{p_hsteep_is_time} (1). Assertion (b) was proved in Theorem \ref{thm:GHnulldistance-conborde}.
\end{proof}

As a corollary to {Lemma \ref{lemma:g_w-completa} and Theorem \ref{intro:thm_A},} we obtain a new characterization of global hyperbolicity that also encompasses the case with timelike boundary.

\begin{corollary}\label{coro:wick-rotated_globally hyperbolic}
		A spacetime-with-timelike-boundary $(\overline{M},g)$ is globally hyperbolic if and only if it admits a temporal function $\tau$ such that $g_W^{\tau}$ is a complete Riemannian metric.
\end{corollary}

\begin{remark}\label{remark:g_w-completa_integralcurrent}
	(1) If a Wick-rotated metric $g_{W}^{\tau}$ is complete, then the canonical representative of the conformal class 
	$g^\tau = -d\tau ^2 + \bar{\sigma}_\tau$ restricts as a complete Riemannian metric $\bar{\sigma}_{\tau_0}$ on each slice $\{\tau_0\} \times \bar{\Sigma}$. In particular, our Lemma \ref{lemma:g_w-completa} gives sufficient conditions for the slices to be complete.
	
	(2) One of the authors constructed an example $(M,g)$ where $\tau$ is a Cauchy temporal function  and $g$ is the canonical representative of the conformal class, but it has incomplete slices, see   \cite[\S 6.4]{sanchez:slicings}. This cannot occur in the case that $\tau$ is $h$-steep (with $h$ complete), as proved in Lemma \ref{lemma:g_w-completa}.	
	\end{remark}

	\begin{remark}\label{r_correccion}
    Burtscher and García-Heveling \cite[Remark 2.7]{GHnulldistance} point out that the existence of $h$-steep temporal functions $\tau$ in globally hyperbolic spacetimes  implies that $\hat d_\tau$ admits a natural structure of local integral current. This relies on the fact that $\tau$ fulfills  the sufficient conditions on completeness imposed by Allen and Burtscher in \cite[Theorem 1.3]{AllenBurtscher2022Properties}. Notice that analogous conditions on completeness are also satisfied in the case of globally hyperbolic spacetimes-with-timelike-boundary, suggesting the extension of results about  integral currents to the case with boundary.
\end{remark}

\subsection{The metric $d_W^\tau$ and the encoding of causality}
\label{subsec:4_3}

Following Burtscher and Garc\'ia-Heveling,  Proposition \ref{thm:leonardo_codificacion} establishes that the null distance associated with an {\em adapted Cauchy temporal function} encodes the causal structure in globally hyperbolic spacetimes-with-timelike-boundary. The Wick-rotated distance $d_W^\tau$, in turn, trivially satisfies:

\begin{proposition} For any globally hyperbolic spacetime-with-timelike-boundary and {Cauchy temporal} function $\tau$:
	\begin{itemize}
		\item[(i)] The projection $\bM=\R\times \bSigma \rightarrow \R$ for the $\tau$-representative is a Riemannian submersion; in particular, 
		$d_W^\tau (p,q)\geq |\tau(q)-\tau(p)|$.
	
	\item[(ii)] For any couple of points $p,q\in \bM$ connected by an (unbroken) lightlike segment $\gamma$, the following identity holds: 
	$L_W^\tau(\gamma)=\sqrt{2} \, |\tau(q)-\tau(p)|$.
	\end{itemize}
	In particular, if $p\leq q$ then
	$$|\tau(q)-\tau(p)| \leq d_W^\tau (p,q) \leq  \sqrt{2}  \, |\tau(q)-\tau(p)|.$$
\end{proposition} 
However, the converse does not hold: the Wick-rotated distance $d_W^\tau$ does not encode causality (as the null distance does in Proposition \ref{thm:leonardo_codificacion}), even when $(\bM, d_W^\tau)$ is a complete -- and hence globally hyperbolic -- metric space, and even if the definition of $d_W^\tau$ is restricted to piecewise causal curves. An explicit counterexample is provided in Appendix \ref{app:no_codifica}.

Indeed, the null distance $\hat{d}_\tau$ encodes causality because, unlike the broad class of piecewise smooth curves used to define $d_W^\tau$, $\hat{d}_\tau$ is constructed using piecewise causal curves -- a class that inherently reflects the causal structure of the spacetime. Indeed, the factor $\sqrt{2}$ was introduced because it is optimal for lightlike curves and {\em if one were to define $d_W^\tau$ using only piecewise null curves\footnote{This class of curves can also be considered to compute the null distance \cite[Remark 3.7]{sormanivega}.}, a similar encoding of causality would emerge}.

More precisely, using $g_W^\tau$ in combination with the temporal function define a null distance-type metric now in terms of $L_W^\tau$, namely
$$\hat{d}_W^\tau (p,q) := \inf \{L_W^\tau (\alpha): \alpha \text{ is a piecewise null curve from } p \text{ to } q\}.$$
Clearly, if $\alpha$ is a piecewise null curve from $p$ to $q$ then
	$$L_W^\tau (\alpha) = \sqrt{2} \hat{L}_\tau (\alpha).$$
Consequently:
\begin{proposition}\label{prop:d_W_encoding}
The identity $\hat d_W^\tau= \sqrt{2} \hat d_\tau$ holds
(thus, $\hat d_W^\tau/\sqrt{2}$ also encodes causality).
\end{proposition}

%
%

This result, together with the previous ones, suggests that the null distance 
implicitly incorporates  the  information about the time function used in its definition. Therefore, using the conformal splitting given directly by the temporal function entails no loss of generality in relation to the null distance.

\subsection{Stability of $h$-steep  functions and adaptability to $\partial M$}\label{subsec:4_4}
The stability of a property satisfied by the metric $g$ (or by a related object, as time functions in Remark \ref{remark:weak_temporal} (b)), means that the property holds not only for $g$ itself but also for any metric within a $C^0$-neighborhood of $g$. In the case of conformally invariant properties -- such as global hyperbolicity  or, more restrictively, the Cauchy character of the slices of a Cauchy temporal function $\tau$ -- stability means that the property remains valid just under a slight widening of the light cones of $g$ (since it would then automatically hold for any metric with narrower cones)\footnote{This also occurs for all the causal properties considered here. However, in the standard causal ladder  of spacetimes, this is not true for {\em  causal continuity} and {\em causal simplicity} \cite{garciaparrado-sanchez}, where one must also examine narrower cone structures to fully capture these properties.}. As checked in  \cite[Theorem 4.2 and Remark 4.9]{zepp:structure}, being Cauchy and steep are stable properties for any temporal function on a spacetime. 

Furthermore, Bernard and Suhr \cite{Bernard2020Cauchyanduniform} proved that being an $h$-steep function (with $h$ complete) is a stable property for {\em hyperbolic cone fields}, which include globally hyperbolic spacetimes-with-timelike-boundary. 

With respect to the adaptability to the boundary, the fact that   $\nabla \tau$ is  tangent to $\partial M$ is conformally invariant but, clearly, it is not a stable property.\footnote{Notice that, here, it is not relevant if one chooses the gradient for the original metric, its $\tau$-representative or the Wick-rotated one.} However, it is possible to obtain a (globally hyperbolic) metric $g'$ with $g < g'$ such that the $g'$-gradient of $\tau$, $\nabla' \tau$, remains tangent to $\partial M$ and $\tau$ is still Cauchy temporal for $g'$. To prove this, we will follow the ideas in the proof of \cite[Thms 4.1 and 4.2]{zepp:structure}.

\begin{theorem}\label{thm:stability} Let $(\bM,g)$ be a (necessarily globally hyperbolic) spacetime-with-timelike-boundary and $\tau$ an $h$-steep (resp. Cauchy steep; $h$-steep and steep) function, with respect to a complete Riemannian metric $h$, adapted to the boundary. Then, there is a metric $g'$ with wider cones ($g<g'$) such that 
$\tau$ is also $h'$-steep (resp. Cauchy steep; $h'$-steep and steep) for a complete (possibly different) Riemannian metric $h'$ and adapted to $\partial M$ for $g'$ (i.e. $\nabla' \tau$ remains tangent to $\partial M$).

Moreover, the Wick-rotated metric of $g'$ is also complete. 
\end{theorem} 

\begin{proof}

For the adaptability to the boundary for $g'$, just notice that the constructive procedure in \cite{zepp:structure} (maintaining the steep structure) uses a deformation of the type $g'=
g-\alpha d\tau^2$ for a suitable function $0<\alpha<\Lambda$. This metric preserves the orthogonal decomposition, and so, the directions of  $\nabla' \tau$ and $\nabla \tau$, which are preserved, must coincide.\footnote{However, not all the metrics $g''$ satisfying $g<g''<g'$ will preserve this property, in general.} 

%

Indeed, $g'=-(\Lambda+\alpha)d\tau^2+\sigma_\tau$, so $\tau$ is temporal for $g'$ and its Wick-rotated metric satisfies ${g'}_W^\tau=d\tau^2+\sigma_\tau/(\Lambda+\alpha)\geq \frac{1}{2}g_W^\tau$, which is complete. Then, Theorem \ref{intro:thm_A} applied to $g'$ shows that $\tau$ is $h'$-steep for the complete metric $h':={g'}^\tau_W/2$ and Cauchy for $g'$; moreover, if $\tau$ is steep for $g$, i.e. $1/\Lambda\geq c$, then $g'(\nabla'\tau,\nabla'\tau)=-1/(\Lambda+\alpha)\leq -c/2$, so it is also steep for $g'$. The last assertion is now immediate since $\frac{1}{2}g_W^\tau < {g'}_W^\tau$ for $0 < \alpha < \Lambda$.
\end{proof}

\begin{remark}\label{remark:convenient_is_stable}
Note that the Riemannian metric $h'$ can be different for $g'$, since the cones are wider. This is also the case in \cite{Bernard2020Cauchyanduniform}, where it is proved that $\tau$ is $\frac{1}{4} h$-steep for the enlarged cone structure.
\end{remark}

\section{Semi-Riemannian Cheeger-Gromov convergence }\label{sec:convergence_semi}


In this section, we address the challenges of studying convergence in the semi-Riemannian setting, highlighting key differences from the well-established Riemannian theory. We introduce a notion of smooth convergence suited to semi-Riemannian manifolds and prove the uniqueness of the limit within any convergent sequence.

We begin with a brief review of smooth convergence for Riemannian manifolds, following the framework established in \cite[Definition 1.1]{Anderson2004CheegerGromov}.

%

\subsection{Brief review on convergence for Riemannian manifolds}\label{subsec:5_1}

The notion of convergence considered here is always uniform on compact subsets. For sequences of tensor fields defined on a fixed manifold, or for sequences of maps between two given manifolds, this concept has a standard and well-understood meaning. However, when dealing with sequences of Riemannian manifolds, the presence of intermediary diffeomorphisms necessitates a more refined approach. To set the stage, we begin by recalling the definition of smooth convergence of Riemannian manifolds following standard references such as \cite{Petersen, Chow2007ricciflow}.


\subsubsection{Setting and uniqueness of the limits}

\begin{definition}[Cheeger-Gromov convergence of Riemannian manifolds] \label{def:riemann_convergence}
	A sequence of pointed complete Riemannian manifolds $\{(M_i, g_i, p_i)\}_{i\in\mathbb{N}}$ is said to {\em converge in the (pointed) $C^{k, \alpha}$ topology} to a pointed complete Riemannian manifold $(M,g,p)$ if there exist
	\begin{enumerate}[label=(\roman*)]
		\item an exhaustion $\{U_i\}_{i \in \N}$ of $M$ by open sets with $p \in U_i$, and
		\item a sequence $\{\phi_i\}_{i \in \N}$ of embeddings $\phi_i : U_i \to V_i := \phi_i (U_i) \subset M_i$ with $\phi_i^{-1} (p_i) \rightarrow p$,\footnote{The usual definition of Cheeger-Gromov convergence requires that $\phi_i (p) = p_i$ for all $i$. However, this condition can be relaxed to convergence of the basepoints in the Riemannian case with no further problem (cf. \cite{bamler2007ricci}). We choose this weaker condition to compare with the semi-Riemannian case and explore the difficulties that this weakening brings to the convergence.}
	\end{enumerate}
	such that for a locally finite collection $\{\varphi_\beta\}_{\beta\in\mathcal{B}}$ of charts covering $M$, we have that 
	\begin{equation*}
		(\phi_i^* g_i)_{jk} \xrightarrow{C^{k, \alpha}} g_{jk}\quad\hbox{uniformly on compact sets,}
	\end{equation*}
	where $(\phi_i^* g_i)_{jk}$ and $g_{jk}$ are the local component functions of the metrics in the charts $\{\varphi_\beta\}_{\beta\in\mathcal{B}}$.
\end{definition}

\begin{remark} \label{remark:riemann_convergence}
	\begin{enumerate}[label=(\Alph*)]
		\item When working with $C^k$ convergence one can trivially deduce 
		that for any compact subset $K \subset M$ and every $\epsilon > 0$, there exists $i_0 = i_0 (\epsilon)$ such that for $i \geq i_0$,
		$$\sup_{0 \leq r \leq k} \sup_{x\in K} |\nabla^{r} (\phi_i ^* g_i - g)|_h < \epsilon,$$
		where the covariant derivative $\nabla$ is taken with respect to an auxiliary Riemannian metric $h$. Since we work on compact sets, the choice of the metric $h$ on $K$ does not affect the convergence, in particular, we can choose $h$ equal to the limit metric $g$. Also, w.l.o.g., one can assume that $U_i = B_g(p,i)$, the $g$-ball of center $p$ and radius $i$. This perspective also provides an alternative formulation of $C^k$ convergence in Definition~\ref{def:riemann_convergence}, as in \cite{Chow2007ricciflow}, which is used in the proof of the standard compactness result (Theorem~\ref{thm:compactness_theorem_riemann} below).

		\item The {\em basepoint} (or origin) is carried along with the manifold and the Riemannian metric to indicate which regions of the manifolds in the sequence remain in focus -- an essential feature for ensuring the uniqueness of the limit (cf.  \cite[Example~8.2.2]{HopperAndrews2011ricciflow}). This framework allows one to compare sequences of different Riemannian manifolds with the limit in both non-compact and compact spaces (notice that in the former case the diameters may be finite but diverging). In the case where all the manifolds (including the limit) are compact, the maps $\phi_i : M \to M_i$ are global diffeomorphisms (for all but finitely many $i$), making it possible to consider  unpointed convergence.
		
	\end{enumerate}
\end{remark}

\begin{convention}
 {\em    For simplicity, in what follows, the convergence of (semi-)Riemannian manifolds will be assumed $C^k$ with $k=0,1,  \dots, \infty$, except if otherwise specified.
 Thus, when we say that a sequence of Riemannian metrics \textit{converges in $C^k$ (via diffeomorphisms $\{\phi_i\}$)}, we mean convergence in the sense of Definition \ref{def:riemann_convergence}, which is equivalent to condition~(A) in Remark \ref{remark:riemann_convergence}.}

\end{convention}

The following result is foundational for  the Riemannian setting.

\begin{theorem}\label{prop:uniqueness_riemann_convergence}
	If a sequence $\{(M_i, g_i, p_i)\}$ of pointed complete Riemannian manifolds converges in $C^0$, then the limit is unique up to isometries.
\end{theorem}

It is worth pointing out that
 $C^0$ convergence suffices because it  implies convergence in the pointed Gromov-Hausdorff sense (see \cite{Petersen}). Moreover, the limit of a sequence converging in this sense is unique up to a distance-preserving map (see \cite[Theorem 8.1.7]{burago:metric}). 
 By the Myers-Steenrod theorem, any such distance-preserving map between smooth Riemannian manifolds is necessarily a smooth isometry (see \cite[Thm 5.6.15]{Petersen}).

\subsubsection{Existence of limits: compactness Riemannian theorem}

We now collect some useful results that guarantee convergence for sequences of Riemannian manifolds. The standard framework extends naturally to Riemannian manifolds with boundary. However, extending compactness results -- i.e., those ensuring the existence of convergent subsequences -- to the manifold with boundary case is more delicate and will be addressed in future work, both in the Riemannian and Lorentzian settings. 

As a preliminary step toward establishing compactness for sequences of complete Riemannian metrics on a fixed manifold, we first recall a compactness theorem for sequences of complete pointed Riemannian manifolds. We first examine the potential obstructions to compactness, before recalling the key lemma and the compactness theorem itself.

\begin{remark}\label{remark:cotas_curvatura_radio_inyectividad}
	Intuitively, there are two main phenomena that may prevent a sequence of complete pointed Riemannian manifolds from having a subsequence that converges to a complete Riemannian manifold: either the dimension drops in the limit, because of the injectivity radius going to zero; or non-smooth features such as edges or corners develop when the curvature blows up (cf. \cite[\S 3.2.2]{bamler2007ricci} and \cite[\S 8.3]{HopperAndrews2011ricciflow}, the latter with an explicit construction).
\end{remark}

The following lemma allows one to locally bound the metric coefficients and their derivatives on normal coordinates when one has control over both the injectivity radius at a point and a bound on the curvature (see \cite[Proposition 4.32]{Chow2007ricciflow}, whose proof can be found in \cite[Corollary 4.11]{hamilton95a}).

\begin{lemma}\label{lemma:uniform_bounds_metrics_derivatives}
	Let $(M^n,h)$ be a Riemannian manifold {without boundary}. Let $p \in M$ and $r_0 \in (0, \frac{1}{4}\inj (p))$. Assume that for $l \geq 0$ there are constants $C_0,...,C_l < \infty$ such that
	\begin{equation*}
		| \nabla^a Rm | \leq C_a \quad \text{ in }\;\; B(p,r_0)\;\; \text{ for }\;\; a = 0,..., l.
	\end{equation*}
	Then, there are constants $\tilde{C}_0,..., \tilde{C}_l$ depending on $n$, $\inj (p)$, $C_0,..., C_l$, and a constant $c_1$ depending only on $n$ such that for any multi-index $\alpha$ with $|\alpha| \geq 1$,
	\begin{equation*}
		\begin{array}{cl}
		\frac{1}{2} (\delta_{ij}) \leq (h_{ij}) \leq 2 (\delta_{ij}) & \text{(as scalar products)} \quad \text{and} \\[4pt]
		\left|\frac{\partial ^\alpha h_{ij}}{\partial x^\alpha} \right| \leq \tilde{C}_{|\alpha|} &\text{ in } B(p, \min\{c_1 / \sqrt{C_0}, r_0\}),
		\end{array}
	\end{equation*}
	where $\{x^\mu\}$ are normal coordinates on $B(p, r_0)$.
\end{lemma}

The following is the fundamental (sequential) compactness theorem for sequences of complete pointed Riemannian manifolds (see \cite{hamilton95a}; a full proof is given in \cite[Chapter 4]{Chow2007ricciflow}).

\begin{theorem}[]\label{thm:compactness_theorem_riemann}
	Let $\{(M_i, g_i, p_i)\}_{i\in\N}$ be a sequence of complete pointed $C^\infty$ Riemannian manifolds without boundary satisfying
	\begin{enumerate}[label=(\arabic*)]
		\item \emph{uniformly bounded geometry,}
		$$\left| \nabla_i^{a} Rm_i \right|_i \leq C_a \;\; \text{on } M_i,\quad\hbox{for all $i\in \mb N$ and for each $a \geq 0$,}$$
		 where $\{C_a\}_{a\in\mathbb{N}}$ is a sequence of real constants independent of $i$, and the
		\item \emph{injectivity radius estimate},
		$$\inj_i (p_i) \geq \iota_0,\quad\hbox{for all $i\in \mb N$ and for some constant $\iota_0 > 0$.}$$
	\end{enumerate}
	Then, there exists a subsequence of $\{(M_i,g_i,p_i)\}_{i \in \N}$ converging in $C^{\infty}$ to a complete pointed Riemannian manifold $(M,g,p)$ in the sense of Definition \ref{def:riemann_convergence}. 
\end{theorem}

In particular, any sequence $\{g_i\}_{i\in\mathbb{N}}$ of complete Riemannian metrics on a manifold without boundary $M^n$, satisfying (1) and (2) of the previous theorem, the latter at a fixed point $p = p_i$ for all $i$, admits a subsequence converging to a complete Riemannian metric $g_\infty$.


\subsection{Notion of  convergence for the semi-Riemannian case} \label{subsec:convergence_semi-riemann}

We now turn to the question of convergence for spacetimes and, more generally, for semi-Riemannian manifolds of arbitrary signature.

\subsubsection{Discussion} \label{subsec:5_2_1}

In the Riemannian case, the positive-definite signature plays a central role in addressing three fundamental aspects, to be faced in the semi-Riemannian setting.

\smallskip 
\smallskip
\noindent {\em 1. Compactness of the  isometry group}. This fact  is crucial in proving the uniqueness of the limits for  the Riemannian case, i.e., the construction of an isometry between any two potential limits in Theorem~\ref{prop:uniqueness_riemann_convergence}. 
M. Anderson already observed in \cite[\S 5]{Anderson2004CheegerGromov}, that the non-compactness of the isometry group poses a significant obstacle when attempting to apply Cheeger-Gromov theory in the Lorentzian setting. 
The following examples illustrate the difficulty of just transplanting the Riemannian notion considering pointed  uniform compact convergence {\em up to diffeomorphisms}, even under geodesic completeness or compactness.

\begin{example}\label{ejemplo1}{\em 
(a) Consider pointed Lorentz-Minkowski space in null coordinates $(\LL^2=(\R^2, -2dudv), 0)$. Define the boost $\mathbb{B}(u,v)=(e^{a}u,e^{-a}v)$, $a>0$. Then both the identity and the iterates $\{\mathbb{B}^k\}_k$ define sequences converging to $(\LL^2,0)$, but in different ways. However, there is no way to construct a pointed isometry $(\LL^2,0)\rightarrow (\LL^2,0)$ from the sequences, in sharp contrast with Theorem \ref{prop:uniqueness_riemann_convergence} (compare with the case of rotations in $\R^2$).

(b) Now, modify $\LL^2$ in a small neighborhood $U$ of the point $(1,0)$ to obtain a non-flat metric $g$, and consider the 
sequence $\{(\R^2, (\mathbb{B}^{-k})^*g,0)\}_{k\in\N}$. The pullback $(\mathbb{B}^{-k})^*g$ translates the region $U$ towards infinity as $k\rightarrow\infty$ (it is non-flat exactly on $\mathbb{B}^{k}(U)$). Thus, for any compact subset $K\subset \R^2$, the metric $(\mathbb{B}^{-k})^*g$ is isometric to the flat metric $-2dudv$ on $K$ for sufficiently large $k$, that is, the pointed sequence $\{(\R^2, (\mathbb{B}^{-k})^*g,0)\}_{k\in\N}$ converges uniformly on compact subsets to $(\LL^2, 0)$ (see Figure \ref{fig:example_uniqueness}). Nonetheless, each $(\mathbb{B}^{-k})^*g$ is isometric to $g$ (being $\mathbb{B}^{k}$ an isometry which preserves $0$), implying that the sequence admits a second, distinct, non-isometric pointed limit (up to diffeomorphisms).
}    
\end{example}

\begin{figure}[h]
	\centering \includegraphics[scale=0.75]{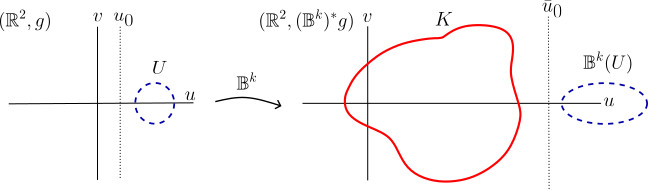}
	\caption{Example \ref{ejemplo1}(b). The neighborhood $U$ is contained in the region $\{u > u_0>0.\}$. For any compact subset $K \subset \R^2$, there exists $\bar{u}_0>0$
     such that $K \subset \{u < \bar{u}_0\}$ and $k \in \N$ such that  $\mathbb{B}^{k} (U)$ is contained in the region $u > \bar{u}_0 $.}
	\label{fig:example_uniqueness} 
\end{figure}

\begin{example}[\cite{Mounoud2003}]\label{ejemplo2}{\em
Non-uniqueness of limits may also arise on compact manifolds. Indeed, consider the moduli space of flat metrics with volume equal to  1 on a torus, up to isometries. The orbit of almost any flat torus under the action of the diffeomorphism group is dense in this moduli space. That is,  almost any unit volume flat metric $g_0$ on $\R^2 / \Z ^2$ satisfies that, for any other unit volume  flat metric $g_1$ on $\R^2 / \Z^2$, 
there exists a sequence $\{f_i\}_i$ of diffeomorphisms $f_i \in SL (2, \Z)$ such that the sequence of pull-backs $\{f_i ^* g_0\}_i$ Cheeger-Gromov converges to $g_1$, see \cite{Mounoud2003}. Notice that non-isometric flat Lorentzian tori may exhibit pathologies (from the Riemannian point of view), such as non-compact isometry groups. }
\end{example} 
This issue will be addressed satisfactorily by using  an {\em anchor}, as in Definition \ref{def:anchor}, because the group of  linear isometries preserving the anchor is compact.

    
        \smallskip 
\smallskip
\noindent {\em 2.  Norms and quasi-isometries}. 
 In the Riemannian case, the positive-definite structure provides a natural pointwise norm for tensors  for each manifold. 
 The fact that the metrics are  converging permits a comparison of norms with the limit and, then, to obtain quasi-isometries, which are relevant for uniqueness of the limits. Moreover,  these tools allow for a well-defined injectivity radius, uniform curvature norms and the application of the Arzel\`a--Ascoli theorem which are the ingredients for the existence of the limits (see the estimates in Lemma \ref{lemma:uniform_bounds_metrics_derivatives} and the convergence result in Theorem \ref{thm:compactness_theorem_riemann}).

In the general semi-Riemannian case, the anchors will make it possible to construct canonical auxiliary Riemannian metrics around basepoints and, then, uniqueness up to isometries for the limits around the basepoints (Theorem \ref{teor:uniqueness_semi-riemannian}). 

Stronger results for uniqueness and existence of limits will be obtained for globally hyperbolic spacetimes in \S \ref{sec:convergence_globally_hyperbolic} by using the stronger anchoring of an $h$-steep temporal function. Indeed, by Theorem \ref{intro:thm_A} this will be a Riemannian issue for Wick-rotated metrics.  

\smallskip 
\smallskip
\noindent {\em 3. Hopf-Rinow equivalences}. Riemannian completeness implies the equivalence of several global properties which are frequently used implicitly.

In the semi-Riemannian setting, these properties are not equivalent and one or more of them should be imposed, eventually complemented with others. We will give a general semi-Riemannian result of uniqueness of limits (Theorem \ref{coro_simplyconnected}) by imposing geodesic completeness for only one causal type of geodesics on one of the limits, together with inextensibility on the other one; Example \ref{ex:misner_double_limit} shows that the latter condition cannot replace the former.  
Other global hypotheses (for example, a controlled geodesic connectivity from the basepoint) may lead to different strengthenings of global  uniqueness results. In any case, the Hopf--Rinow equivalences will be available for the Wick-rotated metrics in the globally hyperbolic setting.


\subsubsection{Redefining Cheeger-Gromov semi-Riemannian convergence} \label{subsec:5_2_2}

\begin{definition}\label{def:marked_manifolds}
	An {\em anchored semi-Riemannian manifold $(M,g,\mathcal{B}_p)$} is a (connected) semi-Riemannian manifold $(M,g)$ of dimension $n+\nu$ and index $\nu$ with a distinguished basepoint $p \in M$, and a subspace $
    \mathcal{B}_p\subset T_p M$ of 
    dimension $\nu$ on which $g$ is negative definite. The {\em anchor} $\mathcal{B}_p$  determines an orthogonal decomposition $T_p M = \mathcal{B}_p \oplus \mathcal{B}_p^\perp$, where $g_p$ is negative definite on $\mathcal{B}_p$ and positive definite on $\mathcal{B}_p^\perp$. An orthonormal basis
$B_p$ of $T_pM$ is {\em adapted to the anchor} if its $\nu$ first vectors lie in $\mathcal{B}_p$ (and, thus, the others in $\mathcal{B}^\perp_p$); we will write then $B_p\in \B_p$ ( abusing notation).
\end{definition}
In particular, an anchored semi-Riemannian manifold is a pointed one (in the Riemannian setting both notions are equivalent)
 and Definition \ref{def:riemann_convergence} naturally extends as follows.

\begin{definition}\label{def:convergence_semi-riemann} 
	A sequence of anchored semi-Riemannian manifolds $\{(M_i, g_i, \mathcal{B}_{p_i})\}$ is said to  {\em (Cheeger-Gromov) $C^k$ converge} to an anchored semi-Riemannian manifold $(M, g,  \mathcal{B}_{p})$, for $k=0,1,2, \dots, \infty$, if
	\begin{enumerate}[label=(\roman*)]	
            \item\label{def:semiriem_convergence_2} the pointed manifolds $(M_i, g_i, p_i)$ converge in $C^k$ to $(M,g,p)$ in the sense of Definition \ref{def:riemann_convergence}, and
		\item\label{def:semiriem_convergence_3} $\{(d\phi_i ^{-1})_{p_i} (\mathcal{B}_{p_i})\}$ converges to $\mathcal{B}_p$ on $T_p M$.
	\end{enumerate}
\end{definition}

\begin{remark}
	In the Lorentzian case, condition (ii) 
    is equivalent to fixing a unit timelike vector $v_i\in T_{p_i} M_i$ such that $\{(d\phi_i^{-1})_{p_i} (v_i)\}$ converges to a unit timelike vector $v\in T_p M$. 
    In the semi-Riemannian setting, 
	the condition amounts to requiring directly that a negative-definite subspace of $T_{p_i} M_i$ converges, via $(d\phi_i ^{-1})_{p_i}$, to a fixed negative-definite subspace of $T_p M$. However, the following result will show an essential equivalence with the convergence of adapted orthonormal bases.
	\end{remark}
    
    \begin{proposition}
        \label{prop:compacidad_adaptedframe} Assume that anchored convergence holds as in Def. \ref{def:convergence_semi-riemann} and let $B_i$ be an adapted orthonormal basis for each $\mathcal{B}_{p_i}$. Then, there exists an adapted orthonormal basis $B\in \mathcal{B}_{p}$ such that, up to a subsequence, $
    (d\phi_i ^{-1})_{p_i}(B_i)\longrightarrow B
    $,
    in the frame bundle of $M$. 
    \end{proposition}
	\begin{proof} 
Choose coordinates around $p$ in $M$ yielding an orthonormal and adapted basis $B_0$ at $p$, and let $A_i$ be the matrix of $d\phi_i ^{-1}|_{p_i}$ expressed in the bases $B_i$ and $B_0$.
As the hypothesis of convergence \ref{def:semiriem_convergence_2} implies $\{\phi_i ^{-1}(p_i)\}\rightarrow p$, the problem reduces to checking that, up to a subsequence, $\{A_i\}$ converges to a matrix $A\in O(\nu )\times O(n)$.
Since  $\phi_i^{*} g_i$ converges to $g$ in $C^{0}$, this map is close to a linear isometry, that is, for each $\epsilon>0$ 
	 
\begin{equation*}
			|A_i^T \eta A_i - \eta| < \epsilon \quad\hbox{on $T_p M$}, \qquad \text{where} \; \eta = \text{diag}(-1,...,-1,1,...,1),
		\end{equation*}
for $i$ large. Then, using  that all $B_i$ are adapted and the convergence of the anchors, 

\begin{equation*}\label{eq:matrices_Ai} 
			A_i = \left(\begin{matrix}
				A_{11}^i & E_{12}^i \\
				E_{21}^i & A_{22}^i
			\end{matrix}\right)
		\end{equation*}
(for larger $i$), 
		where $A_{11}^i \in O_\epsilon (\nu), A_{22}^i \in O_\epsilon (n)$, ($O_\epsilon$  denotes  a $\epsilon$-small  compact neighborhood in the Euclidean space of the square matrices for the corresponding orthogonal group) and the submatrices $E_{12}^i$ and $E_{21}^i$ have entries in the compact interval $[-\epsilon, \epsilon]$. 
		Thus, converging subsequences of $\{A_i\}$ must exist, and the limit will be of the required form 
        $	A = \left(\begin{matrix}
				A_{11} & 0 \\
				0 & A_{22}
			\end{matrix}\right)\in
            O(
            \nu) \times O (n)
            $, thus providing the transition matrix from $B_0$ to $B$.
	\end{proof}

\begin{remark}\label{rema:compacidad_adaptedframe} Notice that the previous proof does not require full Cheeger-Gromov convergence on $M$ but just convergence in a compact neighborhood of $p$.     
\end{remark}
    
\subsection{Uniqueness of Cheeger-Gromov limits }\label{subsec:5_2_3}
	The uniqueness of the limits is essential for the applications of Cheeger-Gromov convergence. The following general result shows how anchoring yields it for at least $C^2$-convergence locally, without global alternatives to Riemannian completeness such as geodesic completeness, geodesic connectedness or compactness. 
    
\begin{theorem}\label{teor:uniqueness_semi-riemannian}
	Let $\{(M_i,g_i,\mathcal{B}_{p_i})\}$ be a sequence of anchored semi-Riemannian manifolds $C^2$-converging to two limits   $(M, g, \mathcal{B}_p)$, $(\tilde M, \tilde g, \mathcal{B}_{\tilde p})$
       (in the sense of Def. \ref{def:convergence_semi-riemann}). Then, there exists an isometry between neighborhoods of $p$ and $\tilde p$, $F: W_p\longrightarrow \tilde W_{\tilde p}$. 
       
       If the $\tilde g$-geodesics starting at $\tilde p$ are complete, then $F$ can be extended to an isometry from any (starshaped) $g$-normal neighborhood $\C$  of $p$ which includes $W_p$.
       
\end{theorem}

As a preparation, we sketch the idea of the proof and give some definitions and lemmas illustrating the difficulties specific to the semi-Riemannian case.

Let $\{\phi_i\}, \{\tilde \phi_i\}$ be  the diffeomorphisms associated with the convergence with domains in $M$ and $\tilde M$, respectively. 
    Let $U^i:=\phi_i^{-1}(\phi_i(U_i)\cap \tilde \phi_i(\tilde U_i))$, so that $\psi_i:=\tilde\phi_i^{-1}\circ\phi_i$ is an injective isometry from $(U^i,\phi_i^*g_i)$ onto an open subset of $(\tilde U_i,\tilde\phi_i^*g_i)$. Notice that  each $\phi^{-1}_i(p_i)\in U^i$ but, a priori, $p$ might never belong.  
    
    In the Riemannian case, pointed convergence suffices to ensure that  
    $\{(\tilde \phi_i^{-1}\circ\phi_i)^* \tilde g\}$ 
    converges uniformly on compact subsets, as well as $p\in U^i$ for large $i$ and, then, $\{\tilde \phi_i^{-1}\circ\phi_i(p)\}\rightarrow\tilde p$. 
    This relies on the fact that the maps $\tilde \phi_i^{-1}\circ\phi_i$ are, for some $\lambda > 1$, $\lambda$ quasi-isometries for the distances associated with  $g$ and $\tilde g$ \footnote{See for example \cite[Prop 4.2 (ii)]{Chow2007ricciflow}.}, that is, they satisfy as quadratic forms
\begin{equation*}
\frac{1}{\lambda}g \leq     (\tilde \phi_i^{-1}\circ\phi_i)^* {\tilde g} \leq \lambda g,
\end{equation*}
 and the inequalities are preserved by the associated distances. 

For the general semi-Riemannian case, the anchors will be used to find local Riemannian metrics around $p$ and $\tilde p$ for which $\{(\tilde \phi_i^{-1}\circ\phi_i)\}$ are quasi-isometries. 
These metrics depend on the choice of a representative $B$ adapted to each anchor $\B$, and the following convention will be used. 

\begin{convention}\label{conv} In what follows, we will choose any adapted representative $B_{p_i} \in \B_{p_i}$ and will assume  the convergences (up to a subsequence)  
$$(d\phi_i ^{-1})_{p_i}(B_{p_i})\longrightarrow B_p, \qquad (d\tilde \phi_i ^{-1})_{p_i}(B_{p_i})\longrightarrow \tilde B_{\tilde p},$$ 
for suitable representatives  $B_p\in \B_p$ and $\tilde B_{\tilde  p}\in \B_{\tilde p}$ and without loss of generality by Prop. \ref{prop:compacidad_adaptedframe}.
\end{convention}
These anchored metrics will allow us to obtain uniform convergence of $\{(\tilde \phi_i^{-1}\circ\phi_i)^* \tilde g\}$ on small 
compact neighborhoods of $p, \tilde p$ (Lemmas \ref{lema1}, \ref{lema2} and \ref{lema3}) and, then, to find the required isometries. 

\begin{definition}\label{def:anchor}
Let     $(M, g, \mathcal{B}_p)$ be an anchored semi-Riemannian manifold and let $g^A_p$ be the Euclidean  scalar product at $T_pM$ obtained reversing the sign of $\B$, i.e, $g^A_p=-g_p|_{\B} \oplus g_p|_{\B^\perp}$ and fix an orthonormal basis $B \in \B$ (which will be implicitly chosen using Convention \ref{conv}).

The {\em ($B$-)anchored metric} $g^A$  of radius $r_0>0$ 
is the Riemannian metric constructed in a compact 
normal neighborhood $\B^A(r_0)$ of $p$  as follows.
For each $g^A_p$-unit vector $u$ consider the $g$-geodesic $\gamma_u$ with initial velocity $\dot \gamma_u(0)=u$ and take $g^A_{\gamma_u(r)}, 0<r\leq r_0$ by $g$-parallelly propagating the orthonormal basis $B$ of $g^A_p$ along $
\gamma_u$ (here, $r_0$ is taken small enough so that  $\B^A(r_0)$ is normal).  

The natural $g^A$ distance will be denoted $d^A$ and the norms (on tensors or tensor fields on compact sets) $|\cdot |_{g^A}$.   
\end{definition}

\begin{lemma}\label{lema1}
    Assume that the anchored  metric $g^A$ of the limit  $(M, g, \mathcal{B}_p)$ is well  defined in $\B^A(4r_0)$.
    Then, for each $\lambda >1$ there exists $i_\lambda$ such that for $i>i_\lambda$: 
    
    (a) the anchored metric $g^A_i$ of each element $(M_i,g_i,\mathcal{B}_{p_i})$ in the  
    converging sequence is well-defined  on $\B_i ^A(2r_0)$ and 
    
    (b) the restrictions of the maps $\phi_i$ 
    to $\B^A(r_0)$ are well defined and become $\lambda$ quasi-isometries between $g^A$ and each $g^A_i$.
\end{lemma}

\begin{proof} The $C^2$ uniform convergence of 
$\{\phi_i^*g_i\}$ to $g$ on the compact set  $\B^A(2r_0)$
implies the $C^1$-convergence $\{\phi_i^*\nabla^i\} \longrightarrow \nabla$ of the corresponding Levi-Civita connections. By 
    hypothesis, $\{(d\phi_i ^{-1})_{p_i} (\mathcal{B}_{p_i})\} \longrightarrow \mathcal{B}_p$, thus  
    $g^A_{p_i}\longrightarrow g^A$, and the $C^1$-dependence of the parallel transport  on the initial condition implies the uniform convergence of $\{\phi_i^*g^A_i\}$ to $g^A$. Then, all the assertions are straightforward.
\end{proof}

\begin{lemma} \label{lema2} Given the two limits, $(M, g, \B_p)$ and $(\tilde M, \tilde g, \B_{\tilde p})$, take $r_0>0$  such that  the anchored metrics $g^A$, $\tilde g^A$   with radius $r_0$ and $4r_0$ resp. are well defined. Then:

(a) For sufficiently large $i$,
    the maps 
     $\tilde \phi_i^{-1}\circ\phi_i: \B^A(r_0)\longrightarrow \tilde \B^A(4r_0)$ are well-defined.      
     
     (b) For each $\lambda >1$ there exists $i_\lambda \in \N$ such that $\tilde \phi_i^{-1}\circ\phi_i$ becomes a $\lambda$ quasi-isometry for $g^A$ and 
     $\tilde g^A$ when $i>i_\lambda$.

(c) $\lim_{i\rightarrow\infty}(\tilde \phi_i^{-1}\circ \phi_i)(p)=\tilde p$.
\end{lemma}

\begin{proof} (a) Apply  Lemma \ref{lema1} twice choosing  
$\lambda=2$ (so that $\phi_i(\B^A(r_0))\subset \tilde \phi_i(\tilde \B^A(4r_0))$). 

(b) The composition of  two $\lambda$ quasi-isometries is a  $\lambda^2$ quasi-isometry. 

(c)  $\{\phi_i^{-1}(p_i)\}$ eventually remains in 
    $\B^A(r_0)$ and, therein, $\tilde \phi_i^{-1}\circ \phi_i$ can be regarded as a $\lambda$ 
    quasi-isometry. Thus, the result follows as $\phi_i ^{-1} (p_i) \to p$, $\tilde \phi_i^{-1}(p_i)\rightarrow \tilde p$ and
    $\tilde d^A(\tilde \phi_i^{-1}(p_i), (\tilde \phi_i^{-1}\circ \phi_i)(p)) \leq
    \lambda^2 d^A(\phi_i^{-1}(p_i), p) \longrightarrow 0.
     $
\end{proof}

The following step  emphasizes an essential semi-Riemannian property yielded by the anchoring (which is straightforward in the Riemannian case).
\begin{lemma}[Anchoring initiates uniform convergence] \label{lema3} 
    For $\B^A(r_0)$ as in Lemma \ref{lema2}, and any choice of basis  adapted to the anchoring $B_p \in \B_p$:
    \begin{equation*}
        \begin{array}{cc}
         (\tilde \phi_i^{-1}\circ\phi_i)^*(\tilde g)\longrightarrow g     &  C^2-\hbox{uniformly on}
    \; \B^A(r_0)   \\
    d(\tilde \phi_i ^{-1}\circ \phi_i)_{p}(B_p)\longrightarrow  B_{\tilde p}, 
             & \hbox{up to a subsequence},
        \end{array}
    \end{equation*}
    for some $B_{\tilde p}\in  \B_{\tilde p}$.
\end{lemma}

\begin{proof}  Using that all $\tilde \phi_i^{-1}\circ\phi_i$ (and their inverses) are quasi-isometries for the same $\lambda$, 
$$  \begin{array}{rl}
     |(\tilde \phi_i^{-1}\circ\phi_i)^*(\tilde g)-g|_{g^A} &\leq \left|(\tilde \phi_i^{-1}\circ\phi_i)^*(\tilde g-\tilde \phi_i^{*}g_i)\right|_{g^A} 
     +\left|\phi_i^*g_i-g\right|_{g^A} 
     \\
     & \leq  \lambda^2  \left|\tilde g-\tilde \phi_i^{*}g_i\right|_{\tilde g^{A}}
     +\left|\phi_i^*g_i-g\right|_{g^A},   
\end{array}    $$
which $C^2$-converges to $0$ (because, by hypothesis $\{\phi_i^*g_i\}\rightarrow g$ and $\tilde \phi_i^{*}g_i\rightarrow \tilde g$, $C^2$-uniformly on the compact sets  $\B^A(r_0)$, $\tilde \B^A(4r_0)$, resp.), 
yielding the first convergence  (the bound $\lambda^2$ holds for the $C^0$ norms; for the $C^2$ norms one uses that $\tilde\phi_i^{-1}\circ\phi_i$ is an isometry between $\phi_i^*g_i$ and $\tilde\phi_i^*g_i$, so that its second derivatives are determined by the first ones and by the Christoffel symbols of these metrics, which converge in $C^1$). The second one follows from Lemma \ref{lema2}(c) and Prop. \ref{prop:compacidad_adaptedframe} with Remark \ref{rema:compacidad_adaptedframe}. 
\end{proof}

This result on convergence allows us to obtain the required isometries between the original semi-Riemannian metrics $g, \tilde g$ close to $p, \tilde p$.
 %

    

\begin{proof}[Proof of Theorem \ref{teor:uniqueness_semi-riemannian}] For the first assertion,  construct an isometry $$F: (\B^A(r_0), g)\longrightarrow (\tilde \B^A(r_0), \tilde g)$$ 
as follows.  
Choose $B_p\in \B_p$, take the basis $B_{\tilde p}$ provided by Lemma \ref{lema3}, and let $L$ 
be the linear isometry $T_pM \longrightarrow T_{\tilde p}\tilde M$ satisfying $L (B_p)=B_{\tilde p}$. 
Now, define:
$$
F: \B^A(r_0)\longrightarrow \tilde \B^A(4r_0), \qquad 
F=   \exp_{\tilde p} \circ L \circ \exp_p ^{-1},
$$   
This map, as well as all $\tilde \phi_i ^{-1}\circ \phi_i$ are defined on the whole $\B^A(r_0)$ because their domains and codomains depend on the exponential map as in Lemma \ref{lema2}.
Now, the $C^2$ uniform convergence of 
$\{(\tilde \phi_i ^{-1}\circ \phi_i)^*(\tilde g)\}$ to $g$ stated in Lemma \ref{lema3} permits a similar reasoning as in Lemma \ref{lema1}. Namely, first it implies  the convergence  of the corresponding Levi-Civita connections  $\{(\tilde \phi_i ^{-1}\circ \phi_i)^*(\tilde \nabla)\}$  to $\nabla$ and, then,  the  $C^1$-dependence of the parallel transport  on the initial condition makes $F^*(\tilde g)$ converge uniformly to $g$  (see Remark \ref{r_puntilla} for a more precise argument).\footnote{\label{foot:Ambrose}Notice that the required convergence at this step as well as in Lemma \ref{lema1}, is $C^2$ because this implies that the Christoffel symbols (as coefficients in the ODE system for parallel transport) are $C^1$.  This agrees with the regularity required for the classical  Cartan--Ambrose--Hicks theorem, where a map similar to $F$ is constructed and the existence of an isometry is obtained under the preservation of the radial curvature  (notice that this theorem can be extended to the affine case \cite[Theorem 4]{NY} and, then, to the semi-Riemannian one).}
Finally, the fact that $F$ is an isometry also implies $F (\B^A(r_0)) = \tilde \B^A(r_0)$ by the construction of these convex neighborhoods.\footnote{We use \textit{convex neighborhoods} in the sense of O'Neill \cite[Definition 5.5]{Oneill}, that is, an open set which is a (star-shaped) normal neighborhood of all its points.} 

For the last assertion, notice first  that the completeness of  the geodesics starting at $\tilde p$  
makes it possible to extend $F$ as a local isometry on any normal neighborhood  of\footnote{\label{foot_fin}  The completeness of the geodesics from $\tilde p$ is used only for this purpose. Thus, given $\C$ it is enough that the domains of the corresponding geodesics from $\tilde p$ permit to define $F$.} $p$.  Let us consider any direction in the tangent space 
$u\in T_pM$ and the corresponding radial segment 
$\{\exp_p(ru):  0\leq r < r_u\}\subset \C$. Assume that  there is a maximum $r_1<r_u$ such that $F$ is not a local isometry beyond $r_1$.  Notice that, as  $F$ is smooth, it  must preserve the metric also at $p_0:=\exp_p(r_1 u)$. In particular, if $B_0$ is the basis at $T_{p_0}M$ obtained by parallelly propagating along the segment the anchor  $B_p$, then $\tilde B_0:= (dF)_{p_0}B_0$ is a basis at a point $\tilde p_0\in \tilde M$. Thus, we can apply the previous part  for convergence to the two limits with the new anchors. This will provide a local isometry between neighborhoods of  $p_0$ and $\tilde p_0$ extending $F$, a contradiction. We also claim that this local isometry is injective (thus, global) by the argument stressed in the next remark.
    \end{proof}

    \begin{remark} \label{r_puntilla}
        By construction, $\{\tilde \phi_i^{-1}\circ\phi_i\}$ converges uniformly to $F$ on compact sets of its domain  $\C$; in fact, the convergence is $C^1$: as $\psi_i=\tilde\phi_i^{-1}\circ\phi_i$ is an isometry between $\phi_i^*g_i$ and $\tilde\phi_i^*g_i$, it maps radial geodesics to radial geodesics, $\psi_i=\exp^{\tilde \phi_i^*g_i}_{\psi_i(p_i')}\circ (d\psi_i)_{p_i'}\circ (\exp^{\phi_i^*g_i}_{p_i'})^{-1}$ with $p_i':=\phi_i^{-1}(p_i)$, where $\psi_i(p_i')\rightarrow \tilde p$ and $(d\psi_i)_{p_i'}\rightarrow L$ by Lemma \ref{lema3} and the exponential maps converge in $C^1$ on compact sets because the metrics converge in $C^2$. In particular, this implies the injectivity of $F$ claimed at the end of the previous proof. Indeed, otherwise there would exist points $q,q'\in \C$ satisfying $F(q)=F(q')=: \tilde q$. 
        Let $W_q, W_{q'} \subset \C$ be open disjoint neighborhoods of     $q, q'$, resp. Then, $F(W_q)$ is a neighborhood of $\tilde q$ and, by assumption $\tilde q=F(q')=\lim_i (\tilde \phi_i^{-1}\circ\phi_i) (q')$. Thus,  $(\tilde \phi_i^{-1}\circ\phi_i) (q')$ goes into $F(W_q)$ (eventually, in $(\tilde \phi_i^{-1}\circ\phi_i) (W_q)$), in contradiction with the injectivity of  $\tilde \phi_i^{-1}\circ\phi_i$.    
    \end{remark}

Next, this semi-local result is used to obtain  the global one stated in the Introduction  (Theorem  \ref{coro_simplyconnected}). Notice that a connected semi-Riemannian metric is called {\em inextensible} when it does not admit an isometric embedding as an open subset strictly included in another connected semi-Riemannian manifold. 

\begin{remark}\label{r_inext}
This notion  is standard, 
and it is developed in
Beem et al. book \cite[\S 6]{BEE}  in the indefinite case, where some subtleties appear. Among 
them, notice that {\em geodesic completeness, even only for geodesics on one of the three types (timelike, lightlike or spacelike), implies inextensibility} \cite[Prop. 6.16]{BEE}. 

It is worth pointing out  that we use (global) inextensibility, which is less restrictive than the local one defined in  \cite[Definition 6.17]{BEE}. This happens even in the Riemannian case. For example, 
consider a cone (up to its vertex 0), constructed from a  circular section of the plane of angle $\theta\in]0,\pi[$ by identifying naturally its two sides. Clearly, it  is not globally extensible to 0 (the length of the circles of radius $r$ centered at $0$ is $\theta r$ rather than $2\pi r$) but it is locally extensible. In fact, an open circular  subsection of angle $\theta/2$ is  isometric to an open region of $\R^2$ and, thus,  extensible to $\R^2$, (compare with 
\cite[Example 6.21]{BEE}).

Theorem~\ref{coro_simplyconnected} combines both notions: completeness for one causal type of geodesics in one of the limits and inextensibility in the other one. Example \ref{ex:misner_double_limit} below shows that inextensibility of the two limits is not enough. 
\end{remark}

The proof of Theorem \ref{coro_simplyconnected} relies on the following description of the maximal domain of the local isometry $F$ provided by Theorem \ref{teor:uniqueness_semi-riemannian}. Recall that $\psi_i=\tilde \phi_i^{-1}\circ \phi_i$ is an injective isometry from $(U^i,\phi_i^*g_i)$ onto an open subset of $(\tilde U_i,\tilde \phi_i^*g_i)$ and that, by Definition \ref{def:riemann_convergence}, every compact subset of $M$ (resp. $\tilde M$) is contained in $U_i$ (resp. $\tilde U_i$) for $i$ large. Let us introduce a working definition. 

\begin{definition}\label{def:maximal_domain}
Let $\mathcal{D}$ be the set of points $x\in M$ admitting an open neighborhood $N$ and a local isometry $F_x:N\rightarrow \tilde M$ such that $N\subset U^i$ for $i$ large and $\psi_i\rightarrow F_x$ in $C^1$ on compact subsets of $N$. Clearly, $\mathcal{D}$ is open and the local isometries $F_x$ agree on overlaps (all of them are limits of the same sequence), so that they define a local isometry $F:\mathcal{D}\rightarrow \tilde M$.

Let $D$ be the connected component of $\mathcal{D}$ containing $p$, which is nonempty by Theorem~\ref{teor:uniqueness_semi-riemannian} and Remark \ref{r_puntilla}.
\end{definition}

\begin{lemma}\label{lem:inj_D}
$F$ is injective on $D$; thus, $F$ is an isometry from $(D,g)$ onto the open subset $F(D)$ of $\tilde M$.
\end{lemma}
\begin{proof} Reasoning as in Remark \ref{r_puntilla}, 
assume by contradiction that there are $y\neq y'$ in $D$ with $F(y)=F(y')$ and disjoint open neighborhoods $W,W'$ of $y,y'$ with compact closure in $D$ on which $F$ is a diffeomorphism onto its image. Since $\psi_i\rightarrow F$ in $C^1(\overline W)$, for $i$ large $\psi_i(W)$ contains a fixed neighborhood of $F(y)$ (by a standard degree argument), and $\psi_i(y')\rightarrow F(y')=F(y)$. So, $\psi_i(y')\in \psi_i(W)$ with $y'\notin W$, in contradiction with the injectivity of $\psi_i$.
\end{proof}

\begin{lemma}[Extension along geodesics]\label{lem:ext_geod}
Let $x\in D$ and $u\in T_xM$ be such that $\gamma(s):=\exp_x(su)$ is defined for $s\in [0,1]$ and the $\tilde g$-geodesic $\tilde\gamma(s):=\exp_{F(x)}(s\, dF_x(u))$ is also defined for $s\in[0,1]$. Then $\gamma([0,1])\subset D$ and $F\circ\gamma=\tilde \gamma$.
\end{lemma}
\begin{proof}
Let $S:=\{s\in [0,1]: \gamma([0,s])\subset D\}$, a nonempty interval, open in $[0,1]$ because $D$ is open. On $S$, $F\circ \gamma$ is a $\tilde g$-geodesic with the same initial data as $\tilde\gamma$, so $F\circ\gamma=\tilde\gamma$ there. Assume by contradiction $S=[0,s^*)$ with $s^*\leq 1$ and $\gamma(s^*)\notin D$. We will show that $\gamma(s^*)\in \mathcal{D}$; this is contradictory because $\gamma(s^*)$ lies in the closure of the connected component $D$ of $\mathcal{D}$.

Fix compact neighborhoods $K$ of $\gamma([0,1])$ in $M$ and $\tilde K$ of $\tilde\gamma([0,1])$ in $\tilde M$; for $i$ large, $K\subset U_i$ and $\tilde K\subset \tilde U_i$. Let $\gamma^{(i)}$ be the $\phi_i^*g_i$-geodesic with initial data $(x,u)$: by continuous dependence on the metric, it is defined on $[0,1]$ for $i$ large and $\gamma^{(i)}\rightarrow \gamma$ in $C^1([0,1])$. As $x\in D$, $\psi_i(x)\rightarrow F(x)$ and $(d\psi_i)_x(u)\rightarrow dF_x(u)$, so the $\tilde\phi_i^*g_i$-geodesic $\tilde\gamma^{(i)}$ with initial data $(\psi_i(x),(d\psi_i)_x u)$ is defined on $[0,1]$ for $i$ large, it lies in $\tilde K$ and $\tilde\gamma^{(i)}\rightarrow\tilde\gamma$ in $C^1([0,1])$. Now, $\phi_i\circ\gamma^{(i)}$ and $\tilde\phi_i\circ\tilde\gamma^{(i)}$ are $g_i$-geodesics of $M_i$ with the same initial data $(\phi_i(x), d\phi_i(u))$. Hence, they coincide, that is,
\begin{equation}\label{eq:same_geodesic}
\gamma^{(i)}([0,1])\subset U^i \qquad \hbox{and} \qquad \psi_i\circ \gamma^{(i)}=\tilde\gamma^{(i)} \quad \hbox{on } [0,1].
\end{equation}
Choose an anchor $\mathcal{B}_x\subset T_xM$ and let $\mathcal{B}^{(i)}_s$, $\mathcal{B}_s$ and $\tilde{\mathcal{B}}_s$ be the parallel transports of $\mathcal{B}_x$ along $\gamma^{(i)}$ (for $\phi_i^*g_i$), of $\mathcal{B}_x$ along $\gamma$ (for $g$) and of $dF_x(\mathcal{B}_x)$ along $\tilde\gamma$ (for $\tilde g$), respectively. Consider the new basepoints and anchors
$$q_i:=\phi_i(\gamma^{(i)}(s^*)), \qquad \mathcal{B}_{q_i}:=d\phi_i(\mathcal{B}^{(i)}_{s^*}).$$
Then, $\phi_i^{-1}(q_i)=\gamma^{(i)}(s^*)\rightarrow \gamma(s^*)$ and $d\phi_i^{-1}(\mathcal{B}_{q_i})=\mathcal{B}^{(i)}_{s^*}\rightarrow \mathcal{B}_{s^*}$. On the other side, by \eqref{eq:same_geodesic}, $\tilde\phi_i^{-1}(q_i)=\tilde\gamma^{(i)}(s^*)\rightarrow \tilde\gamma(s^*)$, and $d\tilde\phi_i^{-1}(\mathcal{B}_{q_i})=d\psi_i(\mathcal{B}^{(i)}_{s^*})$ is the $\tilde\phi_i^*g_i$-parallel transport of $(d\psi_i)_x(\mathcal{B}_x)$ along $\tilde\gamma^{(i)}$ (because $\psi_i$ is an isometry for the $i$-th metrics), which converges to $\tilde{\mathcal{B}}_{s^*}$ by the continuous dependence of the parallel transport on the metric, the curve and the initial data. Summing up, the sequence converges, through the same embeddings $\phi_i,\tilde\phi_i$, to $(M,g,\gamma(s^*),\mathcal{B}_{s^*})$ and to $(\tilde M,\tilde g,\tilde\gamma(s^*),\tilde{\mathcal{B}}_{s^*})$, and Theorem \ref{teor:uniqueness_semi-riemannian} together with Remark \ref{r_puntilla} yield $\gamma(s^*)\in\mathcal{D}$.
\end{proof}

\begin{remark}
Lemma \ref{lem:ext_geod} is the precise form of the last assertion of Theorem \ref{teor:uniqueness_semi-riemannian} and footnote \ref{foot_fin}. Notice that the new basepoints $q_i$ must be taken on the $g_i$-geodesics $\phi_i\circ\gamma^{(i)}$ (rather than as $\phi_i(\gamma(s^*))$), which is what makes \eqref{eq:same_geodesic} available.
\end{remark}
The next lemma will apply specifically to the case of lightlike completeness in Theorem~\ref{coro_simplyconnected}.
\begin{lemma}[Broken lightlike geodesics]\label{lem:broken_null}
Let $V$ be a convex neighborhood in a semi-Riemannian manifold of dimension $n\geq 2$ and index $\nu$, $0<\nu<n$, and let $r\in V$. Then there exists a neighborhood $V'\subset V$ of $r$ such that every $z\in V'\setminus\{r\}$ can be joined to $r$ by a broken lightlike geodesic contained in $V$ with at most three segments; that is, there are $z=w_0, w_1,\dots ,w_m=r$, $m\leq 3$, with $E(w_j,w_{j+1}):=\exp_{w_j}^{-1}(w_{j+1})$ lightlike for all $j$.
\end{lemma}
\begin{proof}
Take normal coordinates $\xi$ centered at $r$ on $V$, let $|\cdot|$ be the Euclidean norm and $\eta$ the flat metric $g_r$ in these coordinates. Then $g=\eta+O(|\xi|)$ and the smooth map $E$ satisfies $E(a,b)=b-a+O(|b-a|^2)$ (as $\partial_bE(a,a)$ is the identity), so that
\begin{equation}\label{eq:approx_null}
g\big(E(a,b),E(a,b)\big)=\eta(b-a,b-a)+O\big((|a|+|b-a|)\,|b-a|^2\big),
\end{equation}
and $\xi(\exp_a(v))=\xi(a)+v+O(|a|\,|v|^2)$ (these are just the first-order Taylor expansions of the smooth maps $E$ and $\exp$ in normal coordinates; the same kind of expansion is used in Step 3 of the proof of Theorem \ref{prop:uniqueness_observed_spacetimes}). Since the $\eta$-lightlike vectors span $\R^n$, for every unit $\zeta$ there is an $\eta$-lightlike unit vector $n$ with $\eta(n,\zeta)<0$ and, by compactness, there is $c>0$ such that $\eta(n,\zeta)\leq -c$ can always be achieved. Put $d:=\xi(z)$ and $\lambda:=|d|$; all the $O(\cdot)$ below are uniform in $z$ for $\lambda$ small.

{\em Step 1 (a first lightlike leg making the displacement uniformly timelike).} If $\eta(d,d)\leq -\lambda^2/2$, put $w_1:=z$. Otherwise choose $n$ as above for $\zeta=d/\lambda$, let $n^g$ be a $g$-lightlike vector at $z$ with $|n^g-n|=O(\lambda)$ (the lightlike cones of $g_z$ and $\eta$ are $O(\lambda)$-close), and put $w_1:=\exp_z(\frac{\lambda}{c}\, n^g)$. Then $\xi(w_1)=d+\frac{\lambda}{c}n+O(\lambda^2)$, whence $\eta(\xi(w_1),\xi(w_1))\leq \lambda^2-2\lambda^2+O(\lambda^3)$ and $|\xi(w_1)|\leq (1+\frac1c)\lambda$. In both cases, by \eqref{eq:approx_null}, $\mu:=|\xi(w_1)|$ satisfies $\mu\leq (1+\frac1c)\lambda$ and $g(E(w_1,r),E(w_1,r))\leq -\kappa \mu^2$ for some $\kappa>0$ independent of $z$; in particular, $w_1$ lies on no lightlike geodesic through $r$.

{\em Step 2 (two lightlike legs from $w_1$ to $r$).} Choose an $\eta$-lightlike unit $m$ with $\eta(m,\xi(w_1))\leq -c\mu$, let $m^g$ be a $g$-lightlike vector at $r$ with $|m^g-m|=O(\lambda)$, and consider the lightlike geodesic $\sigma(s):=\exp_r(s\, m^g)$ and the function $f(s):=g(E(w_1,\sigma(s)),E(w_1,\sigma(s)))$, $s\in [0,\mu/c]$. Then $f(0)\leq -\kappa\mu^2<0$ while, by \eqref{eq:approx_null} and $\xi(\sigma(s))=sm+O(\lambda s)$,
$$
f(\mu/c)=\eta\big(\tfrac{\mu}{c}m-\xi(w_1),\tfrac{\mu}{c}m-\xi(w_1)\big)+O(\lambda^3)\geq 2\mu^2-\mu^2+O(\lambda^3)>0
$$
for $\lambda$ small. By the intermediate value theorem there is $s^*\in (0,\mu/c)$ with $f(s^*)=0$. Put $w_2:=\sigma(s^*)$: $E(w_1,w_2)$ is lightlike and nonzero (as $w_2\neq w_1$), and $E(w_2,r)$ is lightlike because the geodesic of $V$ joining $w_2$ and $r$ is $\sigma$. All the segments lie in $V$ by convexity, and $V'$ can be taken as a small ball around $r$.
\end{proof}

\begin{proof}[Proof of Theorem \ref{coro_simplyconnected}]
Interchanging the roles of the two limits if necessary, assume that $\tilde M$ is geodesically complete for one causal type of geodesics and that $M$ is inextensible. Let $D$ and $F$ be as in Definition \ref{def:maximal_domain}. We claim that $D=M$. Otherwise, $M$ being connected, there is some $r_\infty\in \partial D$; let $V$ be a convex neighborhood of $r_\infty$ and $z_k\in D\cap V$ with $z_k\rightarrow r_\infty$. By Lemma \ref{lem:ext_geod}: 
{\em if $x\in D$, $y\in V$ and $E(x,y)$ has the causal type for which $\tilde M$ is complete, then $y\in D$}, because $\tilde\gamma(s)=\exp_{F(x)}(s\, dF_x(E(x,y)))$ is then a complete geodesic (recall that $dF_x$ preserves the causal characters). 

{\em Timelike (resp. spacelike) completeness.} Pick $y\in V$ with $E(r_\infty,y)$ timelike (resp. spacelike). As $E$ is continuous on $V\times V$, $E(z_k,y)$ is timelike (resp. spacelike) for $k$ large, so $y\in D$. Now, $E(y,r_\infty)$ is timelike (resp. spacelike), being minus the final velocity of the geodesic of $V$ from $r_\infty$ to $y$; hence, $r_\infty\in D$, a contradiction.

{\em Lightlike completeness.} By Lemma \ref{lem:broken_null}, for $k$ large there is a broken lightlike geodesic $z_k=w_0,w_1,\dots ,w_m=r_\infty$ in $V$ with $m\leq 3$. Applying the claim successively, $w_1\in D$, $w_2\in D$ and, finally, $r_\infty\in D$, a contradiction.

Thus, $D=M$ and $F:M\rightarrow \tilde M$ is a local isometry, injective by Lemma \ref{lem:inj_D}, hence an isometry onto the open subset $F(M)$, with $F_*(\mathcal{B}_p)=\mathcal{B}_{\tilde p}$ by construction. As $M$ is inextensible, $F(M)$ cannot be a proper open subset of the connected manifold $\tilde M$, so $F(M)=\tilde M$.
\end{proof}

\begin{remark}\label{rem:proof_uniqueness}
(1) In the definite case, the previous proof recovers Theorem \ref{prop:uniqueness_riemann_convergence} without using the Hopf--Rinow theorem (every complete Riemannian manifold is inextensible) and no covering argument is needed. In general, the proof shows that $F$ extends to an isometry of $M$ onto an open subset of $\tilde M$ whenever $\tilde M$ is complete for one causal type, with no assumption on $M$.

(2) The key point is that the domain $D$ is closed under ``reaching'' by geodesics of the complete type. If no type of completeness is available, $D$ may accumulate at a boundary point $r_\infty$ with $F(z_k)$ convergent and $\|dF_{z_k}\|\rightarrow \infty$ (a boost), which is exactly the non-Hausdorff branching of the extensions of Misner. 
The following example shows that this actually happens, even for flat metrics and $C^\infty$ convergence, so that inextensibility of both limits does not suffice for the uniqueness of the limit.
\end{remark}

\begin{example}[A Misner-type sequence with two inextensible anchored limits]\label{ex:misner_double_limit}{\em
Let $\LL^2=(\R^2,-2\,du\,dv)$ in null coordinates and let $\mathbb{B}(u,v)=(e^{a}u,e^{-a}v)$, $a>0$, be the boost of Example \ref{ejemplo1}. Put $\Omega_1=\{u>0\}$, $\Omega_2=\{v>0\}$ and $\Omega_0=\Omega_1\cap\Omega_2$. The boost $\mathbb{B}$ acts freely and properly discontinuously on $\Omega_1$ and $\Omega_2$ (with fundamental domains $\{1\leq u<e^a\}$ and $\{1\leq v<e^a\}$), so that the {\em Misner cylinders} $C_j:=\Omega_j/\langle\mathbb{B}\rangle$, $j=1,2$, are smooth flat Lorentzian surfaces: the two Misner extensions of $C_0:=\Omega_0/\langle\mathbb{B}\rangle$. On $C_1$ we use the global coordinates $\theta=\frac{1}{a}\log u$ (mod $1$) $\in S^1$ and $w=uv\in\R$, in which the metric reads $-2a\,d\theta\,dw+2a^2w\,d\theta^2$ and $H_1:=\{w=0\}$ is a closed lightlike geodesic; on $C_2$, $\theta'=-\frac{1}{a}\log v$, $w=uv$, the metric reads $2a\,d\theta'\,dw+2a^2w\,d\theta'^2$ and $H_2:=\{w=0\}$. On $C_0=\{w>0\}$ both charts are related by $\theta'=\theta-\frac{1}{a}\log w$.

{\em The sequence.} Let $\mathbb{B}_i(u,v,z)=(e^au,e^{-a}v,z+\frac{1}{i})$ act on $\LL^3=(\R^3,-2\,du\,dv+dz^2)$. As $\mathbb{B}_i^m$ translates $z$ by $m/i$, the group $\langle\mathbb{B}_i\rangle\cong\Z$ acts freely and properly discontinuously on the whole $\LL^3$, so that $M_i:=\LL^3/\langle\mathbb{B}_i\rangle$ is a connected Hausdorff flat Lorentzian $3$-manifold (all the $M_i$ are diffeomorphic: $\LL^2$-bundles over a circle with monodromy $\mathbb{B}$); we denote by $\pi_i:\LL^3\rightarrow M_i$ the quotient map. Fix $x_0=(u_0,v_0)\in\Omega_0$ and take as basepoints and anchors $p_i:=\pi_i(x_0,0)$, $\mathcal{B}_{p_i}:=\Span(\partial_t)$, where $\partial_t:=\partial_u+\partial_v$. The two limits will be $M:=C_1\times\R$ and $\tilde M:=C_2\times\R$ with the product metrics $g$ and $\tilde g$, both with basepoint $q:=([x_0],0)\in C_0\times\R$ and anchor $\Span(\partial_t)$; notice that $q$ and the anchor are the same object in $M$ and $\tilde M$ through the common open subset $C_0\times\R$. We denote by $\rho_1:\Omega_1\times\R\rightarrow M$ and $\rho_2:\Omega_2\times\R\rightarrow \tilde M$ the quotient maps by $\mathbb{B}\times\mathrm{id}$. Notice that $\partial_t$ is invariant neither by $\mathbb{B}_i$ nor by $\mathbb{B}\times\mathrm{id}$ and, so, it does not descend to the quotients; thus, $\Span(\partial_t)$ at $p_i$ means $d\pi_i(\Span\,\partial_t|_{(x_0,0)})$ and $\Span(\partial_t)$ at $q$ means $d\rho_j(\Span\,\partial_t|_{(x_0,0)})$, $j=1,2$ (recall that $\rho_1=\rho_2$ on $\Omega_0\times\R$).

{\em (a) $(M_i,g_i,\mathcal{B}_{p_i})$ converges in $C^\infty$ to both $(M,g,\Span\partial_t)$ and $(\tilde M,\tilde g,\Span\partial_t)$.} Consider the shears $\Phi_i(u,v,z)=(u,v,z+\frac{\log u}{ai})$ on $\Omega_1\times\R$ and $\tilde\Phi_i(u,v,z)=(u,v,z-\frac{\log v}{ai})$ on $\Omega_2\times\R$. They are diffeomorphisms of $\Omega_j\times\R$ onto itself and, since $\log(e^{a}u)=\log u+a$ and $\log(e^{-a}v)=\log v-a$, they satisfy $\Phi_i\circ(\mathbb{B}\times \mathrm{id})=\mathbb{B}_i\circ\Phi_i$ and $\tilde\Phi_i\circ(\mathbb{B}\times \mathrm{id})=\mathbb{B}_i\circ\tilde\Phi_i$. Hence, they induce diffeomorphisms $\phi_i:M\rightarrow \pi_i(\Omega_1\times\R)$ and $\tilde\phi_i:\tilde M\rightarrow \pi_i(\Omega_2\times\R)$, determined by $\phi_i\circ\rho_1=\pi_i\circ\Phi_i$ and $\tilde\phi_i\circ\rho_2=\pi_i\circ\tilde\Phi_i$, onto the open subsets $\pi_i(\Omega_j\times\R)=(\Omega_j\times\R)/\langle\mathbb{B}_i\rangle$ of $M_i$ (injectivity: $\mathbb{B}_i^m\Phi_i(x,z)=\Phi_i(x',z')$ forces $x'=\mathbb{B}^m x$ and then $z'=z$). Using $d\theta=\frac{du}{au}$,
$$
\phi_i^*g_i=-2\,du\,dv+\Big(dz+\frac{du}{aiu}\Big)^2=g+\frac{2}{i}\,d\theta\,dz+\frac{1}{i^2}\,d\theta^2 ,
$$
a globally defined tensor which converges to $g$ in $C^\infty$ uniformly on compact subsets. Moreover, $\phi_i^{-1}(p_i)=([x_0],-\frac{\log u_0}{ai})\rightarrow q$ and $d\phi_i^{-1}(\partial_t)=\partial_t-\frac{1}{aiu}\partial_z\rightarrow \partial_t$, so $(d\phi_i^{-1})_{p_i}(\mathcal{B}_{p_i})\rightarrow\Span(\partial_t)$. The same computation with $\tilde\Phi_i$ gives $\tilde\phi_i^*g_i=\tilde g+\frac{2}{i}d\theta'dz+\frac{1}{i^2}d\theta'^2$, $\tilde\phi_i^{-1}(p_i)\rightarrow q$ and $d\tilde\phi_i^{-1}(\partial_t)\rightarrow\partial_t$. Thus, the requirements of Definitions \ref{def:riemann_convergence} and \ref{def:convergence_semi-riemann} hold with the trivial exhaustions $U_i=M$, $\tilde U_i=\tilde M$.

{\em (b) $M$ and $\tilde M$ are inextensible.} First, a connected semi-Riemannian manifold $N$ in which every incomplete geodesic $\sigma:[0,s^*)\rightarrow N$ is partially imprisoned (i.e. there exist a compact $K\subset N$ and $s_j\rightarrow s^*$ with $\sigma(s_j)\in K$) is inextensible.\footnote{Indeed, let $N\subset N'$, $N\neq N'$, be an isometric embedding as an open subset of a connected $N'$, $q'\in \partial N\subset N'$, $W$ a convex neighborhood of $q'$ in $N'$, $x\in W\cap N$, $\sigma:[0,1]\rightarrow W$ the geodesic from $x$ to $q'$ and $s^*:=\sup\{s: \sigma([0,s])\subset N\}$. As $N$ is open and $q'\notin N$, $\sigma(s^*)\notin N$; so, $\sigma|_{[0,s^*)}$ is an incomplete geodesic of $N$ (its continuation in $N'$ leaves $N$) which converges in $N'$ to $\sigma(s^*)\notin N$ and, thus, leaves every compact subset of $N$ (such subsets are closed in $N'$), a contradiction.} 
Now, geodesics of $M$ lift to straight lines $s\mapsto(u_1+s\dot u, v_1+s\dot v, z_1+s\dot z)$ of $\Omega_1\times \R$; such a line is incomplete iff $\dot u<0$, and then $u\rightarrow 0$, $v\rightarrow v^*$ finite and $z$ bounded as $s\rightarrow s^*=-u_1/\dot u$, so that $w=uv\rightarrow 0$ and $\theta=\frac1a\log u\rightarrow -\infty$: for $s$ close to $s^*$,  the geodesic stays in the compact set $\{|w|\leq 1\}\times[ z_1-|\dot z|s^*, z_1+|\dot z|s^*]$ (recall $C_1\cong S^1_\theta\times\R_w$), winding infinitely often around $H_1\times\R$. The same holds for $\tilde M$.

{\em (c) If $u_0\neq v_0$, there is no isometry $F:M\rightarrow \tilde M$ with $F(q)=q$ and $dF_q(\Span\partial_t)=\Span\partial_t$; moreover, for any $x_0$, there is no such isometry with $dF_q=\mathrm{id}$.} Otherwise, since $\Omega_j\times\R$ is convex, $\rho_1$ and $\rho_2$ are the universal coverings of $M$ and $\tilde M$, with deck group generated by $\Gamma:=\mathbb{B}\times\mathrm{id}$. As $F(q)=q$, we can lift $F$ to a diffeomorphism $\hat F:\Omega_1\times\R\rightarrow \Omega_2\times\R$ with $\hat F(x_0,0)=(x_0,0)$. Being a local isometry between connected open subsets of $\LL^3$, $\hat F$ is the restriction of an affine isometry $A$ of $\LL^3$, and $A(\Omega_1\times\R)=\Omega_2\times\R$ because $\hat F$ is onto. By uniqueness of lifts, $\hat F\circ\Gamma=\Gamma^m\circ\hat F$ for some $m\in\Z$ (an equality between affine maps which, thus, holds on the whole $\LL^3$), hence $dA\,(\mathbb{B}\oplus 1)\,dA^{-1}=\mathbb{B}^m\oplus 1$; comparing eigenvalues, $m=\pm1$, and $dA$ preserves the eigenspace $\R\partial_z$ and the plane $\R^2\times\{0\}$. Write $dA=L\oplus\epsilon$ with $L\in O(1,1)$, $\epsilon=\pm1$. Moreover, $A$ maps the fixed point set of $\Gamma$ (the $z$-axis) into the one of $\Gamma^m$ (again the $z$-axis); so, $A(u,v,z)=(L(u,v),\epsilon z+e)$, where $e=0$ because $A(x_0,0)=(x_0,0)$. Then, $A(\Omega_1\times\R)=\Omega_2\times\R$ reads $L\{u>0\}=\{v>0\}$. By the choice of the lift, the condition on the anchor reads $L(\partial_t)=\pm\partial_t$, hence $L=\mathrm{diag}(\pm1,\pm1)$ in the basis $(\partial_t,\partial_x)$, where $\partial_x:=\partial_u-\partial_v$; that is, $L(u,v)$ is one of $(u,v)$, $(-u,-v)$, $(v,u)$, $(-v,-u)$, which map $\{u>0\}$ onto $\{u>0\}$, $\{u<0\}$, $\{v>0\}$, $\{v<0\}$, respectively. So, $L(u,v)=(v,u)$, that is, $A(u,v,z)=(v,u,\epsilon z)$, and $A(x_0,0)=(x_0,0)$ yields $u_0=v_0$, a contradiction. Finally,  $dF_q=\mathrm{id}$ would force $L=\mathrm{ id}$, in contradiction with $L\{u>0\}=\{v>0\}$. The hypothesis $u_0\neq v_0$ is sharp: if $u_0=v_0$, the map $(u,v,z)\mapsto (v,u,\pm z)$ does induce an isometry $M\rightarrow \tilde M$ preserving the basepoint and the anchor.

Summing up, if $u_0\neq v_0$, the sequence $\{(M_i,g_i,\mathcal{B}_{p_i})\}$ has two inextensible anchored limits which are not related by any isometry preserving the basepoint and the anchor.

}
\end{example}


\begin{remark}\label{r_Geroch} In 
1969  \cite{Geroch},  R. P. Geroch studied  the limit of a smooth one-parameter family of spacetimes from an intrinsic viewpoint. This is obviously more restrictive than the Cheeger-Gromov convergence explained above (for example,  
 the difficulties which led to Lemmas \ref{lema1}--\ref{lema3} would not appear). However, the main dissimilarity lies in the different aims of both approaches.  
Indeed, Geroch's goal was to construct a  universal limit space. With this aim, he introduced ideas such as the convergence of a curve of frames, which underlie pointed Cheeger--Gromov convergence and the anchored one here. However, he did not intend to obtain results on uniqueness of limits up to isometries, as is standard in Cheeger--Gromov theory and here. 
Next, we summarize briefly his approach and explain the relation with the Cheeger-Gromov setting.

 Essentially, Geroch studies a smooth family of spacetimes 
$\{(M_\lambda, g_\lambda), \lambda >0\}$ which depend smoothly on $\lambda$ and match in  a 1-higher dimensional  manifold $\m$ endowed with a degenerate dual metric $g^{ab}$ with kernel Span$\{d\lambda\}$. Then,  he considers any possible  limit manifold  $\m'$ with boundary satisfying: (i) $\m$ is isometrically embedded in the interior of $\m '$  inducing a function $\lambda'$ on it, and (ii) the boundary $\partial \m'= \{\lambda' =0\}$ is also a smoothly matched spacetime.  He stresses that $\m'$ might be non-Hausdorff  (see \cite[footnote 3]{Geroch}). 

 For each limit space $\m'$,  Geroch constructs a maximal one   $\bar \m$ and claims that  it is unique.  This is obtained by  considering all its further possible extensions and, then, their disjoint union $\n$  up to two relations of equivalence. The first one identifies all the points in the interior of the limit manifolds with  the original one in $\m$. For the second one, given two extensions  $\m_1, \m_2$ he identifies two points in the boundary $p_1 \in \m_1, p_2\in \m_2$ when there is a curve of frames $\omega(\lambda), \lambda >0$ in $\m$ such that the curve in each $\m_i$ identified with $\omega(\lambda)$ converges to a frame of $\partial \m_i$ at each $p_i$. 

The assumption of the convergence of frames  implies that, whenever $p_1$ and $p_2 $ are identified, small open neighborhoods of them $U_1 \subset \partial \m_1$ and $U_2\subset \partial \m_2 $ are identified too. Based on this, Geroch asserts that the universal limit $\bar \m$ is a  (possibly non-Hausdorff) manifold. 

Indeed, when $\m_1, \m_2$ admit a common converging  curve of frames,  one can ensure only that two (biggest) open neighborhoods $U_1, U_2$ are isometric. However, the isometry cannot be extended to the points of the boundary of each $U_i$ in $\partial \m_i$. These non-identified points naturally lead to non-Hausdorffness. 

In comparison, when the uniqueness of limits holds as studied here (i.e., under the hypotheses of Theorem \ref{coro_simplyconnected}), each sequence of converging frames determines a unique isometry class of limits, and each Geroch limit $\m'$ would be Hausdorff. Example \ref{ex:misner_double_limit} shows, however, that non-Hausdorff limits do appear under mere inextensibility of the limits, in agreement with Geroch's warning.  
 This seems enough for any purpose, even though one might try to reformulate Geroch's proposal regarding the set of all the non-isometric limits as non-Hausdorff separated limits of the original sequence (or family)  of spacetimes\footnote{
Notice, however, that Geroch requires $\partial \m_i$ to be Hausdorff, connected and non-empty \cite[p. 182, item 2]{Geroch}.}.   



\end{remark}

\section{Convergence of globally hyperbolic spacetimes}\label{sec:convergence_globally_hyperbolic}

As discussed in Section~\ref{subsec:5_2_1}, in the spacetime setting there is no canonical norm available to establish convergence. However, a temporal function enables the selection of a representative metric $g^\tau$ within the conformal class $[g]$, along with a timelike vector field $T$, which can be normalized to have unit length. Furthermore, in the globally hyperbolic case, temporal functions that are $h$-steep yield a Wick-rotated Riemannian metric $g_W ^\tau$ that is complete and can thus be used to study convergence.

We begin by studying spacetimes endowed with a privileged timelike vector field $T$. In future subsections, the vector field $T$ will be $\nabla \tau$ for a temporal function $\tau$.

\subsection{Convergence of Wick-complete anchored spacetimes}\label{sec:convergence_observed_spacetimes}


\begin{definition}
A {\em Wick-complete anchored spacetime $(M,g,T,p)$} is a spacetime $(M,g)$ together with a basepoint $p$ and a unit\footnote{\label{foot:T} This condition is imposed for simplicity. Otherwise, one can equally carry out a Wick rotation. Moreover, a normalization could be used and consider the split between convergence of conformal structures and conformal factors as in the case of $h$-steep anchoring later.}
 timelike vector field $T$ (with $g$-dual form denoted $T^\flat = g(T, \cdot)$) such that the Wick-rotated metric 
$$  g_T := g + 2 T^\flat \otimes T^\flat  $$ 
is a complete Riemannian metric and we will denote $|\cdot|_T = |\cdot|_{g_T}$. In particular, such a spacetime is anchored as a semi-Riemannian manifold by $\Span(T_p)$.
\end{definition}


\begin{definition}[Cheeger-Gromov convergence for Wick-complete anchored spacetimes]\label{def:lorentz_convergence}
	A sequence of Wick-complete anchored spacetimes $(M_i, g_i, T_i, p_i)$ is said to {\em (Cheeger-Gromov) converge in the $C^{k}$ topology}, $k \geq 0$, to a Wick-complete anchored spacetime $(M,g,T,p)$ if
	\begin{enumerate}[label=(\roman*)]
		\item the pointed manifolds $(M_i, g_i, p_i)$ converge in $C^k$ to $(M,g,p)$ in the sense of Definition \ref{def:riemann_convergence}, and
		\item $(\phi_i^{-1})_* (T_i)$ converges in 
         $C^{k}$ to $T$.
	\end{enumerate}
\end{definition}

{We now prove that the limit of a sequence of Wick-complete anchored spacetimes is unique up to isometry. The associated Wick-rotated metrics are constructed from the two elements that remain under control: the sequence of Lorentzian metrics and the corresponding observer fields, both of which converge. This structure enables the comparison of the resulting Riemannian metrics and provides uniform control over the diffeomorphisms between different limits $M,\tilde M$, specifically $\psi_i = \tilde\phi_i^{-1} \circ \phi_i$, with the notation of \S \ref{subsec:5_2_3}.
Moreover, the completeness of the Wick-rotated metrics yields a natural exhaustion for the spacetimes.}  Notice that the results of this section rely on the Riemannian uniqueness of limits (Theorem \ref{prop:uniqueness_riemann_convergence}) through the Wick rotation, and not on Theorem \ref{coro_simplyconnected}. 

The proof of the uniqueness of the limit under $C^0$ convergence is based on the behavior of the {\em local} Lorentzian distance. For an open subset $U$ of a spacetime and $x,y\in U$, let $d^U(x,y)$ be the Lorentzian distance in $U$, that is, the supremum of the lengths of the future-directed causal curves from $x$ to $y$ contained in $U$ (and $d^U(x,y)=0$ if there is none); clearly, $d^U\leq d^{U'}$ if $U\subset U'$. If $U$ is a convex neighborhood, writing $E(x,y):=\exp_x^{-1}(y)$ (a smooth map on $U\times U$), we have $J^+(x,U)=\{y\in U: E(x,y) \hbox{ is future-directed causal}\}$, $d^U(x,y)=|E(x,y)|_g$ when $y\in J^+(x,U)$ and $d^U(x,y)=0$ otherwise (the causal geodesic segment is the longest causal curve in $U$ joining its endpoints, see \cite[Lemma 14.2]{Oneill} and, for the maximality, the proof of the next lemma); in particular, for $x,y$ in a convex $U_0\subset U$, the value $d^U(x,y)$ does not depend on the convex $U\supset U_0$.

\begin{lemma}\label{lcorrec}
Let $(M,g)$ be a spacetime, $W\subset M$ a convex neighborhood and $U\subset W$ a convex neighborhood with compact closure $\overline{U}\subset W$. Let $\{g^{(i)}\}$ be a sequence of time-oriented Lorentzian metrics on $W$ such that $g^{(i)}\rightarrow g$ in $C^0$ uniformly on $\overline U$ and the time orientations are compatible, in the sense that some $g$-future-directed timelike vector field $T$ on $W$ is $g^{(i)}$-future-directed for $i$ large. Then
$$
\lim_{i\rightarrow\infty} d^U_{g^{(i)}}(x_i,y_i)=d^U_g(x,y)\qquad \hbox{whenever } x_i\rightarrow x,\ y_i\rightarrow y \hbox{ in } U .
$$
\end{lemma}
\begin{proof}
The inequality $\liminf_i d^U_{g^{(i)}}(x_i,y_i)\geq d^U_g(x,y)$ holds because $d^U_g(x,y)$ is approximated by the lengths of $g$-timelike curves in $U$ from $x$ to $y$, which are $g^{(i)}$-timelike for $i$ large, have $g^{(i)}$-lengths converging to their $g$-lengths and can be joined to $x_i$, $y_i$ by short timelike segments. 

The inequality $\limsup_i d^U_{g^{(i)}}(x_i,y_i)\leq d^U_g(x,y)$ follows from the limit curve lemma (the $g^{(i)}$-causal curves in $U$ from $x_i$ to $y_i$ have uniformly bounded length for any auxiliary Riemannian metric $h$, so they subconverge to a $g$-causal curve in $\overline U$ from $x$ to $y$) together with the upper semicontinuity of the Lorentzian length \cite[Ch. 3]{BEE}, applied to the metrics $g-\epsilon h$, $\epsilon\rightarrow 0$, whose cones are wider than those of all the $g^{(i)}$ for $i$ large.
\end{proof}

\begin{theorem}\label{prop:uniqueness_observed_spacetimes}
	The limit of a $C^0$ converging sequence of Wick-complete anchored spacetimes $(M_i, g_i, T_i, p_i)$ is unique up to isometries; more precisely, if the sequence converges to $(M,g,T,p)$ and to $(\tilde M,\tilde g,\tilde T,\tilde p)$, then there exists an isometry $F:(M,g)\rightarrow (\tilde M,\tilde g)$ with $F(p)=\tilde p$ and $F_*(T)=\tilde T$.
\end{theorem}

\begin{proof}
Assume that $(M_i,g_i,T_i,p_i)$ converges in $C^0$ with diffeomorphisms $\phi_i$ to $(M,g,T,p)$ and with diffeomorphisms $\tilde\phi_i$ to $(\tilde M,\tilde g,\tilde T,\tilde p)$, and put $\psi_i:=\tilde\phi_i^{-1}\circ\phi_i$ as in \S \ref{subsec:5_2_3}. We take on $M_i$, $M$ and $\tilde M$ the time orientations that make $T_i$, $T$ and $\tilde T$ future-directed; as $(\phi_i^{-1})_*T_i\rightarrow T$ and $(\tilde\phi_i^{-1})_*T_i\rightarrow \tilde T$, the map $\psi_i$ sends $\phi_i^*g_i$-future-directed causal vectors to $\tilde\phi_i^*g_i$-future-directed causal vectors for $i$ large. The proof is carried out in three steps.

{\em Step 1: the Wick-rotated metrics.} Fix a compact subset $K\subset M$. Since $(\phi_i^{-1})_*T_i\rightarrow T$ and $\phi_i^*g_i\rightarrow g$ in $C^0$ uniformly on $K$, we have $\phi_i^*(T_i^\flat)\rightarrow T^\flat$ and, thus,
$$
\phi_i^{*} g_{i, T_i} = \phi_i^{*} g_i + 2 \phi_i^{*}(T_i^{\flat} \otimes T_i^{\flat}) \rightarrow g_T = g + 2 T^{\flat} \otimes T^{\flat} \quad \hbox{in $C^0$ uniformly on } K
$$
(here $T^\flat$, $\tilde T^\flat$ and $T_i^\flat$ are taken with respect to $g$, $\tilde g$ and $g_i$, respectively). Analogously, $\tilde\phi_i^*g_{i,T_i}\rightarrow \tilde g_{\tilde T}$ in $C^0$ on compact subsets of $\tilde M$. Hence, the pointed complete Riemannian manifolds $(M_i,g_{i,T_i},p_i)$ converge in $C^0$ to both $(M,g_T,p)$ and $(\tilde M,\tilde g_{\tilde T},\tilde p)$. By Theorem \ref{prop:uniqueness_riemann_convergence} these limits are isometric; in fact, its proof (see \cite[Theorem 8.1.7]{burago:metric} and \cite[Thm 5.6.15]{Petersen}) provides more: on each compact subset, the maps $\psi_i$ are bi-Lipschitz for the distances of $g_T$ and $\tilde g_{\tilde T}$ with constants converging to $1$, so that, by the Arzel\`a--Ascoli theorem (applied to $\psi_i$ and $\psi_i^{-1}$), a subsequence converges uniformly on compact subsets to a distance-preserving bijection $F:(M,d_{g_T})\rightarrow (\tilde M,d_{\tilde g_{\tilde T}})$ with $F(p)=\tilde p$, which is a smooth Riemannian isometry by the Myers--Steenrod theorem, $F^*(\tilde g_{\tilde T})=g_T$. It remains to prove that $F^*\tilde g=g$ and $F_*(T)=\tilde T$ or, equivalently (because $g=g_T-2T^\flat\otimes T^\flat$), that $F^*(\tilde T^\flat)=T^\flat$.

{\em Step 2: the case of $C^1$ convergence.} If the convergence in Definition \ref{def:lorentz_convergence} is $C^1$, the isometries $\psi_i$ between the Riemannian metrics $\phi_i^*g_{i,T_i}$ and $\tilde\phi_i^*g_{i,T_i}$ (which are $C^1$-close to $g_T$ and $\tilde g_{\tilde T}$ on compact subsets) have uniformly bounded second derivatives on compact subsets, because these are determined by the first derivatives and by the Christoffel symbols, which converge in $C^0$. Then, the subsequence above converges in $C^1$ and, taking limits in $\psi_i^*(\tilde\phi_i^*T_i^\flat)=\phi_i^*T_i^\flat$, we get $F^*(\tilde T^\flat)=T^\flat$ directly. However, $C^0$ convergence does not ensure the convergence of the differentials $d\psi_i$, and we argue as follows.

{\em Step 3: the case of $C^0$ convergence, via the local Lorentzian distance.} Fix $x_0\in M$ and choose, in this order: $\tilde g$-convex neighborhoods $V'\subset\subset V''$ of $F(x_0)$; $g$-convex neighborhoods $W_1\subset\subset W_2\subset\subset W_3$ of $x_0$ with $F(\overline{W_3})\subset V'$; a $\tilde g$-convex neighborhood $V$ of $F(x_0)$ with $\overline V\subset F(W_1)$; and a neighborhood $W_0\subset W_1$ of $x_0$ with $F(W_0)\subset V$. Since $\psi_i\rightarrow F$ uniformly on $\overline{W_3}$ and $F$ is a homeomorphism, $F(W_1)\subset \psi_i(W_2)\subset F(W_3)$ for $i$ large. As $\psi_i$ is an isometry between $\phi_i^*g_i$ and $\tilde\phi_i^*g_i$ preserving the time orientations, for $x,y\in W_0$ and $i$ large, the monotonicity of $d^U$ with respect to $U$ yields
\begin{equation*}
\begin{aligned}
d^{V}_{\tilde\phi_i^*g_i}(\psi_ix,\psi_iy)\ &\leq\ d^{F(W_1)}_{\tilde\phi_i^*g_i}(\psi_ix,\psi_iy)\ \leq\ d^{W_2}_{\phi_i^*g_i}(x,y)= d^{\psi_i(W_2)}_{\tilde\phi_i^*g_i}(\psi_ix,\psi_iy)\\
&\leq\ d^{F(W_3)}_{\tilde\phi_i^*g_i}(\psi_ix,\psi_iy)\ \leq\ d^{V'}_{\tilde\phi_i^*g_i}(\psi_ix,\psi_iy).
\end{aligned}
\end{equation*}
Now, Lemma \ref{lcorrec} applies to $U=W_2$ (with $W=W_3$) and the metrics $\phi_i^*g_i$, as well as to $U=V$ and $U=V'$ (with $W=V'$ and $W=V''$, respectively) and the metrics $\tilde\phi_i^*g_i$, the time orientations being compatible because $(\phi_i^{-1})_*T_i\rightarrow T$ and $(\tilde\phi_i^{-1})_*T_i\rightarrow \tilde T$. Taking limits (with $\psi_ix\rightarrow Fx$, $\psi_iy\rightarrow Fy$),
$$
d^{V}_{\tilde g}(Fx,Fy)\ \leq\ d^{W_2}_{g}(x,y)\ \leq\ d^{V'}_{\tilde g}(Fx,Fy)=d^{V}_{\tilde g}(Fx,Fy),
$$
the last equality because $Fx,Fy\in V$, both $V\subset V'$ are convex and $F(W_0) \subset V$. Therefore, $F$ preserves the local Lorentzian distance: $d^{W_2}_g(x,y)=d^V_{\tilde g}(Fx,Fy)$ for all $x,y\in W_0$. In particular, $F$ maps $I^+(x,W_2)\cap W_0$ into $I^+(Fx,V)$. Now, take $v\in T_xM$ future-directed timelike and put $w_s:=\exp_{Fx}^{-1}(F(\exp_x(sv)))$ for $s>0$ small; as $F$ is smooth, the first-order Taylor expansion of $\exp_{Fx}^{-1}\circ F\circ \exp_x$ at $0$ (in normal coordinates, as in \eqref{eq:approx_null}) gives $w_s=s\,dF_x(v)+o(s)$. Then,
$$
s\,|v|_g=d^{W_2}_g(x,\exp_x(sv))=d^V_{\tilde g}(Fx,F(\exp_x(sv)))=|w_s|_{\tilde g},\qquad\hbox{i.e.}\qquad -\tilde g(w_s,w_s)=s^2\,|v|_g^2,
$$
and, dividing by $s^2$ and letting $s\rightarrow 0$, we get $-\tilde g(dF_x(v),dF_x(v))=|v|_g^2>0$: $dF_x(v)$ is timelike with $|dF_x(v)|_{\tilde g}=|v|_g$, and it is future-directed because it is the limit of the future-directed causal vectors $w_s/s$. By polarization on the open cone of timelike vectors, $F^*\tilde g=g$ at every $x$, and $F$ preserves the time orientation. Finally, $F^*(\tilde g_{\tilde T})=g_T$ and $F^*\tilde g=g$ give $F^*(\tilde T^\flat\otimes \tilde T^\flat)=T^\flat\otimes T^\flat$, i.e. $F_*T=\pm \tilde T$, and the sign is $+$ because $F$ preserves the time orientation. So, $F:(M,g)\rightarrow (\tilde M,\tilde g)$ is an isometry with $F(p)=\tilde p$ and $F_*(T)=\tilde T$.

\end{proof}


\subsection{Convergence of conformal structures endowed with a temporal function} \label{subsec:6_2}


Having established a notion of smooth convergence for spacetimes that guarantees uniqueness, we now turn to the convergence of the conformal structures on globally hyperbolic spacetimes-with-timelike-boundary.

By Lemma \ref{lemma:g_w-completa} and Corollary \ref{coro:wick-rotated_globally hyperbolic}, a spacetime-with-timelike-boundary is globally hyperbolic if it admits an $h$-steep temporal function $\tau$. In this case, the $\tau$-representative $g^\tau$ of a globally hyperbolic conformal class yields a complete Riemannian metric through its Wick-rotated metric $g_{W}^\tau$, and $\tau$ is $(g_W^\tau /2)$-steep (recall Proposition \ref{p_hsteep_is_time}). 

\begin{definition}\label{def:convenient_temporal_function}
	A temporal function $\tau$ on a globally hyperbolic spacetime-with-timelike-boundary $(\bM,g)$ is said to be {\em adapted} if its $g$-gradient is tangent to $\partial M$.
	
	An {\em $h$-steep anchored spacetime} is a quadruple $(\bM, g, \tau, p)$ where $(\bM, g)$ is a (necessarily globally hyperbolic) spacetime-with-timelike-boundary, $\tau$ is an adapted $h$-steep temporal function and $p \in M$ is the basepoint.
\end{definition}

\begin{remark}
	We will frequently work with adapted $h$-steep temporal functions in the context of the convergence of globally hyperbolic spacetimes, and we will show that the limit spacetime also admits an $h$-steep function. If these functions are steep (a property satisfied by any $h$-steep function for the $\tau$-representative $g^\tau$, see Definition \ref{d_wick}), then this property is also preserved in the limit.
\end{remark}

Observe that any $h$-steep anchored spacetime $(\bM,g,\tau,p)$ naturally determines a Wick-complete anchored spacetime $(\bM, g^\tau, \nabla \tau, p)$: indeed, one can take $T = \nabla \tau$ and normalize it to obtain a unit timelike vector field with respect to the $\tau$-representative $g^\tau$. This construction enables the study of convergence of globally hyperbolic spacetimes as Wick-complete anchored spacetimes, by analyzing the convergence of the associated complete Wick-rotated Riemannian metrics. In the limit, global hyperbolicity is recovered via a reverse Wick rotation. 

Following the nomenclature of Definition \ref{d_wick}, we work with the $\tau_i$-representatives $g_i^{\tau_i}$ of the conformal class of each metric $g_i$ (this representative is chosen uniquely so that $|\nabla_i \tau_i|_{g_i} = 1$). Although this involves fixing a specific representative in each conformal class, the resulting convergence still captures the behavior of the underlying conformal structures, as the limit yields a temporal function and the corresponding canonical representative. 

\begin{proposition}[Uniform convergence on compact sets]\label{prop:convergence_conformal}
    Let $\bM$ be a smooth manifold with (possibly empty) boundary. Let $\{g_i\}_{i \in \N}$ be a sequence of globally hyperbolic metrics each one endowed with an adapted $\frac{1}{2} g^{\tau_i}_W$-steep temporal function $\tau_i$. Assume that 
	\begin{enumerate}[label=(\roman*)]
		\item the Wick-rotated metrics $h_i := g_W^{\tau_i} = g_i^{\tau_i} + 2d\tau_i ^2$ converge in $C^k$, $k\geq 0$, uniformly on compact sets to a complete Riemannian metric $h$ , and
		\item the sequence $\{\tau_i\}_{i \in \N}$ converges in $C^{k+1}$ to a function $\tau$ on $(\bM, h)$.
	\end{enumerate}
	Then the metrics $\{g_i^{\tau_i}\}_{i \in \N}$ converge in $C^k$ to the metric $g:= h - 2d\tau^2$  uniformly on compact sets. Moreover, $g$ inherits a time orientation and an adapted $\frac12 h$-steep temporal function $\tau$ (in particular, making it globally hyperbolic).
\end{proposition}

\begin{proof}
	Fix $1/2 >\epsilon > 0$. The $C^{k+1}$ convergence of $\{\tau_i\}_{i \in \N}$ implies the convergence of the tensors $d\tau_i$ to $d\tau$ and $d\tau_i^2$ to $d\tau^2$ uniformly on compact sets; that is, for any compact set $K \subset \bM$ there exists $i_0 = i_0 (\epsilon)$ such that, for $i \geq i_0$,
	\begin{equation} \label{eq:dtau_convergencia}
		\sup_{x \in K} |d\tau_i - d\tau|_h < \epsilon, \quad  \text{and} \quad \sup_{x \in K} |d\tau_i ^2 - d\tau^2|_h < \epsilon.
	\end{equation}
	Then, for any $x \in K$, 
	\begin{equation*}
		|g_i^{\tau_i} - g|_h =  | h_i - 2d\tau_i ^2 - h + 2d \tau^2 |_h 
		\leq |h_i - h|_h + 2|d\tau_i ^2 - d\tau^2|_h 
		< \epsilon + 2 \epsilon,
	\end{equation*}
	which proves the $C^0$ convergence on compact subsets of $\bM$.  
	To ensure full $C^k$ convergence, let $\nabla$ denote the Levi-Civita connection associated with the limit Riemannian metric $h$. Applying $\nabla$ 
	\begin{equation}\label{eq:ck_derivatives}
		\nabla \left( g^{\tau_i}_i - g \right) = \nabla \left( h_i - h \right) - 2 \nabla \left(d\tau_i \otimes d\tau_i - d\tau \otimes d\tau \right).
	\end{equation}
	By the linear nature of $\nabla$, the first term on the right-hand side of \eqref{eq:ck_derivatives} 
	vanishes uniformly on compact sets due to the hypothesis of $C^k$ convergence of the Wick-rotated metrics. For the second term, 
	the Leibniz product rule expands the derivatives of the tensor product as:
	\begin{equation*}
		\nabla \left(d\tau_i \otimes d\tau_i \right) = \nabla d\tau_i \otimes d\tau_i + d\tau_i \otimes \nabla  d\tau_i.
	\end{equation*}
	The triangle inequality for $|\cdot|_h$ and the $C^{k+1}$ convergence of $\tau_i$ gives us that the second term on the right-hand side in \eqref{eq:ck_derivatives} can be bounded (adding and subtracting the cross terms $\nabla d \tau \otimes d \tau_i$ and $d\tau_i \otimes \nabla d\tau$):
	\begin{align*}
		| \nabla d\tau_i \otimes d\tau_i + d\tau_i \otimes \nabla  d\tau_i - \nabla d\tau \otimes d\tau - d\tau \otimes \nabla  d\tau|_h &\leq \\
		|\left (\nabla d\tau_i - \nabla d\tau \right) \otimes d\tau_i |_h &+ |\nabla d\tau \otimes (d\tau_i - d\tau)|_h + \hbox{(analogous terms)}
	\end{align*}
	since all the terms on the right-hand side  vanish  as $i \to \infty$. Similarly for all derivatives up to order $k$.
	
	We only need to prove that $g$ is Lorentzian. 
	First we will prove that $|\nabla \tau|_h = 1$. To do this, recall that $|\nabla_{i} \tau_i|_{i} = 1$ which, together with the first inequality in \eqref{eq:dtau_convergencia}, implies that
	\begin{equation}\label{eq:dtau_unitario_1}
		|d\tau|_h < |d\tau_i - d\tau|_h + |d\tau_i|_h  < \epsilon + |d\tau_i|_h \leq \epsilon + (1+\epsilon)^{1/2}|d\tau_i|_{h_i} < 1 + 2\epsilon.\\[2pt]
	\end{equation}
	Similarly, we get that 
	\begin{equation}\label{eq:dtau_unitario_2}
		1-2\epsilon < |d\tau|_h,
	\end{equation}
	which together with \eqref{eq:dtau_unitario_1}, implies that $|\nabla \tau|_h = |d\tau|_h = 1$. Hence, $g = h - 2d\tau^2$ is a Lorentzian metric, and we can choose the time-orientation given by the unit timelike vector field $-\nabla \tau$, which makes $\nabla \tau$ past-directed and unit. So, for any future-directed $g$-causal vector $v \in T\bM$, necessarily 		$h (v,v) - 2d\tau^2 (v) \leq 0$, and then,
	\begin{equation}\label{eq:tau is hsteep}
		d\tau (v) \geq \frac{1}{\sqrt{2}} |v|_h,
	\end{equation}
	which means that $\tau$ is $\frac{1}{2}h$-steep, where $h$ is complete. By Theorem \ref{thm:GHnulldistance-conborde}, the spacetime $(\bM, g)$ is globally hyperbolic, and $\tau$ is a Cauchy temporal function.
	
	Finally, the fact that the $g$-gradient of $\tau$ is tangent to $\partial M$ follows directly from the convergence of the metrics $g_i$ and the $C^{k+1}$ convergence of the temporal functions $\tau_i$.
\end{proof}

\begin{remark} \label{remark:convergence_conformal}
	\begin{enumerate}[label=(\Alph*)]
		\item The uniqueness of the limit in Proposition \ref{prop:convergence_conformal} is ensured by { Theorem~\ref{prop:uniqueness_observed_spacetimes}}, since the $C^{k+1}$ convergence of the adapted $h$-steep temporal functions provides the required convergence of unit timelike vector fields.

		\item\label{remark:convergence_conformal_B} Note that, if we were working with non-normalized metrics, it would suffice to consider the splitting
		$$g_i = -\Lambda_i d\tau_i ^2 + \sigma_{i},$$
		and, under the same assumptions as before, together with additional uniform convergence of the conformal factors $\Lambda_i$ to a smooth function $\Lambda$, we would obtain convergence to a metric of the form $g = \Lambda (h - 2d\tau^2)$.
	\end{enumerate}
\end{remark}

We can upgrade Proposition \ref{prop:convergence_conformal} -- using the same ideas from its proof -- to obtain Cheeger-Gromov convergence of the conformal structures of a sequence of $h$-steep anchored spacetimes, proving Theorem \ref{thm:convergence_uptodiffeo} in the Introduction.


\begin{proof}[Proof of Theorem \ref{thm:convergence_uptodiffeo}]
	The proof is very similar to that of Proposition \ref{prop:convergence_conformal}. Hypotheses (i) and (ii) imply that for any compact subset $K \subset \bM$ and any $\epsilon > 0$, there exists $i_0 = i_0 (\epsilon)$ such that for $i \geq i_0$
	\begin{equation*}\label{eq1:uptodiffeo}
		\begin{array}{c}
			\sup_{x \in K}  |\phi_i^* g_W^{\tau_i} - h|_h< \epsilon \\ \\
		\sup_{x \in K} |\phi_i^* (d\tau_i) - d\tau|_h < \epsilon, \quad  \text{and} \quad \sup_{x \in K} |\phi_i^* (d\tau_i ^2) - d\tau^2|_h < \epsilon.
	\end{array}
	\end{equation*}
	Then, for any $x \in K$,
	\begin{equation*}
		|\phi_i^* (g_i^{\tau_i}) - g|_h =  | \phi_i^*(g_W^{\tau_i} - 2d\tau_i ^2) - h + 2d \tau^2 |_h \\[2pt]
		\leq |\phi_i^* g_W^{\tau_i} - h|_h + 2|\phi_i^* d\tau_i ^2 - d\tau^2|_h \\[2pt]
		< \epsilon + 2 \epsilon.
	\end{equation*}
	This establishes $C^k$ convergence of the metrics $g_i^{\tau_i}$, uniformly on compact subsets of $\bM$. Global hyperbolicity follows analogously to \eqref{eq:dtau_unitario_1} and \eqref{eq:dtau_unitario_2}, considering only the pullbacks via $\phi_i$. Since $h$ is a complete Riemannian metric, this also implies that $\tau$ is an adapted $\frac12 h$-steep temporal function.

    Uniqueness is ensured by Theorem~\ref{prop:uniqueness_observed_spacetimes}, as discussed in Remark \ref{remark:convergence_conformal}.
\end{proof}

The following result is a direct consequence of the previous theorem. 

\begin{corollary} \label{prop:equivalence_convergence}
	Let $(\bM_i, g_i, \tau_i, p_i)$ be a sequence of $h$-steep anchored spacetimes. Assume that the sequence $\{\tau_i \circ \phi_i \}$ converges in $C^{k+1}$ to a smooth function $\tau$ on a manifold $\bM$, where $\phi_i$ are diffeomorphisms defined on an exhaustion of $\bM$. Then the following are equivalent:
	\begin{enumerate}[label = (\roman*)]
		\item The spacetimes $(\bM_i, g_i^{\tau_i}, \tau_i, p_i)$ converge in $C^k$ to an $h$-steep anchored spacetime $(\bM, g^\tau, \tau, p)$ with diffeomorphisms $\phi_i$ (according to Definition \ref{def:lorentz_convergence}).
		\item The pointed complete Riemannian manifolds $(\bM_i, g_W^{\tau_i}, p_i)$ converge in $C^k$ to the complete Riemannian manifold $(\bM, g_W ^\tau, p)$ with diffeomorphisms $\phi_i$ (according to Definition \ref{def:riemann_convergence}).
	\end{enumerate}
	In the case they hold, if the lapses $\Lambda_i = |\nabla_i \tau_i|_i^{-2}$ converge uniformly to $\Lambda$, the metrics $g_i$ converge to $g := \Lambda g^\tau$.
\end{corollary}

Notice that if $\bM_i = \bM$ and $\phi_i = id$ for all $i$, then the previous corollary becomes an equivalence between convergence of globally hyperbolic Lorentzian metrics and the convergence of their Wick-rotated counterparts. Thus we could use the convergence of the temporal functions, their lapses and the Wick-rotated metrics as a definition for convergence of $h$-steep anchored spacetimes and their conformal structures.


In summary, we have introduced the following definitions for convergence each one more restrictive than the previous one, but more practical.
\begin{enumerate}
    \item Anchored semi-Riemannian manifolds (Definition \ref{def:convergence_semi-riemann}). General and implying at least uniqueness of the limits around  the selected points for $C^2$ convergence (Theorem \ref{teor:uniqueness_semi-riemannian}, Theorem \ref{coro_simplyconnected}).
    \item Wick-complete anchored spacetimes (Definition \ref{def:lorentz_convergence}) For spacetimes endowed with a unit timelike vector field. Uniqueness of the limit already under $C^0$ convergence (Theorem~\ref{prop:uniqueness_observed_spacetimes})
    \item $h$-steep anchored spacetimes (Definition \ref{def:convenient_temporal_function} and Theorem \ref{thm:convergence_uptodiffeo}), which highlights the convergence of the conformal structure.
\end{enumerate}
Indeed, the last one is especially useful to obtain existence of limits, as explained next.

\subsection{A compactness theorem for globally hyperbolic spacetimes} \label{subsec:6_3}

We now employ Riemannian techniques to establish a compactness theorem for globally hyperbolic spacetimes, formulated in terms of their Wick-rotated metrics. This result illustrates the kind of compactness theorems that can be obtained by applying Riemannian methods to (complete) Wick-rotated metrics. Moreover, alternative notions of Riemannian convergence may also be explored within this framework to derive further convergence results for globally hyperbolic spacetimes.

\begin{corollary} \label{thm:compacidad_globalmente_hiperbolicos}
	Let $M$ be a smooth manifold without boundary, and let $\{g_i\}$ be a sequence of globally hyperbolic Lorentzian metrics on $M$, each endowed with adapted $h$-steep temporal functions $\tau_i$. Assume that
	\begin{enumerate}[label=(\alph*)]
		\item[(1)] the temporal functions $\tau_i$ converge in $C^{k+1}$ to a smooth function $\tau_{\infty}$, and
		\item[(2)] the lapse functions $\Lambda_i = |\nabla_i \tau_i|^{-2}_{g_i}$ converge to $\Lambda_\infty$ in $C^k$, 
	\end{enumerate} 
	and the associated Wick-rotated metrics $g_W^{\tau_i}$ satisfy the conditions in Theorem \ref{thm:compactness_theorem_riemann} for a fixed point $p \in M$.
    
	Then, there exists a subsequence of $\{g_i\}_i$ that converges in $C^{k}$, uniformly on compact subsets of $M$, to a globally hyperbolic Lorentzian metric $g_\infty$ that admits $\tau_\infty$ as an $h$-steep temporal function. 
\end{corollary}

\begin{proof}
	The proof follows from Theorem \ref{thm:compactness_theorem_riemann}, which guarantees the existence of a subsequence $\{h_i\}_i$ converging to a complete Riemannian metric $h_\infty$. {Subsequently, applying Corollary~\ref{prop:equivalence_convergence} (which also holds with empty boundary) we get that a (further) subsequence $\{g_i\}$ converges to a globally hyperbolic Lorentzian metric $g_\infty$.}
\end{proof}

The case with boundary may be approached by adapting  the estimates from Lemma \ref{lemma:uniform_bounds_metrics_derivatives} to the setting of manifolds with boundary and using extension techniques like those developed by M\"uller \cite{muller2018cheeger} but we    leave this line of investigation for future work.

\begin{remark}\label{remark_correccion}
It is worth pointing out that the convergence for the temporal functions $\tau_i$ to a limit one $\tau$ must be imposed even if these functions are constructed in a canonical way, as the  cosmological time or  volume functions in \cite{Galloway1998cosmologicaltime, Leo2024volume, Leo2026differentiability}. Specifically, consider  the class of spacetimes where such cosmological functions are $h$-steep Cauchy temporal functions. If a sequence of spacetimes converges in the Cheeger--Gromov sense to a limit, then the corresponding cosmological functions satisfy
$$\tau_{cosm_i}\not \rightarrow \tau_{cosm},$$
 in general. Indeed, from the very definition of Cheeger-Gromov convergence (Def.\ref{def:riemann_convergence}) it is clear that the convergence is independent of the value of  the metrics $g_i$ outside of the open subsets $V_i$, that can be chosen relatively compact. The cosmological functions depend, at each point $q_i$, on the value of the metric $g_i$ in $I^-(q_i)$, and the latter is not relatively compact. Thus, each metric $g_i$ can be modified arbitrarily in  $I^-(q_i)$ minus the closure of $V_i$ and the sequence will still converge. However, the cosmological functions of the sequence can be severely modified and  made non-convergent.

\end{remark}

\subsection{Reduction to a single product manifold} \label{subsec:diff_cauchy_hyp}

In the simple case where all the metrics lie on the same manifold, convergence on compact sets proves useful for establishing Cheeger–Gromov convergence and, even though Riemannian Cheeger-Gromov convergence does not require the converging manifolds to be diffeomorphic, the particular case when this occurs and suitable diffeomorphisms are found becomes especially interesting. Notice that in the globally hyperbolic case, the manifolds are diffeomorphic if and only if they admit diffeomorphic Cauchy hypersurfaces \cite{splitting}. 
We will next show that the Cauchy orthogonal splitting of such spacetimes \cite{BernalSanchez2005Splitting,zepp:structure} allows us to construct a common splitting $\R \times \Sigma$ for all the metrics, where the $\R$-projection corresponds to the image of {\em all} temporal functions in the sequence -- and, eventually, the limit as well.

Consider a sequence of spacetimes $(\bM_i, g_i)$, each endowed with an adapted $h$-steep temporal function $\tau_i$ and a Cauchy hypersurface $\Sigma_i = \tau_i^{-1} (0)$, all diffeomorphic to one another. We resolve this by showing that there exists a natural and intrinsic identification of each $(\bM_i,g_i)$ with a common product $\R \times \bSigma$, allowing all the metrics to be viewed as defined on a shared domain.\footnote{This point of view has been used to study other questions such as the non-convexity of the space of globally hyperbolic metrics \cite[Section~3.4]{sanchez:slicings}.}

\begin{proposition}
    \label{lemma:difeomorfismo_RxSigma}
	Let $(\bM, g)$ be a globally hyperbolic spacetime-with-timelike-boundary, and $\tau: \bM \to \R$ an adapted $h$-steep $C^{k+1}, k \geq 1$, temporal function with Cauchy hypersurfaces $C^k$-diffeomorphic to $\bSigma$. If $\phi : \tau^{-1} (0) \to \bSigma$ is a $C^k$ diffeomorphism then there exists a canonical $C^k$ diffeomorphism $\Phi: (\bM,g) \to \R \times \bSigma$ such that the following diagram commutes,
	\[
	\begin{tikzcd}
		& (\bM, g) \arrow[dr, "\tau" below] \arrow[r, "\Phi"] & \mathbb{R} \times \bSigma \arrow[d, "t"] \\
		& & \mathbb{R}
	\end{tikzcd}
	\]
	and the induced metric on the product can be written as $\Phi_* g = -\Lambda dt^2 + \sigma_t$, where $\sigma_t$ is a Riemannian metric on each slice $\{t\} \times \bSigma$.
\end{proposition}

\begin{proof}
	Given any point $p \in M$, let $p_0\in \tau^{-1} (0)$ be its projection along the integral curve of the gradient vector field $\nabla \tau$ passing through $p$. Define $x_p := \phi (p_0)$, where $\phi:\tau^{-1}(0)\rightarrow \bSigma$ is a chosen diffeomorphism. Then, we define a global diffeomorphism $\Phi : \bM \to \R \times \bSigma$ by
	\begin{equation*}
		\Phi (p) = (\tau (p), x_p).
	\end{equation*}
As $t$ is the projection on the first factor, $t (\Phi(p)) = \tau (p)$ for all $p$, and the diagram commutes.
	
	Notice that $\Phi_* (g)$ transfers the orthogonal splitting in Theorem \ref{theorem:splitting_boundary} to $\R \times \Sigma$ as well as the conformal factor.
\end{proof}

\begin{remark} \label{remark:cauchy_diffeomorphic}
	The previous proposition offers a natural framework for analyzing the convergence of globally hyperbolic spacetimes with diffeomorphic Cauchy hypersurfaces (which would  also apply to Geroch's  parametric case discussed in Remark \ref{r_Geroch}, supporting the approach for his study). In particular, it enables us to treat the sequence $\{(M_i,g_i)\}$ as a family of globally hyperbolic metrics defined on a common product manifold $\R \times \Sigma$, in a manner that preserves the splitting structure of both the Lorentzian metrics and their Wick-rotated Riemannian counterparts. Implicitly, all the adapted $h$-steep temporal functions on the individual manifolds are identified with the natural projection onto $\R \times \Sigma$. As a result, convergence of the induced metrics on compact subsets of $\R \times \Sigma$ also implies convergence in the sense of Definition \ref{def:lorentz_convergence}.
	
	Indeed, Corollary \ref{prop:equivalence_convergence} provides sufficient conditions for the sequence $(M_i, g_i, \tau_i)$ to converge -- up to diffeomorphisms -- provided the Cauchy hypersurfaces $\phi_i: \Sigma_i := \tau_i ^{-1} (0) \to \Sigma$ are diffeomorphic. Specifically, if:
	\begin{enumerate}[label=(\roman*)]
		\item the Wick-rotated metrics associated to the metrics $(\Phi_i^{-1})^* g_i$ converge to a complete Riemannian metric $h$, and
		\item the lapse functions $\Lambda_i$ converge uniformly to a function $\Lambda$,
	\end{enumerate}
	then the sequence $(M_i,g_i)$ converges to $(\R \times \Sigma, g)$ in the sense of Definition \ref{def:lorentz_convergence}, where $g := \Lambda (h- 2dt^2)$ defines a globally hyperbolic Lorentzian metric. 

\end{remark}

One could alternatively impose conditions analogous to those in compactness theorems directly on the Cauchy hypersurfaces $\Sigma^i_0 := \tau_i^{-1} (0)$, by requiring the sequence of Riemannian manifolds $(\Sigma^i_0, \sigma^i_0)$ -- where $\sigma^i_0$ denotes the induced metric -- to converge to a fixed Riemannian manifold $(\Sigma, \sigma_0)$. This convergence would, in turn, provide diffeomorphisms $\phi_i: \Sigma^i_0 \rightarrow \Sigma$ from each $i$.

\smallskip
\smallskip

\noindent 
This general possibility is useful to find sufficient conditions for 
Cheeger-Gromov convergence and has been used in simple cases. 
For example,  GRW spacetimes $I\times_f F, I\subset \R$ (see \cite{Romero1995GRW}) have been investigated using the null distance in \cite{AllenBurtscher2022Properties,Allen2023Nulldistanceandconvergence}, as well as within the framework of Lorentzian length spaces in \cite{Kunzinger2022Nulldistance}.
 Typically,  the convergence of the warping functions $f_i$ to a limit function $f_\infty$ in a fixed GRW splitting is taken as a standing assumption, and convergence is then considered in the sense of uniform convergence on compact sets. We emphasize, however, that this corresponds to a restrictive scenario within the broader framework of Cheeger-Gromov (intrinsic) convergence, which is capable of exploiting the presence of isometries to relate different representations of the same underlying geometry. In the next example, Cheeger–Gromov convergence  ``untangles'' divergences caused purely by changes of coordinates or isometries and  
 Cheeger-Gromov intrinsic compactness theorems (for Riemannian and Wick-rotated metrics) also rule out coordinate divergences.

\begin{example}[Different representations of de Sitter space]\label{example:de-Sitter}
        {\em Classical de Sitter spacetime admits more than one decomposition as a GRW spacetime.\footnote{   Guti\'errez and Olea \cite{Gutierrez2009GRW} proved that it has distinct timelike, closed, and conformal 
        vector fields that induce different GRW structures  (being the only GRW spacetime with this property).  They have the form $\mathbb{R} \times_f \mathbb{S}^{n-1} (\mu)$, with warping function $f(t) = \frac{r}{\mu} \cosh (\frac{1}{r} t + b)$, where $r, \mu \in \mb R^+$ and $b \in \mb R$.}
        Choosing them, it is easy to find sequences of different expressions of (the same) de Sitter spacetime as 
         GRW metrics which do not converge uniformly on compact sets  and, for instance,  \cite[Thm 1.4]{AllenBurtscher2022Properties} would not apply. However, they converge in the Cheeger-Gromov setting 
        (as in Theorem~\ref{thm:convergence_uptodiffeo}) 
        and the compactness Theorem \ref{thm:compacidad_globalmente_hiperbolicos} may be applied for the Wick-rotated metrics $g_W^{\tau_i}$ of the corresponding splittings. 
        For example,  the GRW decompositions of de Sitter $g_i = -dt^2 + \cosh^2(t+i) g_{\mb S^{n-1}}$,  with the $h$-steep temporal functions $\tau_i = t+i$ do not converge  on compact subsets. Trivially, Cheeger--Gromov convergence holds (pulling back each $g_i$ via the diffeomorphism $F_i (t,x) = (t-i,x)$) and this can also be achieved by using Proposition \ref{lemma:difeomorfismo_RxSigma}. Moreover, the corresponding Wick-rotated Riemannian metrics $g_W^{\tau_i} = dt^2 + \cosh^{2} (t+i) g_{\mathbb{S}^{n-1}}$ are complete, have injectivity radius bounded away from 0 and upper bounded Riemann tensor $R$. 
        }
\end{example}

\begin{appendices}

	\section{Independence of $h$-steep and steep  (Example \ref{ex:h-steep_no-steep})} \label{appendix:example2.2}
	An explicit construction of an $h$-steep non-steep 
temporal function as explained in the first part of Example \ref{ex:h-steep_no-steep} is provided.
    
	{\em Construction of $\tau$.} Consider the smoothed step function defined by
	$$s(t) = \left\lbrace\begin{array}{cc}
		\frac{\exp(-1/t)}{\exp(-1/t) + \exp(-1/(1-t))}, & 0<t<1
		\\
		0 & t \leq 0, \\
		1 & t \geq 1; \end{array}\right.$$
	with $s'(t) > 0$ for any $t\in ]0,1[$. To get a step in any real interval $[a,b]$ it is enough to consider the composition $t \mapsto s\left(\frac{t-a}{b-a} \right)$, which is also increasing in $]a,b[$. 

	For each $k \in \mathbb{N}$, denote $l_k(t) = \frac{t}{k^2}+k^2-1$ and define $\tau : \LL^2 \rightarrow \R$ by 
	\begin{equation*} 
		\tau(t,x) = \varphi(t) := \begin{cases} 
			t+\left(l_k(t)-t \right) s \left( \frac{t-(k^2-1/k^2)}{1/k^2} \right)  & \text{if } t \in [k^2-1/k^2,k^2), \\
			l_k(t)+\left(t-l_k (t)\right) s \left( \frac{t-k^2}{1/k^2} \right)  & \text{if } t \in [k^2,k^2+1/k^2], \\
			t & \text{otherwise}.
		\end{cases}
	\end{equation*}
	{\em $\tau$ is temporal but not steep.} Note $\frac{\partial \tau}{\partial t} (t,x)= \varphi' (t)$ with  
	\begin{equation*} 
		\varphi'(t) = \begin{cases} 
			1+\left(\frac{1}{k^2}-1 \right) s \left( \frac{t-(k^2-1/k^2)}{1/k^2} \right) + (l_k(t) - t) s' \left( \frac{t-(k^2-1/k^2)}{1/k^2} \right) k^2 & \text{if } t \in [k^2-1/k^2,k^2), \\
			\frac{1}{k^2}+\left(1-\frac{1}{k^2} \right) s \left( \frac{t-k^2}{1/k^2} \right) + (t-l_k (t)) s' \left( \frac{t-k^2}{1/k^2} \right) k^2 & \text{if } t \in [k^2,k^2+1/k^2], \\
			1 & \text{otherwise}.
		\end{cases}
	\end{equation*}
	A direct computation gives us that $0< 1/k^2 \leq \varphi ' (t) \leq 3$ for every $t \in [k^2-1/k^2, k^2 + 1/k^2]$, thus $\nabla \varphi = (-\varphi' (t), 0)$ and $\tau$ is temporal.

As $\frac{\partial \tau}{\partial t}|_{t = k^2} = 1/k^2$, the sequence $\{p_k = (k^2,0): k \in \mathbb{N}\}$ satisfies
\begin{equation*}
	g_{p_k} (\nabla \tau _{p_k}, \nabla \tau _{p_k}) = -1/k^4,
\end{equation*}
forbidding the steepness of $\tau$.

{\em $\tau$ is $h$-steep with respect to a complete Riemannian metric $h$ on $\mathbb{R}^2$.} Consider the metric $h = \frac{1}{2} \varphi' (t)^2 h_0$, where  $h_0 = dt^2 + dx^2$. Since $0 < 1/k^2 \leq \varphi ' \leq 3$ in each interval $[k^2-1/k^2,k^2+1/k^2]$, and any two of these intervals are separated by an interval of length greater than $4$, we get that $\frac{1}{2}\varphi ' (t)^2 h_0$ is complete. Let $v$ be a future-directed $g$-causal vector, that is,
\begin{equation*}
	-dt^2 (v) + dx^2 (v) \leq 0.
\end{equation*}
Since $d\tau = \varphi' (t) dt$ we obtain
\begin{equation*}
	\frac{d\tau ^2 (v)}{(\varphi ' (t)) ^2} = dt^2 (v) \geq dx^2 (v),\qquad 
	2\frac{d\tau ^2 (v)}{(\varphi ' (t)) ^2} = 2dt^2 (v) \geq h_0(v,v),\qquad 
	d\tau ^2 (v) \geq h (v,v),
\end{equation*}
as required.

\section{Construction of steep, $h$-steep, adapted functions 
} \label{appendix:thm_hsteep}

In this appendix we explain how to adapt the arguments in the proof by M\"uller and S\'anchez in \cite[Theorem 1.2]{muller-sanchez} and \cite[\S 5.2]{zepp:structure}  to get a proof of Theorem \ref{thm:GHnulldistance-conborde} (iii), namely, that any globally hyperbolic spacetime-with-timelike-boundary  $(\bM,g)$ admits an {\em adapted} Cauchy temporal function $\tau$ which is both {\em steep} and {\em $h$-steep} for any prescribed (complete) Riemannian metric $h$, thus, convenient for both Nash-type embeddings and Wick-rotation. The argument follows the main line of reasoning in \cite{muller-sanchez}, with suitable modifications to also guarantee $h$-steepness.

In the first subsection we construct a Cauchy temporal function $\tau$ on $\bM$ that is both steep and $h$-steep for any  Riemannian metric $h$, postponing the proof that it can be chosen with gradient tangent to the boundary in the second subsection.

\subsection{Construction of the temporal function} \label{subapp:construction}



In this subsection, we  first  present the notions and technical results that are relevant for our argument  but are absent from the original proof in \cite[Section 4]{muller-sanchez}, have been substantially modified, or require a somewhat different proof.  These ingredients—adapted from \cite[Section 4]{muller-sanchez} so as to ensure $h$-steepness at each step -- will be incorporated into the proof at the end of the subsection. We deliberately omit  the proof of statements whose extension to the boundary case is straightforward 
and focus instead on the subtler points and genuine differences introduced by our specific  problem.

\begin{lemma}[Lemma 4.1, \cite{muller-sanchez}]\label{lemma1:hsteep}
	Let $\tau$ be a function satisfying $g (\nabla \tau, \nabla \tau) < 0$ on a subset $U$, and let $K \subset U$ be a compact set. Then, for any function $f$, there exists a constant $c$ such that $g (\nabla (f+ c\tau), \nabla (f+ c\tau)) < -1$ on $K$.
\end{lemma}

%

Note first that, if $\tau: \bM \to \mathbb{R}$ is a temporal function then $\nabla \tau$ is past-pointing timelike. So, in our setting, \cite[Lemma 4.10]{sormanivega} yields the following statement:
\begin{lemma}\label{remark:h-steepcompactos}
	  For any Riemannian metric $h$ and any compact subset $K \subset \bM$, there exists a constant $c > 0$ such that the function $c \tau$ is $h$-steep on $K$.
\end{lemma}


%
%

A property that will be extensively used in our approach is that, being $h$-steep is preserved under the sum of functions:
\begin{lemma}\label{lemma:suma-es-hsteep}
	Let $\tau_1$ and $\tau_2$ be temporal functions with $\tau_1$ being $h$-steep. Then, $\tau_1 + \tau_2$ is an $h$-steep temporal function.
\end{lemma}

\begin{proof}
	Trivially the sum of temporal functions is itself temporal. Let $V \in T \bM$ be future-directed causal. Since $\tau_2$ is temporal, $g (\nabla \tau_2, V) > 0$ and, thus,
	\begin{equation*}
		g(\nabla (\tau_1 + \tau_2), V) = g (\nabla \tau_1 , V) + g(\nabla \tau_2, V) > g (\nabla \tau_1 , V)\geq \| V \|_h. 
	\end{equation*}
\end{proof}

The following technical result is also required for the argument.
\begin{lemma} \label{lemma:constant-hsteep}
	Let $h$ be a Riemannian metric and $\tau$ a function such that $g(\nabla \tau, \nabla \tau ) < 0$ and $d\tau (v) \geq \|v\|_h$ for every future causal vector $v$ in a subset $U$. Given a compact set $K \subset U$, for any function $f$ there exists a constant $c$ such that $d (f+c\tau) (v) \geq \| v \|_h$ for all future-directed causal vectors $v$ at every point in $K$. 
    Such a constant $c$  will  be always chosen  bigger if necessary so that \cite[Lemma 4.1]{muller-sanchez} holds too.

\end{lemma}

\begin{proof}
	Notice that, at each $x\in K$, 
	$$d (f+c\tau) (v)=g (\nabla (f + c\tau) , v) = g (\nabla f, v) + c g (\nabla \tau, v) = df(v) + cd\tau (v) \geq df(v) + c\|v\|_h.$$
	{Since $K$ is compact,} the constant $c>0$ can be chosen large  so that  $ df(v) + c\|v\|_h\geq\| v \|_h.$ 
\end{proof}

%

Now, one can prove this refined version of \cite[Proposition 4.2]{muller-sanchez}. 

\begin{proposition}\label{prop1:hsteep}
	Let $S$ be a Cauchy hypersurface and $p \in J^{-} (S)$. For every neighborhood $V$ of $J (p,S)$, there exists a smooth function $\tau \geq 0$ such that:
	\begin{enumerate}[label=(\roman*)]
		\item Supp$(\tau) \subset V$,
		\item $\tau > 1$ on $S \cap J^{+} (p)$,
		\item $\nabla \tau$ is timelike and past-pointing in Int(Supp$(\tau) \cap J^{-} (S))$,
		\item $g(\nabla\tau,\nabla\tau)<-1$  on $J(p,S)$  
         \item $d\tau (v) \geq \| v\|_h$ for all  future-directed  causal $v \in T\bM$   on $J(p,S)$.
	\end{enumerate} 
\end{proposition}

\begin{proof}
	The proof of this result follows  the argument in \cite[Proposition 4.2]{muller-sanchez}, where the properties {\em (i)--(iv)} are obtained. The relevant  difference is that, at each step where a constant $c$ was chosen therein to ensure steepness of the temporal function using \cite[Lemma 4.1]{muller-sanchez}, one must now consider a possibly larger constant, applying  Lemma 
    \ref{lemma:constant-hsteep} to guarantee that the function is also $h$-steep. Then, Lemma \ref{lemma:suma-es-hsteep} also ensures the $h$-steepness of the functions obtained at all the stages of the proof, yielding {\em (v)}.
\end{proof}

Following the line of reasoning in \cite{muller-sanchez}, we employ {\em fat cone coverings} of Cauchy hypersurfaces \cite[Def 4.3]{muller-sanchez} to extend the locally defined temporal functions from Proposition \ref{prop1:hsteep} into a global temporal function. For this purpose, the following special class of functions plays a role analogous to that of the {\em steep forward cone functions} in \cite[Definition 4.5]{muller-sanchez}.




\begin{definition}\label{def:hSFC}
	Let $p',p \in T_{a-1}^{a}, p' \ll p$ (as in \cite{muller-sanchez}, $S_a$ denotes the level $t=a$ of a fixed Cauchy temporal function $t$ and $T^b_a:=t^{-1}([a,b])$). An {\em $h$-steep and steep forward cone (hSFC) function} for $(a,p',p)$ is a smooth function $h^{+}_{a,p',p} : \bM \to [0,\infty)$ which satisfies the following conditions:
	\begin{enumerate}[label=(\roman*)]
		\item Supp$(h^{+}_{a,p',p}) \subset J(p', S_{a+2})$;
		\item $h^{+}_{a,p',p} > 1$ on $S_{a+1} \cap J^{+} (p)$;
		\item If $x \in J^{-} (S_{a+1})$ and $h^{+}_{a,p',p} (x) \neq 0$ then $\nabla h^{+}_{a,p',p} (x)$ is timelike and past-pointing;
        \item $g(\nabla h^{+}_{a,p',p}, \nabla h^{+}_{a,p',p})<-1$ on $J(p,S_{a+1})$.
		\item for any future-directed causal $v \in T\bM$, $dh^{+}_{a,p',p} (v) \geq \| v\|_h$ on $J(p,S_{ a+1})$.
	\end{enumerate}
\end{definition}

\begin{proposition}\label{prop:hSFC}
	For every $(a,p',p)$, there exists an hSFC function.
\end{proposition}

\begin{proof}
	Apply Proposition \ref{prop1:hsteep} to $S = S_{a+1}$ and $V = I^{+} (p') \cap I^{-} (S_{a+2}) \subset J(p',S_{a+2})$.
\end{proof}

The following lemma allows us to globalize the properties of an hSFC function by using fat cone coverings. Even if this fact is straightforward, we include the proof to highlight the differences with the original version in \cite[Lemma 4.7]{muller-sanchez}.

\begin{lemma}\label{lemma3:hsteep}
	Choose $a\in \mathbb{R}$ and take any fat cone covering $\{p_i ' \ll p_i : i\in \mathbb{N} \}$ for $S = S_a$. For every sequence $\{c_i \geq 1 : i \in \mathbb{N} \}$, the non-negative function $h_{a}^{+} = (|a| + 1)\sum_{i} c_i h^{+}_{a,p_i ',p_i}$ satisfies:
	\begin{enumerate}[label=(\roman*)]
		\item Supp$(h^{+}_{a}) \subset J(S_{a-1}, S_{a+2})$.
		\item $h^{+}_{a} > |a|+ 1$ on $S_{a+1}$.
		\item If $x \in J^{-} (S_{a+1})$ and $h^{+}_{a} (x) \neq 0$ then $\nabla h^{+}_{a} (x)$ is timelike and past-pointing.
            \item $g(\nabla h^{+}_{a}, \nabla h^{+}_{a})<-1$ on $J(S_{a},S_{a+1})$.
		\item For any future-directed causal $v \in T\bM$, $dh^{+}_{a} (v) \geq \| v\|_h$ on $J(S_{a},S_{a+1})$.
	\end{enumerate} 
\end{lemma}

\begin{proof}
{\em	(i)} Each hSFC function $h^{+}_{a,p_i ',p_i}$ has its support included in $J(p_i ', S_{a+2}) \subset J(S_{a-1}, S_{a+2})$.
	
{\em 	(ii)} Since $c_i \geq 1$, it follows that $h^{+}_{a,p_i',p_i} > 1$ on $S_{a+1} \cap J^{+} (p_i)$. Moreover,  the corresponding coverings $\mathcal{C}$ and $\mathcal{C}'$ also cover $S_{a+1}$ because the family $\{p_i ' \ll p_i\}$ form a fat cone covering of $S_a$. 

{\em (iii)-(iv) } Observe that
	$$\nabla h^{+}_{a} = (|a|+1) \sum_i c_i \nabla h^{+}_{a,p_i',p_i},$$
	where $h^{+}_{a} \neq 0$ if and only if $h^{+}_{a,p_i',p_i} \neq 0$ for some $i$. In that case, $\nabla h^{+}_{a,p_i',p_i}$ is timelike and past-directed. Then, the fat cone property of  $\{p_i'\ll p_i\}$ for $S_a$ 
    ensures that the regions where the 
hSFC are steep cover $J(S_{a},S_{a+1})$.

{\em (v) } Using the identity
	$$g( \nabla h^{+}_{a} , v ) = (|a|+1) \sum_i c_i g( \nabla h^{+}_{a,p_i',p_i} , v), $$
	and noting that each term in the sum satisfies
	$$g( \nabla h^{+}_{a,p_i',p_i}, v) \geq \| v \|_{h}\quad\hbox{on $J(p_i, S_{a+1})$}\quad \hbox{(i.e. each $h^{+}_{a,p_i',p_i}$ is $h$-steep)},$$
	it follows that the full sum is also $h$-steep by Lemma \ref{lemma:suma-es-hsteep}. 
    Again, the claim follows from the fat cone property of
     $\{p_i'\ll p_i\}$ for $S_a$.
\end{proof}


Since the gradient of $h^{+}_{a}$ becomes spacelike in some subset of $J(S_{a+1},S_{a+2})$, it is necessary to strengthen the previous lemma in order to construct the final temporal function.

\begin{lemma}\label{lemma4:hsteep}
	Let $h^{+}_{a} \geq 0$ be as in Lemma \ref{lemma3:hsteep}. Then, there exists a function $h^{+}_{a+1}$ satisfying all the properties stated in Lemma \ref{lemma3:hsteep} for $S = S_{a+1}$, and such that the sum $h^{+}_{a} + h^{+}_{a+1}$ is steep and  $h$-steep on 
    $J ( S_{a+1}, S_{a+2})$, thus, 
    $g( \nabla (h^{+}_{a} + h^{+}_{a+1}) , \nabla (h^{+}_{a} + h^{+}_{a+1})) <-1$ and, for every future causal vector $v \in T\bM$,
	\begin{equation*}
		g( \nabla (h^{+}_{a} + h^{+}_{a+1}) , v) \geq \| v\|_h, \quad \text{ on }\;\; J(S_{a+1}, S_{a+2}).
	\end{equation*}
\end{lemma}

\begin{proof}
	Consider a fat cone covering $\{p_i ' \ll p_i\}$ of $S = S_{a+1}$ (cf. \cite[Prop 4.4]{muller-sanchez}). By \cite[Lemma 4.1]{muller-sanchez}, for each $p_i$ there exists a constant $c_i \geq 1$ such that the function $c_i h^{+}_{a+1,p_i',p_i} + h^{+}_{a}$ is steep on $J(p_i, S_{a+2})$. Furthermore, applying Lemma \ref{remark:h-steepcompactos} and Lemma \ref{lemma:constant-hsteep}, we can choose larger constants $c_i$ if necessary so that $c_i h^{+}_{a+1,p_i',p_i}$ is $h$-steep, and consequently, the sum $c_i h^{+}_{a+1,p_i',p_i} + h^{+}_{a}$ remains $h$-steep on $J(p_i, S_{a+2})$ by Lemma \ref{lemma:suma-es-hsteep}. The desired function is then defined as $h^{+}_{a+1} = (|a|+2)\sum_{i} c_i h^{+}_{a+1,p_i',p_i}$.
\end{proof}

With these additional ingredients in place, the proof of the main theorem proceeds along the lines of \cite[Theorem 1.2]{muller-sanchez}, with the key difference that we incorporate $h$-steep temporal functions. Lemma \ref{lemma:suma-es-hsteep} guarantees that the sum of such functions remains $h$-steep. Starting with the function $h^{+}_0$ from Lemma \ref{lemma3:hsteep}, we apply Lemma \ref{lemma4:hsteep} inductively for $a = n \in \mathbb{N}$, obtaining
\[
\tau^+=\sum_{n=0}^\infty h_n^+\geq 0,
\]
which is a steep and $h$-steep temporal function on $J^{+}(S_0)$ with support contained in $J^{+} (S_{-1})$. Reversing the time orientation yields $\tau^{-} \geq 0$, an $h$-steep temporal function on $J^{-} (S_0)$. We then define $$\tau = \tau^{+} - \tau^{-},$$ which is a steep and $h$-steep temporal function on the entire manifold $\bM$. The fact that $\tau$ is Cauchy follows exactly as in \cite{muller-sanchez}.


\subsection{Adaptability to the boundary}

To ensure that $\nabla \tau$ is tangent to $\partial M$, so that $\partial M$ remains orthogonal to all the $\tau$-slices, we can follow a strategy similar to that used in \cite[\S 5.2]{zepp:structure} (see also \cite[\S 5.4.2]{sanchez:slicings}).

Consider two copies of $\bM$ and identify homologous points along the boundary $\partial M$ to form the double manifold $\bM ^d$, which inherits the structure of a smooth manifold (without boundary). Next,  extend $g$ smoothly to a metric $g^d$ on $\bM ^d$ in such a way that the natural reflection $r : \bM ^d \to \bM ^d$ becomes an isometry. This can be achieved by choosing a metric $g^* > g$ as a product around $\partial M$ (as in the proof of \cite[Theorem 1.1]{zepp:structure}, see also Proposition~5.2 and Lemma 5.4 in that reference).
Notice also that $g^d$ will be globally hyperbolic since any pair of homologous Cauchy hypersurfaces with boundary of $(\bM^d , g^d)$ merges into a continuous Cauchy hypersurface without boundary. Also, extend $h$ to a complete Riemannian metric $h^d$ on $\bM ^d$ such that $r$ is an isometry (e.g. $h^d = \hat h + r^* \hat h$ for any extension $\hat h$ of $h$ to $\bM^d$).

Using the procedure described in the previous subsection, construct a Cauchy temporal function $\tau^d$ that is steep and $h^d$-steep for $(\bM ^d, g^d)$ (\S \ref{subapp:construction}). Then, define an $r$-invariant temporal function $\btau$ by
$$\btau = \tau^d + r^* \tau^d.$$
Since $r$ is an isometry for both $g^d$ and $h^d$, the function $\btau$ is also steep and $h^d$-steep\footnote{Alternatively, the procedure used in \cite[Proposition 13]{Muller2016invarianttemporal} to obtain a steep function invariant by a compact group of conformal diffeomorphisms could be used at this stage. It is worth pointing out that steep and $h$-steep functions invariant by such a group can be also obtained.}.
The restriction $\tau := \btau |_{\bM}$ is a Cauchy temporal function on $\bM$ that is steep and $h$-steep. Moreover, since $\btau$ is invariant under the reflection $r$, the $g$-gradient of $\tau$ is tangent to $\partial M$.

\section{Example on $d_W^\tau$ and the encoding of causality} \label{app:no_codifica}
The following example shows that, in general, the Wick distance $d_W^\tau$ does not encode causality and, thus, its modification into $\hat d_W^\tau$ would be necessary for this purpose (Proposition~\ref{prop:d_W_encoding}).

Consider $\mathbb{R}^2$ with coordinates $(t,x)$ endowed with a metric of the form
$$g = -dt^2 + f^2 (t,x) dx^2, \qquad g_W^{\tau} = dt^2 + f^2 (t,x) dx^2, $$
where $f:\mathbb{R}^2 \rightarrow \mathbb{R}$ is given by
\begin{equation*}
	f(t,x) = \begin{cases}
		1 & \text{if } (t,x) \notin U:= \{(t,x) \in \mathbb{R}^2: -0.1 \leq t \leq 0.1,\;\; 0.2 \leq x \leq 1.25 \}, \\
		\frac{1}{1000} & \text{if }  -0.06 \leq t \leq 0.06 \text{ and }  0.2 \leq x \leq 1.25, \\
		> 0 & \text{otherwise}.
	\end{cases}
\end{equation*}
\begin{figure}[ht!]
	\centering
	\includegraphics[scale=0.5]{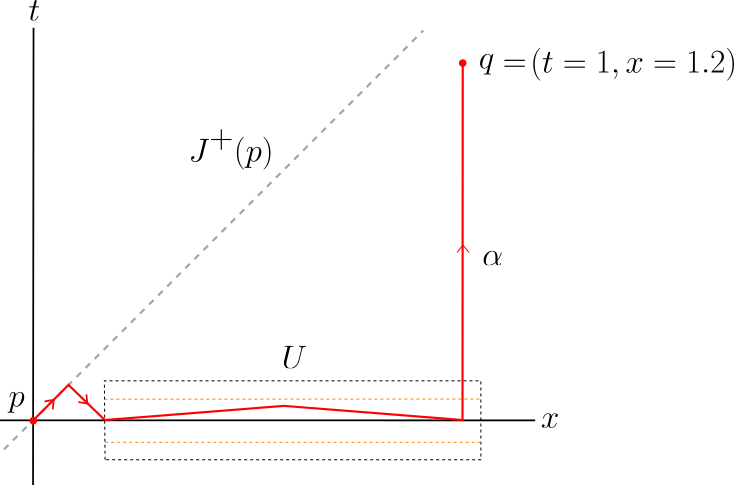}
	\caption{The curve $\alpha$ is piecewise causal and has $L_W^\tau$-length less than $\sqrt{2} (\tau (q) - \tau(p))$, yet $p \nleq q$.}
	\label{fig:pw-causal}
\end{figure}
Next, consider the points $p = (0,0)$ and $q = (1,1.2)$, and the curve $\alpha : [0,2.2] \to \mathbb{R}^2$ from $p$ to $q$ defined as follows (see Figure \ref{fig:pw-causal}):
\begin{equation*}
	\alpha(s) := \begin{cases}
		(s,s) &  0 \leq s < 0.1, \\
		(-s+0.2,s) & 0.1 \leq s < 0.2, \\
		(\frac{1}{\sqrt{1000}}(s-0.2),s) & 0.2 \leq s < 0.7, \\
		(\frac{1}{\sqrt{1000}}(-s+1.2),s) & 0.7  \leq s < 1.2, \\
		(s-1.2,1.2) & 1.2 \leq s \leq 2.2.
	\end{cases}
\end{equation*}
Note that
\begin{equation*}
	d_W^\tau (p,q) \leq L_W^\tau( \alpha) \leq 0.2 \sqrt{2} + \frac{\sqrt{2}}{\sqrt{1000}} + 1 < \sqrt{2}.
\end{equation*}
Thus, $p,q$ satisfy $d_W^\tau (p,q) < \sqrt{2} (t(q) - t(p))$, yet clearly $p \nleq q$.

\end{appendices}

\section*{Acknowledgements}
The authors acknowledge Prof. A. Zeghib (ENS Lyon) for pointing out the reference \cite{Mounoud2003} and further explanations on it (leading to Example \ref{ejemplo2}). We thank A. Burtscher for her careful reading of the manuscript, including pointing out a gap in a previous version among other improvements. Comments by A. Sakovich and C. Sormani (including a correction in the statement of Proposition \ref{thm:leonardo_codificacion}) and by L. Garc\'{\i}a-Heveling (including pointing out references \cite{Galloway2024encoding, Geroch}) are also warmly acknowledged.

The authors are partially supported by the IMAG–María de Maeztu grant CEX2020-001105-M (funded by MCIN/AEI/10.13039/50110001103), which includes a doctoral grant for the first named one, and by the project PID2024-156031NB-I00. In addition, the first and third named authors are partially supported by the project PID2020-116126GB-I00 (MCIN/AEI/10.13039/501100011033), while the second author is partially supported by the grant PID2020-118452GB-I00 (Spanish MICINN) and the project PY20-01391 (PAIDI 2020, Junta de Andalucía–FEDER). \

The first named author also acknowledges the hosting at U. Hamburg for a stay supervised by Prof. M. Graf. The first and second named authors also thank the support by the Austrian Science Fund (FWF) [Grant DOI 10.55776/EFP6] to attend a meeting at U. Vienna.

%

\bibliographystyle{plain}
\bibliography{referencias1}

@article{visser2017wickrotate,
      title={How to Wick rotate generic curved spacetime}, 
      author={Visser, M.},
      year={2017},
      eprint={1702.05572},
      archivePrefix={arXiv},
      primaryClass={gr-qc}, 
}

@article {samuel2016wickrotation,
    AUTHOR = {Samuel, J.},
     TITLE = {Wick rotation in the tangent space},
   JOURNAL = {Classical Quantum Gravity},
  FJOURNAL = {Classical and Quantum Gravity},
    VOLUME = {33},
      YEAR = {2016},
    NUMBER = {1},
     PAGES = {015006, 9},
}

@article{meco2025uniformtemplecharts,
	author       = {Meco, B. and Sakovich, A. and Sormani, C.},
	title        = {Existence of uniform {Temple} charts and applications to null distance},
	year         = {2025},
	eprint       = {2509.11281},
	archivePrefix = {arXiv},
	primaryClass  = {math.DG},
}

@article{Leo2026differentiability,
 author = {Garc{\'{\i}}a-Heveling, L.},
 title = {Differentiability and other properties of the cosmological volume function},
 fjournal = {Letters in Mathematical Physics},
 journal = {Lett. Math. Phys.},
 issn = {0377-9017},
 volume = {116},
 number = {3},
 pages = {20},
 year = {2026},
 doi = {10.1007/s11005-026-02112-5},
 zbMATH = {8229720}
}

@article {Leo2024volume,
    AUTHOR = {Garc\'ia-Heveling, Leonardo},
     TITLE = {Volume singularities in general relativity},
   JOURNAL = {Lett. Math. Phys.},
  FJOURNAL = {Letters in Mathematical Physics},
    VOLUME = {114},
      YEAR = {2024},
    NUMBER = {3},
     PAGES = {Paper No. 71, 20},
      ISSN = {0377-9017,1573-0530},
   MRCLASS = {83C75 (53C50)},
  MRNUMBER = {4751750},
MRREVIEWER = {Roberto\ Giamb\`o},
       DOI = {10.1007/s11005-024-01814-y},
       URL = {https://doi.org/10.1007/s11005-024-01814-y},
}

@phdthesis{Cheeger1967,
  author = {Cheeger, J.},
  title = {Comparison and Finiteness Theorems for Riemannian Manifolds},
  school = {Princeton University},
  year = {1967},
  type = {Ph.D. Thesis},
  note = {ProQuest Dissertations \& Theses ID: 302249825}}

@article{gromov-revolutionary,
 author = {Gromov, M.},
 title = {Almost flat manifolds},
 fjournal = {Journal of Differential Geometry},
 journal = {J. Differ. Geom.},
 issn = {0022-040X},
 volume = {13},
 pages = {231--241},
 year = {1978},
 language = {English},
 doi = {10.4310/jdg/1214434488},
 zbMATH = {3673118},
 Zbl = {0432.53020}
}

@article{Cheeger70,
 author = {Cheeger, J.},
 title = {Finiteness theorems for {Riemannian} manifolds},
 fjournal = {American Journal of Mathematics},
 journal = {Am. J. Math.},
 issn = {0002-9327},
 volume = {92},
 pages = {61--74},
 year = {1970},
 language = {English},
 doi = {10.2307/2373498},
 zbMATH = {3309823},
 Zbl = {0194.52902}
}

@article{Cheeger2010,
 author = {Cheeger, J.},
 title = {Structure theory and convergence in {Riemannian} geometry},
 fjournal = {Milan Journal of Mathematics},
 journal = {Milan J. Math.},
 issn = {1424-9286},
 volume = {78},
 number = {1},
 pages = {221--264},
 year = {2010},
 language = {English},
 doi = {10.1007/s00032-010-0117-4},
 zbMATH = {5853589},
 Zbl = {1205.53034}
}

@article{Galloway2024encoding,
 author = {Galloway, G. J.},
 title = {A note on null distance and causality encoding},
 fjournal = {Classical and Quantum Gravity},
 journal = {Classical Quantum Gravity},
 issn = {0264-9381},
 volume = {41},
 number = {1},
 pages = {5},
 note = {Id/No 017001},
 year = {2024},
 language = {English},
 doi = {10.1088/1361-6382/ad11a6},
 zbMATH = {7787384},
 Zbl = {1535.83086}
}

@article{Geroch,
 author = {Geroch, R. P.},
 title = {Limits of spacetimes},
 fjournal = {Communications in Mathematical Physics},
 journal = {Commun. Math. Phys.},
 issn = {0010-3616},
 volume = {13},
 pages = {180--193},
 year = {1969},
 language = {English},
 doi = {10.1007/BF01645486},
 zbMATH = {7756072},
 Zbl = {1522.83026}
}

@book{BEE,
 author = {Beem, John K. and Ehrlich, Paul E. and Easley, Kevin L.},
 title = {Global {Lorentzian} geometry.},
 edition = {2nd ed.},
 fseries = {Pure and Applied Mathematics, Marcel Dekker},
 series = {Pure Appl. Math., Marcel Dekker},
 volume = {202},
 isbn = {0-8247-9324-2},
 year = {1996},
 publisher = {New York, NY: Marcel Dekker},
 language = {English},
 zbMATH = {862955},
 Zbl = {0846.53001}
}

@article{HKM,
 author = {Hawking, S. W. and King, A. R. and McCarthy, P. J.},
 title = {A new topology for curved space-time which incorporates the causal, differential, and conformal structures},
 fjournal = {Journal of Mathematical Physics},
 journal = {J. Math. Phys.},
 issn = {0022-2488},
 volume = {17},
 pages = {174--181},
 year = {1976},
 language = {English},
 doi = {10.1063/1.522874},
 url = {authors.library.caltech.edu/11027/},
 zbMATH = {3500682},
 Zbl = {0319.54005}
}

@article{CMG,
 author = {Chru{\'s}ciel, P. T. and Grant, J. D. E. and Minguzzi, E.},
 title = {On differentiability of volume time functions},
 fjournal = {Annales Henri Poincar{\'e}},
 journal = {Ann. Henri Poincar{\'e}},
 issn = {1424-0637},
 volume = {17},
 number = {10},
 pages = {2801--2824},
 year = {2016},
 language = {English},
 doi = {10.1007/s00023-015-0448-3},
 zbMATH = {6644375},
 Zbl = {1351.83006}
}

@article{NY,
 author = {Pawel, K. and Reckziegel, H.},
 title = {Affine submanifolds and the theorem of {Cartan}-{Ambrose}-{Hicks}.},
 fjournal = {Kodai Mathematical Journal},
 journal = {Kodai Math. J.},
 issn = {0386-5991},
 volume = {25},
 number = {3},
 pages = {341--356},
 year = {2002},
 language = {English},
 doi = {10.2996/kmj/1071674466},
 zbMATH = {1910506},
 Zbl = {1065.53021}
}

@article{Romero1995GRW,
	author = {Al{\'{\i}}as, L. J. and Romero, A. and S{\'a}nchez, M.},
	title = {Uniqueness of complete spacelike hypersurfaces of constant mean curvature in generalized {Robertson}-{Walker} spacetimes},
	fjournal = {General Relativity and Gravitation},
	journal = {Gen. Relativ. Gravitation},
	issn = {0001-7701},
	volume = {27},
	number = {1},
	pages = {71--84},
	year = {1995},
	language = {English},
	doi = {10.1007/BF02105675},
	zbMATH = {727347},
	Zbl = {0908.53034}
}

@article{Mounoud2003,
	author = {Mounoud, P.},
	title = {Dynamical properties of the space of {Lorentzian} metrics},
	fjournal = {Commentarii Mathematici Helvetici},
	journal = {Comment. Math. Helv.},
	issn = {0010-2571},
	volume = {78},
	number = {3},
	pages = {463--485},
	year = {2003},
	language = {English},
	doi = {10.1007/s00014-003-0767-8},
	zbMATH = {2021080},
	Zbl = {1033.58012}
}

@Book{Petersen,
	Author = {Petersen, P.},
	Title = {Riemannian geometry},
	Edition = {2nd ed.},
	FSeries = {Graduate Texts in Mathematics},
	Series = {Grad. Texts Math.},
	ISSN = {0072-5285},
	Volume = {171},
	ISBN = {0-387-29246-2},
	Year = {2006},
	Publisher = {New York, NY: Springer},
	Language = {English},
	DOI = {10.1007/978-0-387-29403-2},
	zbMATH = {5066371},
	Zbl = {1220.53002}
}

@Article{Levichev1987causalstructure,
	Author = {Levichev, A. V.},
	Title = {Prescribing the conformal geometry of a {Lorentz} manifold by means of its causal structure},
	FJournal = {Soviet Mathematics. Doklady},
	Journal = {Sov. Math., Dokl.},
	ISSN = {0197-6788},
	Volume = {35},
	Pages = {452--455},
	Year = {1987},
	Language = {English},
	zbMATH = {4038718},
	Zbl = {0637.53077}
}

@Misc{SakovichSormani2024Variousnotions,
	Author = {Sakovich, A. and Sormani, C.},
	Title = {Introducing {Various} {Notions} of {Distances} between {Space}-{Times}},
	Year = {2024},
	HowPublished = {Preprint, {arXiv}:2410.16800 [math.{DG}] (2024)},
	URL = {https://arxiv.org/abs/2410.16800},
	arXiv = {arXiv:2410.16800}
}

@Article{Galloway1998cosmologicaltime,
	Author = {Andersson, L. and Galloway, G. J. and Howard, R.},
	Title = {The cosmological time function},
	FJournal = {Classical and Quantum Gravity},
	Journal = {Classical Quantum Gravity},
	ISSN = {0264-9381},
	Volume = {15},
	Number = {2},
	Pages = {309--322},
	Year = {1998},
	Language = {English},
	DOI = {10.1088/0264-9381/15/2/006},
	zbMATH = {1221026},
	Zbl = {0911.53039}
}

@Article{miguel2005revision,
	Author = {S{\'a}nchez, M.},
	Title = {Causal hierarchy of spacetimes, temporal functions and smoothness of {Geroch}'s splitting. {A} revision},
	FJournal = {Matem{\'a}tica Contempor{\^a}nea},
	Journal = {Mat. Contemp.},
	ISSN = {0103-9059},
	Volume = {29},
	Pages = {127--155},
	Year = {2005},
	Language = {English},
	zbMATH = {5031453},
	Zbl = {1094.83004}
}

@Article{Burtscher2015Length,
	Author = {Burtscher, A. Y.},
	Title = {Length structures on manifolds with continuous {Riemannian} metrics},
	FJournal = {The New York Journal of Mathematics},
	Journal = {New York J. Math.},
	ISSN = {1076-9803},
	Volume = {21},
	Pages = {273--296},
	Year = {2015},
	Language = {English},
	zbMATH = {6456672},
	Zbl = {1342.53054}
}

@Article{garciaparrado-sanchez,
	Author = {Garc{\'{\i}}a-Parrado, A. and S{\'a}nchez, M.},
	Title = {Further properties of causal relationship: causal structure stability, new criteria for isocausality and counterexamples},
	FJournal = {Classical and Quantum Gravity},
	Journal = {Classical Quantum Gravity},
	ISSN = {0264-9381},
	Volume = {22},
	Number = {21},
	Pages = {4589--4619},
	Year = {2005},
	Language = {English},
	DOI = {10.1088/0264-9381/22/21/009},
	zbMATH = {5011452},
	Zbl = {1087.83006}
}

@Article{Muller2016invarianttemporal,
	Author = {M{\"u}ller, O.},
	Title = {A note on invariant temporal functions},
	FJournal = {Letters in Mathematical Physics},
	Journal = {Lett. Math. Phys.},
	ISSN = {0377-9017},
	Volume = {106},
	Number = {7},
	Pages = {959--971},
	Year = {2016},
	Language = {English},
	DOI = {10.1007/s11005-016-0850-x},
	zbMATH = {6674474},
	Zbl = {1356.53068}
}

@Misc{Muller2024gromovhausdorff,
	Author = {M{\"u}ller, O.},
	Title = {Gromov-{Hausdorff} metrics and dimensions of {Lorentzian} length spaces},
	Year = {2022},
	HowPublished = {Preprint, {arXiv}:2209.12736 [math.{DG}] (2022)},
	URL = {https://arxiv.org/abs/2209.12736},
	arXiv = {arXiv:2209.12736}
}

@Article{Sakovich2023encodescausality,
	Author = {Sakovich, A. and Sormani, C.},
	Title = {The null distance encodes causality},
	FJournal = {Journal of Mathematical Physics},
	Journal = {J. Math. Phys.},
	ISSN = {0022-2488},
	Volume = {64},
	Number = {1},
	Pages = {18},
	Note = {Id/No 012502},
	Year = {2023},
	Language = {English},
	DOI = {10.1063/5.0118979},
	zbMATH = {7672652},
	Zbl = {1511.83004}
}

@mastersthesis{bamler2007ricci,
	author       = {Bamler, R.},
	title        = {Ricci Flow with Surgery},
	school       = {Ludwig-Maximilians-Universit{\"a}t Munich},
	year         = {2007},
	type         = {Diploma Thesis},
}

@Article{Gutierrez2009GRW,
	Author = {Guti{\'e}rrez, M. and Olea, B.},
	Title = {Global decomposition of a {Lorentzian} manifold as a generalized {Robertson}-{Walker} space},
	FJournal = {Differential Geometry and its Applications},
	Journal = {Differ. Geom. Appl.},
	ISSN = {0926-2245},
	Volume = {27},
	Number = {1},
	Pages = {146--156},
	Year = {2009},
	Language = {English},
	DOI = {10.1016/j.difgeo.2008.06.015},
	zbMATH = {5509733},
	Zbl = {1157.53036}
}

@Article{BernalSanchez2005Splitting,
	Author = {Bernal, A. N. and S{\'a}nchez, M.},
	Title = {Smoothness of time functions and the metric splitting of globally hyperbolic spacetimes},
	FJournal = {Communications in Mathematical Physics},
	Journal = {Commun. Math. Phys.},
	ISSN = {0010-3616},
	Volume = {257},
	Number = {1},
	Pages = {43--50},
	Year = {2005},
	Language = {English},
	DOI = {10.1007/s00220-005-1346-1},
	zbMATH = {2237027},
	Zbl = {1081.53059}
}

@Article{Allen2023Nulldistanceandconvergence,
	Author = {Allen, B.},
	Title = {Null distance and {Gromov}-{Hausdorff} convergence of warped product spacetimes},
	FJournal = {General Relativity and Gravitation},
	Journal = {Gen. Relativ. Gravitation},
	ISSN = {0001-7701},
	Volume = {55},
	Number = {10},
	Pages = {34},
	Note = {Id/No 118},
	Year = {2023},
	Language = {English},
	DOI = {10.1007/s10714-023-03167-8},
	zbMATH = {7772554},
	Zbl = {1534.83069}
}

@Article{AllenBurtscher2022Properties,
	Author = {Allen, B. and Burtscher, A.},
	Title = {Properties of the null distance and spacetime convergence},
	FJournal = {IMRN. International Mathematics Research Notices},
	Journal = {Int. Math. Res. Not.},
	ISSN = {1073-7928},
	Volume = {2022},
	Number = {10},
	Pages = {7729--7808},
	Year = {2022},
	Language = {English},
	DOI = {10.1093/imrn/rnaa311},
	zbMATH = {7542572},
	Zbl = {1507.83083}
}

@InCollection{Anderson2004CheegerGromov,
	Author = {Anderson, M. T.},
	Title = {Cheeger-{Gromov} theory and applications to general relativity},
	BookTitle = {The Einstein equations and the large scale behavior of gravitational fields. 50 Years of the Cauchy problem in general relativity. With DVD},
	ISBN = {3-7643-7130-7},
	Pages = {347--377},
	Year = {2004},
	Publisher = {Basel: Birkh{\"a}user},
	Language = {English},
	zbMATH = {2128102},
	Zbl = {1064.83004}
}

@Article{Bernard2020Cauchyanduniform,
	Author = {Bernard, P. and Suhr, S.},
	Title = {Cauchy and uniform temporal functions of globally hyperbolic cone fields},
	FJournal = {Proceedings of the American Mathematical Society},
	Journal = {Proc. Am. Math. Soc.},
	ISSN = {0002-9939},
	Volume = {148},
	Number = {11},
	Pages = {4951--4966},
	Year = {2020},
	Language = {English},
	DOI = {10.1090/proc/15106},
	URL = {basepub.dauphine.fr/handle/123456789/19411},
	zbMATH = {7243361},
	Zbl = {1446.53054}
}

@Book{Chow2007ricciflow,
	Author = {B. Chow et. al.},
	Title = {The {Ricci} flow: techniques and applications. {Part} {I}: {Geometric} aspects},
	FSeries = {Mathematical Surveys and Monographs},
	Series = {Math. Surv. Monogr.},
	ISSN = {0076-5376},
	Volume = {135},
	ISBN = {978-0-8218-3946-1},
	Year = {2007},
	Publisher = {Providence, RI: American Mathematical Society (AMS)},
	Language = {English},
	zbMATH = {5149301},
	Zbl = {1157.53034}
}

@Article{hamilton95a,
	Author = {Hamilton, R. S.},
	Title = {A compactness property for solutions of the {Ricci} flow},
	FJournal = {American Journal of Mathematics},
	Journal = {Am. J. Math.},
	ISSN = {0002-9327},
	Volume = {117},
	Number = {3},
	Pages = {545--572},
	Year = {1995},
	Language = {English},
	DOI = {10.2307/2375080},
	zbMATH = {807387},
	Zbl = {0840.53029}
}

@Book{HopperAndrews2011ricciflow,
	Author = {Andrews, B. and Hopper, C.},
	Title = {The {Ricci} flow in {Riemannian} geometry. {A} complete proof of the differentiable 1/4-pinching sphere theorem},
	FSeries = {Lecture Notes in Mathematics},
	Series = {Lect. Notes Math.},
	ISSN = {0075-8434},
	Volume = {2011},
	ISBN = {978-3-642-16285-5; 978-3-642-16286-2},
	Year = {2011},
	Publisher = {Berlin: Springer},
	Language = {English},
	DOI = {10.1007/978-3-642-16286-2},
	zbMATH = {5821328},
	Zbl = {1214.53002}
}

@Article{zepp:structure,
 Author = {Ak{\'e} Hau, L. and Flores Dorado, J. L. and S{\'a}nchez Caja, M.},
 Title = {Structure of globally hyperbolic spacetimes-with-timelike-boundary},
 FJournal = {Revista Matem{\'a}tica Iberoamericana},
 Journal = {Rev. Mat. Iberoam.},
 ISSN = {0213-2230},
 Volume = {37},
 Number = {1},
 Pages = {45--94},
 Year = {2021},
 Language = {English},
 DOI = {10.4171/rmi/1201},
 zbMATH = {7318519},
 Zbl = {1475.53074}
}

@Article{bernard-suhr,
 Author = {Bernard, P. and Suhr, S.},
 Title = {Lyapounov functions of closed cone fields: from {Conley} theory to time functions},
 FJournal = {Communications in Mathematical Physics},
 Journal = {Commun. Math. Phys.},
 ISSN = {0010-3616},
 Volume = {359},
 Number = {2},
 Pages = {467--498},
 Year = {2018},
 Language = {English},
 DOI = {10.1007/s00220-018-3127-7},
 zbMATH = {6869089},
 Zbl = {1391.83010}
}

@Article{GHnulldistance,
 Author = {Burtscher, A. and Garc{\'{\i}}a-Heveling, L.},
 Title = {Global hyperbolicity through the eyes of the null distance},
 FJournal = {Communications in Mathematical Physics},
 Journal = {Commun. Math. Phys.},
 ISSN = {0010-3616},
 Volume = {405},
 Number = {4},
 Pages = {35},
 Note = {Id/No 90},
 Year = {2024},
 Language = {English},
 DOI = {10.1007/s00220-024-04936-5},
 zbMATH = {7824918}
}

@misc{galloway2014notes,
  author       = {Galloway, G. J.},
  title        = {Notes on Lorentzian causality},
  howpublished = {ESI-EMS-IAMP Summer School on Mathematical Relativity},
  year         = {2014}
}

@Article{Kunzinger2022Nulldistance,
 Author = {Kunzinger, M. and Steinbauer, R.},
 Title = {Null distance and convergence of {Lorentzian} length spaces},
 FJournal = {Annales Henri Poincar{\'e}},
 Journal = {Ann. Henri Poincar{\'e}},
 ISSN = {1424-0637},
 Volume = {23},
 Number = {12},
 Pages = {4319--4342},
 Year = {2022},
 Language = {English},
 DOI = {10.1007/s00023-022-01198-6},
 zbMATH = {7622225},
 Zbl = {1509.53049}
}

@article {minguzzi:h-steep,
    AUTHOR = {Minguzzi, E.},
     TITLE = {The representation of spacetime through steep time functions},
   JOURNAL = {J. Phys. Conf. Ser.},
  FJOURNAL = {Journal of Physics. Conference Series},
    VOLUME = {968},
      YEAR = {2018},
     PAGES = {012009, 11},
      ISSN = {1742-6588,1742-6596},
   MRCLASS = {53C50 (53C60)},
  MRNUMBER = {3919952},
MRREVIEWER = {Erasmo\ Caponio},
       DOI = {10.1088/1742-6596/968/1/012009},
       URL = {https://doi.org/10.1088/1742-6596/968/1/012009},
}

@InCollection{MinguzziSanchez,
 Author = {Minguzzi, E. and S{\'a}nchez, M.},
 Title = {The causal hierarchy of space-times},
 BookTitle = {Recent developments in pseudo-Riemannian geometry},
 ISBN = {978-3-03719-051-7},
 Pages = {299--358},
 Year = {2008},
 Publisher = {Z{\"u}rich: European Mathematical Society},
 Language = {English},
 zbMATH = {5310408},
 Zbl = {1148.83002}
}

@misc{Minguzzi2024unbounded,
	author = {Bykov, A. and Minguzzi, E. and Suhr, S.},
	title = {Lorentzian metric spaces and {GH}-convergence: the unbounded case},
	year = {2024},
	howpublished = {Preprint, {arXiv}:2412.04311 [math.{MG}] (2024)},
	url = {https://arxiv.org/abs/2412.04311},
	arXiv = {arXiv:2412.04311}
}

@Article{minguzzi-suhr2023,
	Author = {Minguzzi, E. and Suhr, S.},
	Title = {Lorentzian metric spaces and their {Gromov}-{Hausdorff} convergence},
	FJournal = {Letters in Mathematical Physics},
	Journal = {Lett. Math. Phys.},
	ISSN = {0377-9017},
	Volume = {114},
	Number = {3},
	Pages = {63},
	Note = {Id/No 73},
	Year = {2024},
	Language = {English},
	DOI = {10.1007/s11005-024-01813-z},
	zbMATH = {7870841}
}

@misc{muller2018cheeger,
	title = {Cheeger-{G}romov compactness for manifolds with boundary},
        author = {M\"{u}ller,  O.},
        note = {Preprint, arXiv:1808.06458 [math.DG], 2018}
}

@Article{muller-sanchez,
 Author = {M{\"u}ller, O. and S{\'a}nchez, M.},
 Title = {Lorentzian manifolds isometrically embeddable in {{\(\mathbb{L}^{N}\)}}},
 FJournal = {Transactions of the American Mathematical Society},
 Journal = {Trans. Am. Math. Soc.},
 ISSN = {0002-9947},
 Volume = {363},
 Number = {10},
 Pages = {5367--5379},
 Year = {2011},
 Language = {English},
 DOI = {10.1090/S0002-9947-2011-05299-2},
 zbMATH = {5959528},
 Zbl = {1231.53059}
}

@Book{Oneill,
 Author = {O'Neill, B.},
 Title = {Semi-{Riemannian} geometry. {With} applications to relativity},
 FSeries = {Pure and Applied Mathematics (Academic Press)},
 Series = {Pure Appl. Math., Academic Press},
 ISSN = {0079-8169},
 Volume = {103},
 Year = {1983},
 Publisher = {Academic Press, New York, NY},
 Language = {English},
 zbMATH = {3842680},
 Zbl = {0531.53051}
}

@Article{Noldus1,
 Author = {Noldus, J.},
 Title = {A {Lorentzian} {Gromov}-{Hausdorff} notion of distance},
 FJournal = {Classical and Quantum Gravity},
 Journal = {Classical Quantum Gravity},
 ISSN = {0264-9381},
 Volume = {21},
 Number = {4},
 Pages = {839--850},
 Year = {2004},
 Language = {English},
 DOI = {10.1088/0264-9381/21/4/007},
 zbMATH = {2070675},
 Zbl = {1055.83034}
}

@Article{Noldus2,
 Author = {Noldus, J.},
 Title = {The limit space of a {Cauchy} sequence of globally hyperbolic spacetimes},
 FJournal = {Classical and Quantum Gravity},
 Journal = {Classical Quantum Gravity},
 ISSN = {0264-9381},
 Volume = {21},
 Number = {4},
 Pages = {851--874},
 Year = {2004},
 Language = {English},
 DOI = {10.1088/0264-9381/21/4/008},
 zbMATH = {2070676},
 Zbl = {1055.83035}
}

@Article{sanchez:slicings,
 Author = {S{\'a}nchez, M.},
 Title = {Globally hyperbolic spacetimes: slicings, boundaries and counterexamples},
 FJournal = {General Relativity and Gravitation},
 Journal = {Gen. Relativ. Gravitation},
 ISSN = {0001-7701},
 Volume = {54},
 Number = {10},
 Pages = {52},
 Note = {Id/No 124},
 Year = {2022},
 Language = {English},
 DOI = {10.1007/s10714-022-03002-6},
 zbMATH = {7619382},
 Zbl = {1515.83091}
}

@Misc{tesisdidier,
	Author = {Solis, D. A.},
	Title = {Global properties of asymptotically de {Sitter} and {Anti} de {Sitter} spacetimes},
	Year = {2018},
	HowPublished = {Preprint, {arXiv}:1803.01171 [gr-qc] (2018)},
	URL = {https://arxiv.org/abs/1803.01171},
	arXiv = {arXiv:1803.01171}
}

@Article{sormanivega,
 Author = {Sormani, C. and Vega, C.},
 Title = {Null distance on a spacetime},
 FJournal = {Classical and Quantum Gravity},
 Journal = {Classical Quantum Gravity},
 ISSN = {0264-9381},
 Volume = {33},
 Number = {8},
 Pages = {29},
 Note = {Id/No 085001},
 Year = {2016},
 Language = {English},
 DOI = {10.1088/0264-9381/33/7/085001},
 zbMATH = {6585555},
 Zbl = {1337.83015}
}

@misc{sormani2018oberwolfach,
	doi = {10.48550/ARXIV.1805.08886},
	url = {https://arxiv.org/abs/1805.08886},
	author = {Sormani,  C.},
	title = {Oberwolfach Report: Spacetime Intrinsic Flat Convergence},
	publisher = {arXiv},
	year = {2018},
	copyright = {arXiv.org perpetual,  non-exclusive license}
}

@article{splitting,
  doi = {10.1007/s00220-003-0982-6},
  url = {https://doi.org/10.1007/s00220-003-0982-6},
  year = {2003},
    publisher = {Springer Science and Business Media {LLC}},
  volume = {243},
  number = {3},
  pages = {461--470},
  author = {Bernal, A. N. and S{\'{a}}nchez, M.},
  title = {On Smooth Cauchy Hypersurfaces and Geroch's Splitting Theorem},
  journal = {Commun. Math. Phys.}
}

@book {burago:metric,
    AUTHOR = {Burago, D. and Burago, Y. and Ivanov, S.},
     TITLE = {A course in metric geometry},
    SERIES = {Graduate Studies in Mathematics},
    VOLUME = {33},
 PUBLISHER = {American Mathematical Society, Providence, RI},
      YEAR = {2001},
     PAGES = {xiv+415},
      ISBN = {0-8218-2129-6},
   MRCLASS = {53C23},
  MRNUMBER = {1835418},
MRREVIEWER = {Mario Bonk},
       DOI = {10.1090/gsm/033},
       URL = {https://doi.org/10.1090/gsm/033},
}

\end{document}